\documentclass{article}
\usepackage[utf8]{inputenc}
\RequirePackage{amsmath}
\usepackage{amsmath,amsthm,amssymb}
\usepackage{chngcntr}
\usepackage{cite}
\usepackage{dsfont}
\usepackage[nobysame]{amsrefs}
\usepackage[unicode,colorlinks=true,citecolor=green!40!black,linkcolor=red!20!black,urlcolor=blue!40!black,filecolor=cyan!30!black]{hyperref}
\usepackage{enumerate}
\usepackage{caption}
\usepackage{float}
\usepackage{tikz}
\usetikzlibrary{calc}
\usetikzlibrary{decorations.shapes, shapes.geometric}
\usepackage{centernot}
\usepackage[noprefix]{nomencl}
\usepackage{colonequals}

\usepackage{titling}
\thanksmarkseries{arabic}

\theoremstyle{definition}
\newtheorem{theorem}{Theorem}
\newtheorem{definition}[theorem]{Definition}
\newtheorem{proposition}[theorem]{Proposition}
\newtheorem{corollary}[theorem]{Corollary}
\newtheorem{lemma}[theorem]{Lemma}

\newtheorem{claim}[theorem]{Claim}
\newtheorem{assumption}[theorem]{Assumption}
\newtheorem{conjecture}[theorem]{Conjecture}
\newtheorem{question}[theorem]{Question}

\counterwithin*{theorem}{section}
\counterwithin*{theorem}{subsection}

\renewcommand{\thetheorem}{%
 \ifnum\value{subsection}=0
  \thesection
 \else
  \thesubsection
 \fi
  .\arabic{theorem}%
}

\DeclareFontFamily{U}{mathx}{\hyphenchar\font45}
\DeclareFontShape{U}{mathx}{m}{n}{<5> <6> <7> <8> <9> <10> <10.95> <12> <14.4> <17.28> <20.74> <24.88> mathx10}{}
\DeclareSymbolFont{mathx}{U}{mathx}{m}{n}
\DeclareFontSubstitution{U}{mathx}{m}{n}
\DeclareMathAccent{\widecheck}{0}{mathx}{"71}

\makeatletter
 \def\@makefnmark{\leavevmode \raise.9ex\hbox{\fontsize\sf@size\z@\normalfont\tiny\@thefnmark}}
\makeatother

\makeatletter
 \newcommand{\xleftrightarrow}[2][]{\ext@arrow 3359\leftrightarrowfill@{#1}{#2}}
\makeatother

\tikzset{paint/.style={draw=#1!50!black, fill=#1!50}, decorate with/.style = {decorate, decoration={shape backgrounds, shape=#1, shape size = 3pt, shape sep = 5pt}}}

\allowdisplaybreaks

\makenomenclature

\title{Color-avoiding percolation in edge-colored Erd\H{o}s--R\'{e}nyi graphs}

\author{Bal\'{a}zs R\'{a}th\thanks{Afr\'{e}d R\'{e}nyi Institute of Mathematics.} \textsuperscript{,}\thanks{Department of Stochastics, Institute of Mathematics, Budapest University of Technology and Economics.} \textsuperscript{,}\thanks{ELKH-BME Stochastics Research Group.} \and Kitti Varga\protect\footnotemark[1] \textsuperscript{,}\thanks{Department of Computer Science and Information Theory, Faculty of Electrical Engineering and Informatics, Budapest University of Technology and Economics.} \textsuperscript{,}\thanks{ELKH-ELTE Egerv\'{a}ry Research Group.} \and Panna T\'{i}mea Fekete\protect\footnotemark[1] \textsuperscript{,}\thanks{Department of Applied Analysis and Computational Mathematics, E\"{o}tv\"{o}s Lor\'{a}nd University.} \and Roland Molontay\protect\footnotemark[2] \textsuperscript{,}\protect\footnotemark[3]}

\begin{document}

\maketitle

\begin{abstract}

 We study a variant of the color-avoiding percolation model introduced by Krause et al., namely we investigate the color-avoiding bond percolation setup on (not necessarily properly) edge-colored Erd\H{o}s--R\'{e}nyi random graphs. We say that two vertices are color-avoiding connected in an edge-colored graph if after the removal of the edges of any color, they are in the same component in the remaining graph. The color-avoiding connected components of an edge-colored graph are maximal sets of vertices such that any two of them are color-avoiding connected.
 
 We consider the  fraction of vertices contained in color-avoiding connected components of a given size as well as the fraction of vertices contained in the giant color-avoiding connected component. Under some mild assumptions on the color-densities, we prove that these quantities converge and the limits can be expressed in terms of probabilities associated to edge-colored branching process trees. We provide explicit formulas for the limit of the normalized size of the giant color-avoiding connected component, and in the two-colored case we also provide explicit formulas for the limit of the fraction of vertices contained in color-avoiding connected components of a given size.

 \medskip
 
 
\end{abstract}

\section{Introduction} \label{section:intro}

The mathematical analysis of the robustness of random graphs and networks has been in the focus of research interest in the last two decades~\cites{albert2000error, callaway2000network, li2021percolation, shao2015percolation}. The robustness of a network refers to the ability to maintain the overall connectivity against random error or targeted attack, i.e., how its structure varies when a fraction of its vertices or edges are removed. This question can be investigated using the tools of an established field of statistical physics called percolation theory.

Most of the works have focused on classical percolation on networks, meaning that each vertex (site) or edge (bond) is kept (occupied) with probability $p$ and removed (vacant) with probability $1-p$, see, e.g.,~\cite{li2021percolation}. This way of keeping or removing edges models how the network performs under random failures. One fundamental question is whether there is a giant connected component of macroscopic size in the remaining graph. Percolation theory studies the emergence of such a component when we increase the probability $p$. Usually there is a well-defined critical point $p_c$ above which there exists a unique giant component. This critical point is also called the percolation threshold and this phenomenon is referred to as percolation phase transition.

Besides classical network percolation, another well-studied model is when the vertices or edges are removed with preference, for example the nodes can be targeted preferentially according to a structural property such as the degree or betweenness centrality. Scale-free networks were found to be robust under random failures, and extremely vulnerable for targeted attacks~\cite{albert2000error}.

Several other network percolation frameworks have been proposed in the literature throughout the years including $k$-core percolation~\cite{dorogovtsev2006k}, bootstrap percolation~\cite{baxter2010bootstrap}, $l$-hop percolation~\cite{shang2011hop} and limited path percolation~\cite{lopez2007limited}. For an extensive review on network percolation, we refer to Li et al.~\cite{li2021percolation}.

In this paper, we focus on another recent network percolation framework developed by Krause et al., called color-avoiding percolation~\cites{kadovic2018bond, krause2016hidden, krause2017color}. In this framework, each vertex or edge is assigned a color from a color set; a color may represent a shared eavesdropper, a controlling entity, or a correlated failure. Two vertices are color-avoiding connected if they can be reached from each other on paths that avoid any color from the color set. Therefore, while in the traditional percolation framework a single path provides connectivity, here connectivity refers to the ability to avoid all vulnerable sets of vertices or edges. The main rationale behind this percolation model is that in real-world networks some nodes or links may mutually share a vulnerability and it is interesting to see whether the network is robustly connected in the sense that its overall connectivity is maintained in every possible attack scenario.

The color-avoiding percolation framework differs substantially from $k$-core percolation where any $k$ paths are sufficient between two nodes in the percolating cluster~\cites{dorogovtsev2006k, yuan2016k} and it also differs from $k$-connectivity, where $k$ pairwise disjoint paths are required~\cite{penrose1999k}. Another relevant framework is percolation on multiplex networks~\cites{hackett2016bond, mahabadi2021network, son2012percolation} where the layers can be considered as colors. However, this approach also differs in the definition of connectivity as it is discussed in detail by Krause et al. in~\cite{krause2016hidden}. Kryven also investigated percolation on colored networks, however, in that setup, colors do not indicate shared vulnerabilities, but the removal probability depends on the color~\cite{kryven2019bond}.

In~\cite{krause2016hidden}, Krause et al.\ investigated networks with colored vertices. The authors not only presented a heuristic analytical framework for the configuration model together with a numerical algorithm to determine the color-avoiding connected pair of nodes, but they also demonstrated the applicability of the color-avoiding percolation in cybersecurity. The authors extended the theory of color-avoiding site percolation by studying the phase transition for Erd\H{o}s--R\'{e}nyi random graphs and by introducing novel node functions to generalize the concept~\cite{krause2017color}. Shekhtman et al.\ further developed this framework to study secure message-passing in networks with a given community structure~\cite{shekhtman2018critical}. Giusfredi and Bagnoli showed that color-avoiding site percolation can be mapped into a self-organized critical problem~\cite{giusfredi2020color}. 

In~\cite{kadovic2018bond}, Kadovi\'{c} et al.\ studied (not necessarily properly) edge-colored networks. The authors presented both analytical and numerical results about the size of the giant color-avoiding connected component for Erd\H{o}s--R\'{e}nyi and for scale-free graphs.

Molontay and Varga investigated the computational complexity of finding the color-avoiding connected components both in the vertex- and edge-colored setup~\cite{molontay2019complexity}. The authors found that the complexity of color-avoiding site percolation highly depends on the exact formulation: a strong version of the problem is NP-hard, while a weaker notion makes it possible to find the components in polynomial time. However, in the bond percolation setup, the color-avoiding connected components can be found in polynomial time.

In this paper, we study the color-avoiding bond percolation setup on Erd\H{o}s--R\'{e}nyi random graphs with mathematical rigor. We explicitly calculate the empirical density of the giant color-avoiding connected component (if it exists) with an arbitrary number of randomly distributed colors under some minor assumptions on the color frequencies. Moreover, we also give an explicit formula for the empirical density of the sizes of  color-avoiding connected components in the case of two-colored Erd\H{o}s--R\'{e}nyi random graphs.

The paper is organized as follows. In Section~\ref{section:main_results}, we collect the basic definitions needed for this paper and state our main results. In Section~\ref{section:notation}, we introduce some further notation. In Section~\ref{section:connection_between_ECER_and_ECBP}, we show how the giant color-avoiding connected components of edge-colored Erd\H{o}s--R\'{e}nyi random (multi)graphs relate to the survival of certain edge-colored branching processes. In Section~\ref{section:conv_of_f_ell_G_n}, we prove one of our main results about the convergence of empirical color-avoiding connected component size densities. In Section~\ref{section:explicit_formulas}, we give formulas for the asymptotic size of the giant color-avoiding connected component (if it exists) and we also study its near-critical behaviour. In Section~\ref{section:two_colors}, we give a formula for the distribution of the sizes of the small components when there are only two colors. Finally, in Section \ref{section_open_questions}, we propose some open questions.

Let us note that a few months after we posted the first version of this paper on arXiv, a new paper~\cite{lichevschapira} on color-avoiding percolation on edge-colored Erd\H{o}s--R\'{e}nyi graphs also appeared on arXiv. The results of~\cite{lichevschapira} include simplification and generalization of the results that we prove in Section~\ref{section:conv_of_f_ell_G_n}, the answer to some of the open questions that we pose in Section~\ref{section_open_questions} as well as some finer results on the size of the largest color-avoiding component, see Section~\ref{section:main_results} for more details.

\section{Statements of main results} \label{section:main_results}

We denote the sets of real, positive real, non-negative integer, and positive integer numbers by $\mathbb{R}$, $\mathbb{R}_+$, $\mathbb{N}$, and $\mathbb{N}_+$, respectively. Let us use the notation $[n] = \{1, 2, \ldots, n \}$ for any $n \in \mathbb{N}_+$. If $A$ is an event, we denote by $\mathds{1}[A]$ the indicator variable of the event: $\mathds{1}[A] = 1$ if $A$ occurs, $\mathds{1}[A] = 0$ if the complement of $A$, i.e., $A^c$ occurs. Throughout the article, by edge-colorings we always mean not necessarily proper ones and by paths we always mean simple paths (and not walks).

\begin{definition}[Edge-colored (multi)graph] \label{def:edgecolored_graph}
 We say that $G=(V,\underline{E})$ is an \emph{edge-colored (multi)graph} with $k \in \mathbb{N}_+$ colors if $\underline{E} = (E_1, \ldots, E_k)$ and $E_i \subseteq \binom{V}{2}$ for every $i \in [k]$. We say that $V$ is the vertex set of $G$ and $E_i$ is the set of edges of color $i$ for every $i \in [k]$.
\end{definition}

Note that the edge sets $E_1, \ldots, E_k$ are not necessarily disjoint and the edges of $E_i$ are not necessarily independent in the graph theoretical sense for any $i \in [k]$ (i.e., the graph $G$ is not necessarily properly edge-colored).

\begin{definition}[Uncolored color-subgraphs] \label{def:colorsubgraphs}
 Given an edge-colored graph $G = (V, \underline{E})$ with $k \in \mathbb{N}_+$ colors, where $\underline{E} = (E_1, \ldots, E_k)$, let $G^I$ denote the simple uncolored graph $\left( V, \bigcup_{i \in I} E_i \right)$.
 
 In addition, let us use the notation $G^{\text{uc}} \colonequals G^{[k]}$ and $G^{\setminus i} \colonequals G^{[k] \setminus \{ i \}}$ for any $i \in [k]$.
\end{definition}

For an example, see Figure~\ref{fig:colored_graph}.

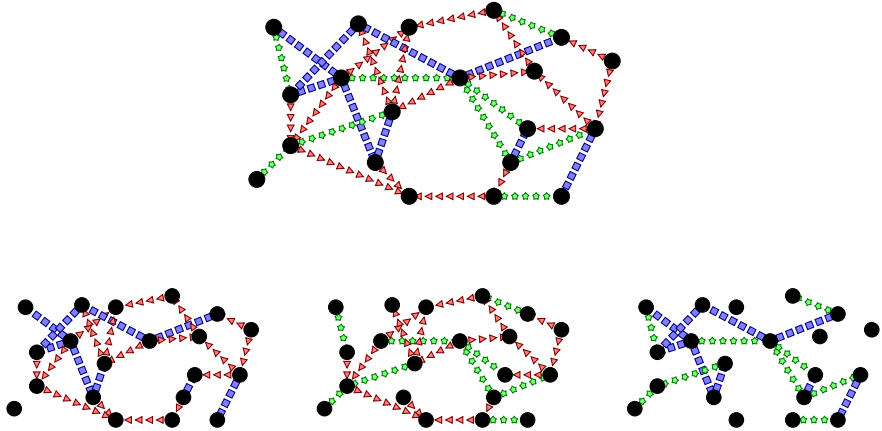
\begin{figure}[H]
 \centering
 \begin{tikzpicture}
  \tikzstyle{vertex}=[draw,circle,fill,minimum size=6,inner sep=0]
  \tikzstyle{vertex2}=[draw,circle,fill,minimum size=5.5,inner sep=0]
  \tikzset{paint/.style={draw=#1!50!black, fill=#1!50}, decorate with/.style = {decorate, decoration={shape backgrounds, shape=#1, shape size = 2.5pt, shape sep = 4pt}}}
  
  \begin{scope}[shift={(0,0)}, scale=0.9]
  \coordinate (v1) at (-1.75,0) {};
  \coordinate (v2) at (0.75,-1.25) {};
  \coordinate (v3) at (2.25,0.25) {};
  \coordinate (v4) at (-2.75,0.75) {};
  \coordinate (v5) at (2,-0.75) {};
  \coordinate (v6) at (0,0) {};
  \coordinate (v7) at (0.5,-1.75) {};
  \coordinate (v8) at (1.5,0.6) {};
  \coordinate (v9) at (-1.25,-1.25) {};
  \coordinate (v10) at (1.5,-1.75) {};
  \coordinate (v11) at (1.1,0.1) {};
  \coordinate (v12) at (-1,-0.5) {};
  \coordinate (v13) at (1,-0.75) {};
  \coordinate (v14) at (0.5,1) {};
  \coordinate (v15) at (-1.5,0.8) {};
  \coordinate (v16) at (-3,-1.5) {};
  \coordinate (v17) at (-0.75,0.75) {};
  \coordinate (v18) at (-2.5,-0.25) {};
  \coordinate (v19) at (-2.5,-1) {};
  \coordinate (v20) at (-0.75,-1.75) {};
  
  \draw[decorate with = rectangle, paint=blue] (v1) -- (v4);
  \draw[decorate with = star, paint=green] (v1) -- (v6);
  \draw[decorate with = rectangle, paint=blue] (v1) -- (v9);
  \draw[decorate with = isosceles triangle, paint=red] (v1) -- (v17);
  \draw[decorate with = rectangle, paint=blue] (v1) -- (v18);
  \draw[decorate with = isosceles triangle, paint=red] (v1) -- (v19);
  
  \draw[decorate with = star, paint=green] (v2) -- (v5);
  \draw[decorate with = star, paint=green] (v2) -- (v6);
  \draw[decorate with = isosceles triangle, paint=red] (v2) -- (v7);
  \draw[decorate with = rectangle, paint=blue] (v2) -- (v13);
  
  \draw[decorate with = isosceles triangle, paint=red] (v3) -- (v5);
  \draw[decorate with = isosceles triangle, paint=red] (v3) -- (v8);
  
  \draw[decorate with = star, paint=green] (v4) -- (v18);
  
  \draw[decorate with = rectangle, paint=blue] (v5) -- (v10);
  \draw[decorate with = isosceles triangle, paint=red] (v5) -- (v11);
  \draw[decorate with = isosceles triangle, paint=red] (v5) -- (v13);
  
  \draw[decorate with = rectangle, paint=blue] (v6) -- (v8);
  \draw[decorate with = isosceles triangle, paint=red] (v6) -- (v11);
  \draw[decorate with = isosceles triangle, paint=red] (v6) -- (v12);
  \draw[decorate with = star, paint=green] (v6) -- (v13);
  \draw[decorate with = rectangle, paint=blue] (v6) -- (v15);
  
  \draw[decorate with = star, paint=green] (v7) -- (v10);
  \draw[decorate with = isosceles triangle, paint=red] (v7) -- (v20);
  
  \draw[decorate with = star, paint=green] (v8) -- (v14);
  
  \draw[decorate with = rectangle, paint=blue] (v9) -- (v12);
  \draw[decorate with = isosceles triangle, paint=red] (v9) -- (v20);
  
  \draw[decorate with = isosceles triangle, paint=red] (v11) -- (v14);
  
  \draw[decorate with = isosceles triangle, paint=red] (v12) -- (v15);
  \draw[decorate with = isosceles triangle, paint=red] (v12) -- (v17);
  \draw[decorate with = star, paint=green] (v12) -- (v19);
  
  \draw[decorate with = isosceles triangle, paint=red] (v14) -- (v17);
  
  \draw[decorate with = rectangle, paint=blue] (v15) -- (v18);
  
  \draw[decorate with = star, paint=green] (v16) -- (v19);
  
  \draw[decorate with = isosceles triangle, paint=red] (v18) -- (v19);
  
  \draw[decorate with = isosceles triangle, paint=red] (v19) -- (v20);
  
  \node[vertex] at (v1) {}; 
  \node[vertex] at (v2) {}; 
  \node[vertex] at (v3) {}; 
  \node[vertex] at (v4) {}; 
  \node[vertex] at (v5) {}; 
  \node[vertex] at (v6) {}; 
  \node[vertex] at (v7) {}; 
  \node[vertex] at (v8) {}; 
  \node[vertex] at (v9) {}; 
  \node[vertex] at (v10) {}; 
  \node[vertex] at (v11) {}; 
  \node[vertex] at (v12) {}; 
  \node[vertex] at (v13) {}; 
  \node[vertex] at (v14) {}; 
  \node[vertex] at (v15) {}; 
  \node[vertex] at (v16) {}; 
  \node[vertex] at (v17) {}; 
  \node[vertex] at (v18) {}; 
  \node[vertex] at (v19) {}; 
  \node[vertex] at (v20) {}; 
  \end{scope}
  
  \begin{scope}[shift={(-4.125,-3.5)}, scale=0.6]
  \coordinate (v1) at (-1.75,0) {};
  \coordinate (v2) at (0.75,-1.25) {};
  \coordinate (v3) at (2.25,0.25) {};
  \coordinate (v4) at (-2.75,0.75) {};
  \coordinate (v5) at (2,-0.75) {};
  \coordinate (v6) at (0,0) {};
  \coordinate (v7) at (0.5,-1.75) {};
  \coordinate (v8) at (1.5,0.6) {};
  \coordinate (v9) at (-1.25,-1.25) {};
  \coordinate (v10) at (1.5,-1.75) {};
  \coordinate (v11) at (1.1,0.1) {};
  \coordinate (v12) at (-1,-0.5) {};
  \coordinate (v13) at (1,-0.75) {};
  \coordinate (v14) at (0.5,1) {};
  \coordinate (v15) at (-1.5,0.8) {};
  \coordinate (v16) at (-3,-1.5) {};
  \coordinate (v17) at (-0.75,0.75) {};
  \coordinate (v18) at (-2.5,-0.25) {};
  \coordinate (v19) at (-2.5,-1) {};
  \coordinate (v20) at (-0.75,-1.75) {};
  
  \draw[decorate with = rectangle, paint=blue] (v1) -- (v4);
  \draw[decorate with = rectangle, paint=blue] (v1) -- (v9);
  \draw[decorate with = isosceles triangle, paint=red] (v1) -- (v17);
  \draw[decorate with = rectangle, paint=blue] (v1) -- (v18);
  \draw[decorate with = isosceles triangle, paint=red] (v1) -- (v19);

  \draw[decorate with = isosceles triangle, paint=red] (v2) -- (v7);
  \draw[decorate with = rectangle, paint=blue] (v2) -- (v13);
  
  \draw[decorate with = isosceles triangle, paint=red] (v3) -- (v5);
  \draw[decorate with = isosceles triangle, paint=red] (v3) -- (v8);
  
  \draw[decorate with = rectangle, paint=blue] (v5) -- (v10);
  \draw[decorate with = isosceles triangle, paint=red] (v5) -- (v11);
  \draw[decorate with = isosceles triangle, paint=red] (v5) -- (v13);
  
  \draw[decorate with = rectangle, paint=blue] (v6) -- (v8);
  \draw[decorate with = isosceles triangle, paint=red] (v6) -- (v11);
  \draw[decorate with = isosceles triangle, paint=red] (v6) -- (v12);
  \draw[decorate with = rectangle, paint=blue] (v6) -- (v15);
  
  \draw[decorate with = isosceles triangle, paint=red] (v7) -- (v20);
  
  \draw[decorate with = rectangle, paint=blue] (v9) -- (v12);
  \draw[decorate with = isosceles triangle, paint=red] (v9) -- (v20);
  
  \draw[decorate with = isosceles triangle, paint=red] (v11) -- (v14);
  
  \draw[decorate with = isosceles triangle, paint=red] (v12) -- (v15);
  \draw[decorate with = isosceles triangle, paint=red] (v12) -- (v17);
  
  \draw[decorate with = isosceles triangle, paint=red] (v14) -- (v17);
  
  \draw[decorate with = rectangle, paint=blue] (v15) -- (v18);

  \draw[decorate with = isosceles triangle, paint=red] (v18) -- (v19);
  
  \draw[decorate with = isosceles triangle, paint=red] (v19) -- (v20);
  
  \node[vertex2] at (v1) {}; 
  \node[vertex2] at (v2) {}; 
  \node[vertex2] at (v3) {}; 
  \node[vertex2] at (v4) {}; 
  \node[vertex2] at (v5) {}; 
  \node[vertex2] at (v6) {}; 
  \node[vertex2] at (v7) {}; 
  \node[vertex2] at (v8) {}; 
  \node[vertex2] at (v9) {}; 
  \node[vertex2] at (v10) {}; 
  \node[vertex2] at (v11) {}; 
  \node[vertex2] at (v12) {}; 
  \node[vertex2] at (v13) {}; 
  \node[vertex2] at (v14) {}; 
  \node[vertex2] at (v15) {}; 
  \node[vertex2] at (v16) {}; 
  \node[vertex2] at (v17) {}; 
  \node[vertex2] at (v18) {}; 
  \node[vertex2] at (v19) {}; 
  \node[vertex2] at (v20) {}; 
  \end{scope}
  
  \begin{scope}[shift={(0,-3.5)}, scale=0.6]
  \coordinate (v1) at (-1.75,0) {};
  \coordinate (v2) at (0.75,-1.25) {};
  \coordinate (v3) at (2.25,0.25) {};
  \coordinate (v4) at (-2.75,0.75) {};
  \coordinate (v5) at (2,-0.75) {};
  \coordinate (v6) at (0,0) {};
  \coordinate (v7) at (0.5,-1.75) {};
  \coordinate (v8) at (1.5,0.6) {};
  \coordinate (v9) at (-1.25,-1.25) {};
  \coordinate (v10) at (1.5,-1.75) {};
  \coordinate (v11) at (1.1,0.1) {};
  \coordinate (v12) at (-1,-0.5) {};
  \coordinate (v13) at (1,-0.75) {};
  \coordinate (v14) at (0.5,1) {};
  \coordinate (v15) at (-1.5,0.8) {};
  \coordinate (v16) at (-3,-1.5) {};
  \coordinate (v17) at (-0.75,0.75) {};
  \coordinate (v18) at (-2.5,-0.25) {};
  \coordinate (v19) at (-2.5,-1) {};
  \coordinate (v20) at (-0.75,-1.75) {};
  
  \draw[decorate with = star, paint=green] (v1) -- (v6);
  \draw[decorate with = isosceles triangle, paint=red] (v1) -- (v17);
  \draw[decorate with = isosceles triangle, paint=red] (v1) -- (v19);
  
  \draw[decorate with = star, paint=green] (v2) -- (v5);
  \draw[decorate with = star, paint=green] (v2) -- (v6);
  \draw[decorate with = isosceles triangle, paint=red] (v2) -- (v7);
  
  \draw[decorate with = isosceles triangle, paint=red] (v3) -- (v5);
  \draw[decorate with = isosceles triangle, paint=red] (v3) -- (v8);
  
  \draw[decorate with = star, paint=green] (v4) -- (v18);
  
  \draw[decorate with = isosceles triangle, paint=red] (v5) -- (v11);
  \draw[decorate with = isosceles triangle, paint=red] (v5) -- (v13);
  
  \draw[decorate with = isosceles triangle, paint=red] (v6) -- (v11);
  \draw[decorate with = isosceles triangle, paint=red] (v6) -- (v12);
  \draw[decorate with = star, paint=green] (v6) -- (v13);
  
  \draw[decorate with = star, paint=green] (v7) -- (v10);
  \draw[decorate with = isosceles triangle, paint=red] (v7) -- (v20);
  
  \draw[decorate with = star, paint=green] (v8) -- (v14);
  
  \draw[decorate with = isosceles triangle, paint=red] (v9) -- (v20);
  
  \draw[decorate with = isosceles triangle, paint=red] (v11) -- (v14);
  
  \draw[decorate with = isosceles triangle, paint=red] (v12) -- (v15);
  \draw[decorate with = isosceles triangle, paint=red] (v12) -- (v17);
  \draw[decorate with = star, paint=green] (v12) -- (v19);
  
  \draw[decorate with = isosceles triangle, paint=red] (v14) -- (v17);
  
  \draw[decorate with = star, paint=green] (v16) -- (v19);
  
  \draw[decorate with = isosceles triangle, paint=red] (v18) -- (v19);
  
  \draw[decorate with = isosceles triangle, paint=red] (v19) -- (v20);
  
  \node[vertex2] at (v1) {}; 
  \node[vertex2] at (v2) {}; 
  \node[vertex2] at (v3) {}; 
  \node[vertex2] at (v4) {}; 
  \node[vertex2] at (v5) {}; 
  \node[vertex2] at (v6) {}; 
  \node[vertex2] at (v7) {}; 
  \node[vertex2] at (v8) {}; 
  \node[vertex2] at (v9) {}; 
  \node[vertex2] at (v10) {}; 
  \node[vertex2] at (v11) {}; 
  \node[vertex2] at (v12) {}; 
  \node[vertex2] at (v13) {}; 
  \node[vertex2] at (v14) {}; 
  \node[vertex2] at (v15) {}; 
  \node[vertex2] at (v16) {}; 
  \node[vertex2] at (v17) {}; 
  \node[vertex2] at (v18) {}; 
  \node[vertex2] at (v19) {}; 
  \node[vertex2] at (v20) {}; 
  \end{scope}
  
  \begin{scope}[shift={(4.125,-3.5)}, scale=0.6]
  \coordinate (v1) at (-1.75,0) {};
  \coordinate (v2) at (0.75,-1.25) {};
  \coordinate (v3) at (2.25,0.25) {};
  \coordinate (v4) at (-2.75,0.75) {};
  \coordinate (v5) at (2,-0.75) {};
  \coordinate (v6) at (0,0) {};
  \coordinate (v7) at (0.5,-1.75) {};
  \coordinate (v8) at (1.5,0.6) {};
  \coordinate (v9) at (-1.25,-1.25) {};
  \coordinate (v10) at (1.5,-1.75) {};
  \coordinate (v11) at (1.1,0.1) {};
  \coordinate (v12) at (-1,-0.5) {};
  \coordinate (v13) at (1,-0.75) {};
  \coordinate (v14) at (0.5,1) {};
  \coordinate (v15) at (-1.5,0.8) {};
  \coordinate (v16) at (-3,-1.5) {};
  \coordinate (v17) at (-0.75,0.75) {};
  \coordinate (v18) at (-2.5,-0.25) {};
  \coordinate (v19) at (-2.5,-1) {};
  \coordinate (v20) at (-0.75,-1.75) {};
  
  \draw[decorate with = rectangle, paint=blue] (v1) -- (v4);
  \draw[decorate with = star, paint=green] (v1) -- (v6);
  \draw[decorate with = rectangle, paint=blue] (v1) -- (v9);
  \draw[decorate with = rectangle, paint=blue] (v1) -- (v18);
  
  \draw[decorate with = star, paint=green] (v2) -- (v5);
  \draw[decorate with = star, paint=green] (v2) -- (v6);
  \draw[decorate with = rectangle, paint=blue] (v2) -- (v13);
  
  \draw[decorate with = star, paint=green] (v4) -- (v18);
  
  \draw[decorate with = rectangle, paint=blue] (v5) -- (v10);
  
  \draw[decorate with = rectangle, paint=blue] (v6) -- (v8);
  \draw[decorate with = star, paint=green] (v6) -- (v13);
  \draw[decorate with = rectangle, paint=blue] (v6) -- (v15);
  
  \draw[decorate with = star, paint=green] (v7) -- (v10);

  \draw[decorate with = star, paint=green] (v8) -- (v14);
  
  \draw[decorate with = rectangle, paint=blue] (v9) -- (v12);

  \draw[decorate with = star, paint=green] (v12) -- (v19);
  
  \draw[decorate with = rectangle, paint=blue] (v15) -- (v18);
  
  \draw[decorate with = star, paint=green] (v16) -- (v19);
  
  \node[vertex2] at (v1) {}; 
  \node[vertex2] at (v2) {}; 
  \node[vertex2] at (v3) {}; 
  \node[vertex2] at (v4) {}; 
  \node[vertex2] at (v5) {}; 
  \node[vertex2] at (v6) {}; 
  \node[vertex2] at (v7) {}; 
  \node[vertex2] at (v8) {}; 
  \node[vertex2] at (v9) {}; 
  \node[vertex2] at (v10) {}; 
  \node[vertex2] at (v11) {}; 
  \node[vertex2] at (v12) {}; 
  \node[vertex2] at (v13) {}; 
  \node[vertex2] at (v14) {}; 
  \node[vertex2] at (v15) {}; 
  \node[vertex2] at (v16) {}; 
  \node[vertex2] at (v17) {}; 
  \node[vertex2] at (v18) {}; 
  \node[vertex2] at (v19) {}; 
  \node[vertex2] at (v20) {}; 
  \end{scope}
 \end{tikzpicture}
 \caption[caption]{In the picture, the colors red, blue, and green are denoted by (red) triangles, (blue) rectangles, and (green) stars, respectively. Above, an edge-colored graph $G$ is shown, and in the second row, we can see its color-subgraphs $G^{\text{red,\,blue}}$, $G^{\text{red,\,green}}$, $G^{\text{blue,\,green}}$, respectively.}
 \label{fig:colored_graph}
\end{figure}

\begin{definition}[Color-avoiding connectivity] \label{def:coloravoiding_connectivity}
 Let $G = (V, \underline{E})$ be an edge-colored graph with $k \in \mathbb{N}_+$ colors. For any color $i \in [k]$, we say that the vertices $v,w \in V(G)$ are \emph{$i$-avoiding connected}, briefly denoted by
 \[ v \stackrel{G^{\setminus i}}{\longleftrightarrow} w \text{,} \]
 if $v$ and $w$ are in the same component in the graph $G^{\setminus i}$, i.e., there exists a path between $v$ and $w$ in $G$ which does not contain any edges of color $i$; such a path is called an \emph{$i$-avoiding $v$-$w$ path}.
 
 We say that the vertices $v,w \in V(G)$ are \emph{color-avoiding connected} if they are $i$-avoiding connected for all $i \in [k]$.
\end{definition}

Note that the relation of color-avoiding connectivity is an equivalence relation and thus it defines a partition of the vertex set. 

\begin{definition}[Color-avoiding\hspace{1pt} connected\hspace{1pt} components]\hspace{1pt} \label{def:coloravoiding_component}
 The\hspace{1pt} equivalence classes of the relation of color-avoiding connectivity are called \emph{color-avoiding connected components}.
 
 The set of vertices of the color-avoiding connected component of a vertex $v$ in an edge-colored graph $G$ is denoted by $\widetilde{\mathcal{C}}(v) \colonequals \widetilde{\mathcal{C}}(v,G)$.
\end{definition}

Note that even if $v$ and $w$ are in the same color-avoiding connected component $\widetilde{\mathcal{C}}$, it might still happen that there exists a color $i$ for which every $i$-avoiding $v$-$w$ path has vertices outside $\widetilde{\mathcal{C}}$. For an example, see Figure~\ref{fig:coloravoiding_components}.

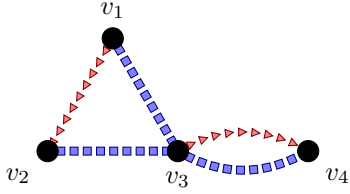
\begin{figure}[H]
 \centering
 \begin{tikzpicture}
  \tikzstyle{vertex}=[draw,circle,fill,minimum size=8,inner sep=0]
  
  \coordinate (v1) at (90:1);
  \coordinate (v2) at (210:1);
  \coordinate (v3) at (330:1);
  \coordinate (v4) at ($(v3)+({sqrt(3)},0)$);
  
  \draw[decorate with = rectangle, paint=blue] (v1) -- (v3);
  \draw[decorate with = rectangle, paint=blue] (v2) -- (v3);
  \draw[decorate with = isosceles triangle, paint=red] (v1) -- (v2);
  \draw[decorate with = isosceles triangle, paint=red] (v3) to [bend left=30] (v4);
  \draw[decorate with = rectangle, paint=blue] (v3) to [bend right=30] (v4);
  
  \node[vertex] at (v1) [label=above: $v_1$] {};
  \node[vertex] at (v2) [label=below left: $v_2$] {};
  \node[vertex] at (v3) [label=below: $v_3$] {};
  \node[vertex] at (v4) [label=below right: $v_4$] {};
 \end{tikzpicture}
 \caption{In the picture, the colors red and blue are denoted by (red) triangles and (blue) rectangles, respectively. The color-avoiding connected components of this graph are $\{ v_1, v_2 \}$ and $\{ v_3, v_4 \}$. Note that the red-avoiding $v_1$-$v_2$ path leaves the component $\{ v_1, v_2 \}$.}
 \label{fig:coloravoiding_components}
\end{figure}

\begin{definition}[Edge-colored Erd\H{o}s--R\'{e}nyi multigraph (ECER graph)] \label{def:ECER}
 Let $k,n \in \mathbb{N}_+$ and $\underline{\lambda} = (\lambda_1,\ldots,\lambda_k) \in \mathbb{R}_+^k$. We say that the edge-colored random multigraph $G = (V, \underline{E})$ is an \emph{edge-colored Erd\H{o}s--R\'{e}nyi multigraph} (or \emph{ECER graph} for short) with parameter $n$ and color density parameter vector $\underline{\lambda}$ -- briefly denoted by $G \sim \mathcal{G}_n (V, \underline{\lambda})$ -- if $|V| \le n$ and $\underline{E}=(E_1, \ldots, E_k)$, where $E_1, \ldots, E_k$ are independent in the stochastic sense and $E_i$ is the edge set of an Erd\H{o}s--R\'{e}nyi graph with vertex set $V$ and edge probability $p_i = 1 - \exp \left( - \lambda_i / n \right)$ (i.e., each possible edge of $\binom{V}{2}$ is included in $E_i$ independently of each other with probability $p_i$) for all $i \in [k]$. 
\end{definition}

Note that similarly to uncolored Erd\H{o}s--R\'{e}nyi graphs, the law of an ECER graph is also invariant under permutations of its vertex set; we refer to this property as \emph{vertex exchangeability}.

\begin{definition}[Empirical density of vertices in color-avoiding connected components of size $\ell$] \label{def:f_ell_G}
 Let $k \in \mathbb{N}_+$ and $\underline{\lambda} \in \mathbb{R}^k_+$. Given an ECER graph $G \sim \mathcal{G}_n ([n], \underline{\lambda})$, let
 \[ f_{\ell}(G) \colonequals f_{\ell}(\underline{\lambda}, G) \colonequals \frac{1}{n} \sum_{v \in [n]} \mathds{1} \bigg[ \Big| \, \widetilde{\mathcal{C}}(v) \Big| = \ell \bigg] \]
 denote the fraction of vertices contained in color-avoiding connected components of size $\ell$ for any $\ell \in \mathbb{N}_+$.
\end{definition}

The next claim directly follows from this definition.

\begin{claim}[Size of color-avoiding connected components in ECER graphs] \label{claim:sum_of_f_ell_G's}
 Let $k, n \in \mathbb{N}_+$, $\underline{\lambda} \in \mathbb{R}^k_+$, and $G \sim \mathcal{G}_n([n],\underline{\lambda})$ an ECER graph. Then
 \[ \sum_{\ell \in [n]} f_{\ell}(G) = 1 \text{.} \]
\end{claim}

In Theorem~\ref{thm:conv_of_f_ell_G_n} below, we prove that given a sequence of ECER graphs $G_n \sim \mathcal{G}_n ([n], \underline{\lambda})$ for all $n \in \mathbb{N}_+$, the sequence $f_{\ell}(G_n)$ converges in probability to a deterministic value $f^*_{\ell}$ for all $\ell \in \mathbb{N}_+$ as $n \to \infty$ if $\underline{\lambda}$ satisfies some condition. We also show that the empirical density of a largest color-avoiding connected component of $G_n$ converges to a deterministic value $f^*_\infty$ as $n \to \infty$. In order to identify the limits $f^*_{\ell}$ and $f^*_\infty$, we need some further notation.

\begin{definition}[Edge-colored Poisson branching process (ECBP) tree] \label{def:ECBP}
 Let $k \in \mathbb{N}_+$ and $\underline{\lambda} \in \mathbb{R}^k_+$. We say that $G_{\infty} \colonequals G_{\infty}(r)$ is an \emph{edge-colored Poisson branching process tree} (or \emph{ECBP tree} for short) with color intensity parameter vector $\underline{\lambda}$ -- briefly denoted by $G_{\infty}(r) \sim \mathcal{G}_{\infty}(\underline{\lambda})$ -- if $G_{\infty} = (V, \underline{E})$ is a (possibly infinite) edge-colored tree on a random vertex set $V$ with a distinguished root vertex $r \in V$, and with a random edge set $\underline{E} = (E_1, \ldots, E_k)$ colored with $k$ colors, and $G_{\infty}$ can be recursively generated in the following way. Assume that we have already generated the graph $G_{\infty}$ up to the $d$-th generation (i.e., up to graph distance $d$ from the root) for some $d \in \mathbb{N}$. Given a vertex $v$ in the $d$-th generation, let $X_{v,i}$ denote the number of vertices in the $(d+1)$-th generation which are connected to $v$ by an edge of color $i$, where $i \in [k]$. Then $X_{v,i} \sim \mathrm{POI}(\lambda_i)$. Moreover, the family
 \[ \big\{ X_{v,i} \; : \; \text{$v$ is a vertex in the $d$-th generation and $i \in [k]$} \big\} \]
 of random variables are conditionally independent given the graph up to generation $d$.
\end{definition}

Since in this article all the edge-colored branching process trees have Poisson offspring distribution, we do not include the word Poisson in the abbreviation ECBP.

\begin{definition}[Friends in ECBP trees] \label{def:friends}
 Let $k \in \mathbb{N}_+$ and $\underline{\lambda} \in \mathbb{R}^k_+$. Given an ECBP tree $G_{\infty} \colonequals G_{\infty}(r) \sim \mathcal{G}_\infty(\underline{\lambda})$, a vertex $v \in V(G_{\infty})$ and a color $i \in [k]$, we say that $v$ is \emph{$i$-avoiding connected to infinity}, briefly denoted by 
 \[ v \stackrel{G_{\infty}^{\setminus i}}{\longleftrightarrow} \infty \text{,} \]
 if there exists an infinitely long path from $v$ which does not contain any edges of color $i$ and all of the vertices on this path are descendants of $v$.
 
 We say that the root $r$ and the vertex $v$ are \emph{$i$-avoiding friends} if the event
 \[ \Big\{ r \stackrel{G_{\infty}^{\setminus i}}{\longleftrightarrow} v \Big\} \cup \left( \Big\{ r \stackrel{G_{\infty}^{\setminus i}}{\longleftrightarrow} \infty \Big\} \cap \Big\{ v \stackrel{G_{\infty}^{\setminus i}}{\longleftrightarrow} \infty \Big\} \right) \]
 occurs. In words: $r$ and $v$ are $i$-avoiding friends if and only if there is an $i$-avoiding $r$-$v$ path or both of them are $i$-avoiding connected to infinity (in which case we say that they are $i$-avoiding connected through infinity).
 
 We say that $r$ and $v$ are \emph{(color-avoiding) friends} if they are $i$-avoiding friends for all $i \in [k]$. The set of friends of the root $r$ are denoted by $\widetilde{\mathcal{C}}^*(r) \colonequals \widetilde{\mathcal{C}}^*(r, G_{\infty})$.
\end{definition}

Note that the root $r$ is always a friend of itself.

\begin{definition}[Probability of $\ell$ friends in ECBP trees] \label{def:fstar_ell}
 Let $k \in \mathbb{N}_+$ and $\underline{\lambda} \in \mathbb{R}^k_+$. Given an ECBP tree $G_{\infty}(r) \sim \mathcal{G}_{\infty}(\underline{\lambda})$ and $\ell \in \mathbb{N}_+ \cup \{ \infty \}$, let
 \[ f^*_{\ell} \colonequals f^*_{\ell}(\underline{\lambda}) \colonequals \mathbb{P} \bigg( \, \Big| \widetilde{\mathcal{C}}^*(r) \Big| = \ell \, \bigg) \]
 denote the probability that the root $r$ has exactly $\ell$ friends.
\end{definition}

The next claim directly follows from this definition.

\begin{claim}[Number of friends in ECBP trees] \label{claim:sum_of_fstar_ell's}
 Let $k \in \mathbb{N}_+$, $\underline{\lambda} \in \mathbb{R}^k_+$, and $G_{\infty}(r) \sim \mathcal{G}_{\infty}(\underline{\lambda})$ an ECBP tree. Then
 \[ f^*_\infty + \sum_{\ell \in \mathbb{N}_+} f^*_{\ell} = 1 \text{.} \]
\end{claim}

\begin{definition}[Density/intensity parameters of uncolored color-subgraphs] \label{def:lambda_I}
 Let $k \in \mathbb{N}_+$, $\underline{\lambda} = (\lambda_1, \ldots, \lambda_k) \in \mathbb{R}_+^k$, and for any $i \in [k]$ or $I \subseteq [k]$, let us denote $\lambda_I \colonequals \sum_{i \in I} \lambda_i$, $\lambda^{\text{uc}} \colonequals \lambda_{[k]}$, and $\lambda^{\setminus i} \colonequals \lambda_{[k] \setminus \{ i \}}$.
\end{definition}

\begin{definition}[Fully supercitical/critical-subcritical $\underline{\lambda}$] \label{def:fully_supercrit}
 Let $k \in \mathbb{N}_+$. We say that a vector $\underline{\lambda} \in \mathbb{R}_+^k$ is \emph{fully supercritical} if $\lambda^{\setminus i} > 1$ for all $i \in [k]$, and we say that $\underline{\lambda}$ is \emph{fully critical-subcritical} if $\lambda^{\setminus i} \le 1$ for all $i \in [k]$.
\end{definition}

Note that a vector might be neither fully supercritical nor fully critical-subcritical.

We make the following simplifying assumption on the color density parameter vector $\underline{\lambda}$ when we state some of our main results.

\begin{assumption}[Simplifying assumption] \label{asm:simplifying_assumption}
 Let $k \ge 2$ be an integer, let $\underline{\lambda} \in \mathbb{R}_+^k$ and assume that $\lambda_I < 1$ for any $I \subseteq [k]$ with $|I| \le k-2$.
\end{assumption}

Heuristically, if $\underline{\lambda}$ is fully supercritical and Assumption~\ref{asm:simplifying_assumption} holds, then one really needs to use all of the other $k-1$ colors to create an infinitely long $i$-avoiding path for each $i \in [k]$ in an ECBP tree.

\begin{theorem}[Convergence of empirical component size densities] \label{thm:conv_of_f_ell_G_n}
 Let $k \in \mathbb{N}_+$, $\underline{\lambda} \in \mathbb{R}_+^k$ and let $G_n \sim \mathcal{G}_n([n],\underline{\lambda})$ for all $n \in \mathbb{N}_+$. If Assumption~\ref{asm:simplifying_assumption} holds, then
 \[ \frac{1}{n} \max_{v \in [n]} \Big| \widetilde{\mathcal{C}}(v, G_n) \Big| \; \stackrel{\mathbb{P}}{\longrightarrow} \; f^*_\infty, \qquad n \to \infty \]
 and for any $\ell \in \mathbb{N}_+$, we have
 \[ f_{\ell}(G_n) \; \stackrel{\mathbb{P}}{\longrightarrow} \; f^*_{\ell}, \qquad n \to \infty \text{.} \]
\end{theorem}

We prove the above theorem in Section~\ref{section:conv_of_f_ell_G_n}. Note that a short alternative proof of the above theorem without using Assumption~\ref{asm:simplifying_assumption} is given in the recent paper~\cite{lichevschapira}.

From Theorem~\ref{thm:conv_of_f_ell_G_n} and Claim~\ref{claim:sum_of_fstar_ell's}, we can conclude the following.

\begin{corollary}[Existence of \emph{giant} color-avoiding connected component in ECER graphs]
 Let $k, n \in \mathbb{N}_+$, $\underline{\lambda} \in \mathbb{R}_+^k$ and let $G_n \sim \mathcal{G}_n([n],\underline{\lambda})$. If Assumption~\ref{asm:simplifying_assumption} holds and $f^*_{\infty} > 0$, then $G_n$ has a unique largest color-avoiding connected component asymptotically almost surely as $n \to \infty$.
\end{corollary}

\begin{proposition}[Characterization of existence of giant color-avoiding connected component in ECER graphs] \label{prop:char_of_giant}
 Let $k \in \mathbb{N}_+$. A vector $\underline{\lambda} \in \mathbb{R}_+^k$ is fully supercritical if and only if $f^*_{\infty} > 0$.
\end{proposition}

We prove Proposition~\ref{prop:char_of_giant} in Section~\ref{subsection:conv_of_giant_block}. In Section~\ref{section:explicit_formulas}, we give explicit formulas for $f^*_{\infty}$ in terms of the Lambert $W$ function. These formulas are recursive, and as the number of colors $k$ increases, the formulas get more complex.

\begin{proposition}[Characterization of triviality of component structure in ECER graphs] \label{prop:char_of_component_structure}
 Let $k \in \mathbb{N}_+$ and let $\underline{\lambda} \in \mathbb{R}_+^k$ for which Assumption~\ref{asm:simplifying_assumption} holds.
 \begin{enumerate}[(i)]
  \item \label{item:fully} If $\underline{\lambda}$ is fully critical-subcritical, then $f^*_{\ell} = \mathds{1}[\ell=1]$ for any $\ell \in \mathbb{N}_+$.
  \item \label{item:not_fully} If $\underline{\lambda}$ is not fully critical-subcritical, then $f^*_{\ell} > 0$ for any $\ell \in \mathbb{N}_+$.
 \end{enumerate}
\end{proposition}

We prove Proposition~\ref{prop:char_of_component_structure} in Section~\ref{subsection:conv_of_small_comps}.

We note that the recent paper~\cite{lichevschapira} includes results about the size of biggest color-avoiding component for certain parameter vectors $\underline{\lambda}$ which are neither fully subcritical nor fully critical-subcritical (see Theorem~1.1(ii) of~\cite{lichevschapira}) as well as fully subcritical vectors $\underline{\lambda}$ (see Theorem~1.1(iii) of~\cite{lichevschapira}).

A vector $\underline{\lambda} = (\lambda_1, \ldots, \lambda_k) \in \mathbb{R}_+^k$ for some integer $k \ge 2$ is called \emph{homogeneous} if $\lambda_1 = \ldots = \lambda_k$.

\begin{definition}[Barely supercritical homogeneous setup] \label{def:barely_supercrit}
 Let $k \ge 2$ be an integer. For any $\varepsilon \ge 0$, let us define $\underline{\lambda}(\varepsilon) \colonequals \left(\frac{1+\varepsilon}{k-1}, \ldots, \frac{1+\varepsilon}{k-1} \right) \in \mathbb{R}_+^k$. Let us define $f^*_{\infty}(\varepsilon) \colonequals f^*_{\infty} \big( \underline{\lambda}(\varepsilon) \big)$. 
\end{definition}

Note that if $\varepsilon > 0$, then $f^*_{\infty}(\varepsilon) > 0$, but if $\varepsilon = 0$, then $f^*_{\infty}(\varepsilon) = 0$ by Proposition~\ref{prop:char_of_giant}. Also note that if 
$\varepsilon \ge 0$ is small enough, then Assumption~\ref{asm:simplifying_assumption} holds for $\underline{\lambda}(\varepsilon)$.

Our next result is about the asymptotic behaviour of $f^*_{\infty}(\varepsilon)$ as $\varepsilon \to 0_+$.

\begin{theorem}[Size of barely supercritical homogeneous giant] \label{thm:barely_supcrit}
 For any integer $k \ge 2$, there exists a constant $C(k) \in \mathbb{R}_+$ such that
 \begin{equation*}
  \lim_{\varepsilon \to 0_+} \frac{f^*_{\infty}(\varepsilon)}{\varepsilon^k} = C(k) \text{.}
 \end{equation*}
 In particular, $C(2)= 4$,  $C(3)=32$, and $C(4)=624$.
\end{theorem}

In Section~\ref{section:explicit_formulas}, we provide a method for computing the value of $C(k)$ for any integer $k \ge 2$ and prove Theorem~\ref{thm:barely_supcrit}.

\section{Notation} \label{section:notation}

In this section, we introduce some notation that we use throughout the paper. The central definition of our paper concerns the set of vertices reachable from a vertex by a fixed chronology of colors, see Definition~\ref{def:layers}. We also compiled a list of all notations that appear in the paper starting on page \pageref{list_of_symbols_starts_here}.

\medskip

Given an (edge-colored) graph $G$ and a set of vertices $W \subseteq V(G)$, let $G[W]$ denote the (edge-colored) subgraph of $G$ induced by $W$, and let $G-W$ denote the (edge-colored) graph obtained from $G$ by deleting the vertices of $W$ (and all the edges incident to any vertex of $W$).

Given an edge-colored graph $G = (V, \underline{E})$ with $k \in \mathbb{N}_+$ colors, where $\underline{E}=(E_1,\ldots,E_k)$, and a colored edge set $\underline{E}' = (E'_1, \ldots, E'_k)$ satisfying $E'_i \subseteq E_i$ for each $i \in [k]$, let $G-\underline{E}'$ denote the edge-colored graph that is obtained from $G$ by deleting the edges of $E'_i$ from $E_i$ for each $i \in [k]$.

For two (edge-colored) graphs $G_1$ and $G_2$, let $G_1 \oplus G_2$ denote their disjoint union.

We say that two edge-colored graphs $G_1$ and $G_2$ are \emph{isomorphic} -- briefly denoted by $G_1 \simeq G_2$ -- if there exists a bijection between their vertex sets which is an isomorphism for the set of edges of each color.

We say that an edge-colored graph is \emph{rooted} if it has a distinguished vertex. We say that the rooted edge-colored graphs $G_1$ and $G_2$ with distinguished vertices $v_1$ and $v_2$, respectively, are \emph{isomorphic} if there exists an isomorphism between them mapping $v_1$ to $v_2$. We say that an edge-colored graph is \emph{two-rooted} if it has two distinct distinguished vertices. We say that the two-rooted edge-colored graphs $G_1$ and $G_2$ with distinguished vertices $v_1, w_1$ and $v_2, w_2$, respectively, are \emph{isomorphic} if there exists an isomorphism between them mapping $v_1$ to $v_2$ and $w_1$ to $w_2$.

Let $G=(V,E)$ be a graph without an edge-coloring. For a vertex $v \in V$, let $\mathcal{C}(v) \colonequals \mathcal{C}(v,G)$ denote the set of vertices of the connected component of $v$ in $G$. Denote by
\begin{equation} \label{eq:U_max}
 \mathcal{U}_{\text{max}}(G) \colonequals \bigcup \, \{ \, \mathcal{C} : \text{$\mathcal{C}$ is a component of $G$ with maximum cardinality} \} \text{,}
\end{equation}
i.e.\ the vertices of the union of the connected components of $G$ with maximum cardinality. For a vertex $v \in V(G)$, let $N(v) \colonequals N(v,G)$ denote the open neighborhood of $v$, i.e., the set of the neighbors of $v$. For a set of vertices $W \subseteq V(G)$, let
\begin{equation} \label{eq:N_W}
 N(W) \colonequals N(W,G) \colonequals \left( \bigcup_{v \in W} N(v) \right) \setminus W
\end{equation}
denote the open neighborhood of $W$, and let
\begin{equation} \label{eq:R_W}
 R(W)\colonequals R(W,G) \colonequals \bigcup_{v \in W} \mathcal{C}(v)
\end{equation}
denote the set of vertices reachable from $W$ in $G$. For the vertices $v,w \in V(G)$, let $\mathrm{dist}(v,w) \colonequals \mathrm{dist}_G(v,w)$ denote the distance of $v$ and $w$ (i.e., the number of edges on a shortest $v$-$w$ path) in $G$. For a vertex $v \in V(G)$ and a vertex set $W \subseteq V(G)$, let
\begin{equation} \label{eq:dist_v_W}
 \mathrm{dist}(v,W) \colonequals \mathrm{dist}_G(v,W) \colonequals \min \big\{ \mathrm{dist}(v,w) : w \in W \big\} \text{.}
\end{equation}

For an edge-colored graph $G$, two vertices $v,w \in V(G)$ and a set of vertices $W \subseteq V(G)$, let us use the notation $\mathcal{C}(v,G) \colonequals \mathcal{C}(v, G^{\text{uc}})$, and $N(W,G) \colonequals N(W, G^{\text{uc}})$, and $R(W,G) \colonequals R(W, G^{\text{uc}})$, and $\mathrm{dist}_G(v,w) \colonequals \mathrm{dist}_{G^{\text{uc}}}(v,w)$, and $\mathrm{dist}_G(v,W) \colonequals \mathrm{dist}_{G^{\text{uc}}}(v,W)$.

\begin{definition}[Color strings of length $h$ without repetition] \label{def:color_strings}
 Let $k \in \mathbb{N}_+$ and let $h \in \{ 0, 1, \ldots, k \}$. Let $S_h$ denote the set of strings $\underline{s} = (i_1, \ldots, i_h)$ of colors without any repetition, i.e., $i_1, \ldots, i_h \in [k]$ and $\big| \{ i_1, \ldots, i_h \} \big| = h$; whereby the set $S_0$ has one element, namely the empty string $\emptyset$.
 
 Let us define $\mathrm{set}(\underline{s}) \colonequals \{ i_1, \ldots, i_h \}$ for a string $\underline{s} = (i_1, \ldots, i_h) \in S_h$.
\end{definition}

Note that $|S_h| = k \cdot (k-1) \cdot \ldots \cdot (k-h+1)$ for any $h \in [k]$.

\begin{definition}[Color string operations] \label{def:color_string_operations}
 Let $k \in \mathbb{N}_+$. For any color string $\underline{s} = (i_1, \ldots, i_h) \in S_h$, where $h \in [k]$, let $\underline{s}^- \colonequals (i_1, \ldots, i_{h-1}) \in S_{h-1}$ denote the color string that can be obtained from $\underline{s}$ by deleting its last coordinate. For any $\underline{s} = (i_1, \ldots, i_h) \in S_h$, where $h \in \{ 0, 1, \ldots, k-1 \}$, and for any $i \in [k] \setminus \mathrm{set}(s)$, let $\underline{s}i \colonequals (i_1, \ldots, i_h, i) \in S_{h+1}$ denote the concatenation of $\underline{s}$ and $i$.
\end{definition}

For an edge-colored forest $F$ with $k \in \mathbb{N}_+$ colors, a vertex $v \in V(F)$ and a color string $\underline{s} = (i_1, \ldots, i_h) \in S_h$, where $h \in [k]$, let us define the following sets of vertices. Let $\widetilde{R}_{\underline{s}}(v,F)$ denote the set of vertices that can be reached from $v$ with a path whose edges are of color $i_1, \ldots, i_h$ and the new colors on this path appear in this specific order. Let $\widetilde{N}_{\underline{s}}(v,F)$ denote the set of vertices that can be reached from $v$ with a path whose edges except the last one are of color $i_1, \ldots, i_{h-1}$, the last edge is of color $i_h$, and the new colors on this path appear in this specific order. For an edge-colored graph $G$, the sets $\widetilde{R}_{\underline{s}}(v,G)$ and $\widetilde{N}_{\underline{s}}(v,G)$, are defined the same way, but with the extra requirement that their vertices cannot be reached with a path that uses at most $h - 1$ colors. In addition, let $\widetilde{R}^{\le}_h(v,G)$ denote the set of vertices that can be reached with a path that uses at most $h$ different colors. For an example, see Figure~\ref{fig:layers}. Formally, these definitions are introduced as follows.

\begin{definition}[Set of vertices reachable by a fixed chronology of colors] \label{def:layers}
 Given an edge-colored graph $G=(V,\underline{E})$ with $k \in \mathbb{N}_+$ colors, and given a vertex $v \in V$, we define the set of verticess $\widetilde{R}^{\le}_{h-1}(v) \colonequals \widetilde{R}^{\le}_{h-1}(v,G)$, $\widetilde{N}_{\underline{s}}(v) \colonequals \widetilde{N}_{\underline{s}}(v,G)$, and $\widetilde{R}_{\underline{s}}(v) \colonequals \widetilde{R}_{\underline{s}}(v,G)$ for all $\underline{s} \in S_h$ and $h \in [k]$ recursively.
 \begin{itemize}
  \item Let $\widetilde{R}_{\emptyset}(v) \colonequals \{ v \}$.
  
  \item Assume that $\widetilde{R}_{\underline{s}}(v)$ is already defined for all strings $\underline{s} \in S_j$ and for all $j \in \{ 0, 1, \ldots, h-1 \}$. Let us first define
  \[ \widetilde{R}^{\le}_{h-1}(v) \colonequals \bigcup_{j=0}^{h-1} \, \bigcup_{\underline{s} \in S_j} \widetilde{R}_{\underline{s}}(v) \text{.} \]
  If $\underline{s} \in S_h$ and $i \in [k]$ is the last coordinate of $\underline{s}$, then let
  \[ \widetilde{N}_{\underline{s}}(v) \colonequals N \left( \widetilde{R}_{\underline{s}^-}(v,G), \, G^{\{ i \}} \right) \setminus \widetilde{R}^{\le}_{h-1}(v,G) \text{,} \]
  and finally let
  \[ \widetilde{R}_{\underline{s}}(v) \colonequals R \left( \widetilde{N}_{\underline{s}}(v,G), G^{\mathrm{set}(\underline{s})} - \widetilde{R}_{\underline{s}^-}(v,G) \right) \text{.} \]
 \end{itemize}
 
 For a set of vertices $W \subseteq V$, let
 \[ \widetilde{R}^{\le}_{h-1}(W) \colonequals \widetilde{R}^{\le}_{h-1}(W,G) \colonequals \bigcup_{w \in W} \widetilde{R}^{\le}_{h-1}(w) \text{.} \]
\end{definition}

\begin{figure}[H]
 \centering
 \begin{tikzpicture}
  \tikzstyle{vertex}=[draw,circle,fill,minimum size=8,inner sep=0]
  
  \coordinate (r) at (0,0);
  \coordinate (v1) at (1,1);
  \coordinate (v2) at (1,-1);
  \coordinate (v3) at (2.25,1.5);
  \coordinate (v4) at (2.25,0.5);
  \coordinate (v5) at (2.25,-0.5);
  \coordinate (v6) at (2.25,-1.5);
  \coordinate (v7) at (3.75,1.5);
  \coordinate (v8) at (3.75,0.5);
  \coordinate (v9) at (3.75,-1.5);
  \coordinate (v10) at (5.25,1.5);
  \coordinate (v11) at (6.75,1.5);
  \coordinate (v12) at (8.25,1.5);
  
  \draw[decorate with = isosceles triangle, paint=red] (r) -- (v1);
  \draw[decorate with = rectangle, paint=blue] (r) -- (v2);
  \draw[decorate with = isosceles triangle, paint=red] (v1) -- (v3);
  \draw[decorate with = rectangle, paint=blue] (v1) -- (v4);
  \draw[decorate with = isosceles triangle, paint=red] (v2) -- (v5);
  \draw[decorate with = star, paint=green] (v2) -- (v6);
  \draw[decorate with = rectangle, paint=blue] (v3) -- (v7);
  \draw[decorate with = star, paint=green] (v4) -- (v8);
  \draw[decorate with = isosceles triangle, paint=red] (v6) -- (v9);
  \draw[decorate with = isosceles triangle, paint=red] (v7) -- (v10);
  \draw[decorate with = rectangle, paint=blue] (v10) -- (v11);
  \draw[decorate with = star, paint=green] (v11) -- (v12);
  
  \node[vertex] at (r) [label=left: $r$] {};
  \node[vertex] at (v1) [label=above: $v_1$] {};
  \node[vertex] at (v2) [label=below: $v_2$] {};
  \node[vertex] at (v3) [label=above: $v_3$] {};
  \node[vertex] at (v4) [label={[xshift=0pt, yshift=-21pt] $v_4$}] {};
  \node[vertex] at (v5) [label={[xshift=0pt, yshift=-3pt] $v_5$}] {};
  \node[vertex] at (v6) [label=below: $v_6$] {};
  \node[vertex] at (v7) [label=above: $v_7$] {};
  \node[vertex] at (v8) [label={[xshift=0pt, yshift=-21pt] $v_8$}] {};
  \node[vertex] at (v9) [label=below: $v_9$] {};
  \node[vertex] at (v10) [label=above: $v_{10}$] {};
  \node[vertex] at (v11) [label=above: $v_{11}$] {};
  \node[vertex] at (v12) [label=above: $v_{12}$] {};
 \end{tikzpicture}
 \captionsetup{singlelinecheck=off}
 \caption[]{In the picture, the colors red, blue, and green are denoted by (red) triangles, (blue) rectangles, and (green) stars, respectively. In this graph,
 \vspace{-7pt}
 \begin{itemize} \setlength{\itemsep}{-2pt}
  \item[] $\widetilde{R}_{(\mathrm{red})}(r) = \{ v_1, v_3 \}$, \quad $\widetilde{R}_{(\mathrm{blue})}(r) = \{ v_2 \}$, \quad $\widetilde{R}_{(\mathrm{green})}(r) = \emptyset$,
  \item[] $\widetilde{R}_{(\mathrm{red}, \mathrm{blue})}(r) = \{ v_4, v_7, v_{10}, v_{11} \}$, \quad $\widetilde{R}_{(\mathrm{blue}, \mathrm{red})}(r) = \{ v_5 \}$,
  \item[] $\widetilde{R}_{(\mathrm{red}, \mathrm{blue}, \mathrm{green})}(r) = \{ v_8, v_{12} \}$, \quad $\widetilde{R}_{(\mathrm{blue}, \mathrm{green}, \mathrm{red})}(r) = \{ v_9 \}$,
  \item[] $\widetilde{R}^{\le}_1 = \{ v_1, v_2, v_3 \}$,
  \item[] $\widetilde{N}_{(\mathrm{red})}(r) = \{ v_1 \}$,~ $\widetilde{N}_{(\mathrm{red}, \mathrm{blue})}(r) = \{ v_4, v_7 \}$,~ $\widetilde{N}_{(\mathrm{red}, \mathrm{blue}, \mathrm{green})}(r) = \{ v_8, v_{12} \}$.
 \end{itemize}}
 \label{fig:layers}
\end{figure}
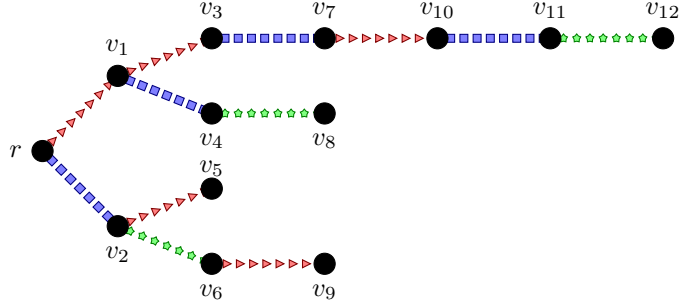

\begin{definition}[Size of the set reachable using $k-2$ colors] \label{def:rho}
 Given an edge-colored graph $G=(V,\underline{E})$ with $k \in \mathbb{N}_+$ colors and $v \in V$, let us denote
 \[ \rho(v) \colonequals \rho(v,G) \colonequals \left| \widetilde{R}^{\le}_{k-2}(v) \right| \text{.} \]
\end{definition}

\begin{claim}[Size of the set reachable using $k-2$ colors in ECBP trees] \label{claim:finiteness_of_rhor}
 Let $k \in \mathbb{N}_+$, $\underline{\lambda} \in \mathbb{R}_+^k$ and $G_{\infty}(r) \sim \mathcal{G}_{\infty}(\underline{\lambda})$. If Assumption~\ref{asm:simplifying_assumption} holds, then $\rho(r)$ is almost surely finite.
\end{claim}

The next claim follows directly from Definition~\ref{def:layers}.

\begin{claim}[Connections between connected components and reachable sets] \label{claim:union_of_bags}
 Let $G=(V,\underline{E})$ be an edge-colored graph with $k \in \mathbb{N}_+$ colors and let $v \in V$. Then for any $I \subseteq [k]$, we have
 \[ \mathcal{C} \big( v, G^I \big) = \bigcup_{\underline{s} \, : \, \mathrm{set}(\underline{s}) \subseteq I} \widetilde{R}_{\underline{s}}(v,G) \]
 and
 \[ \widetilde{R}^{\le}_{k-2}(v,G) = \bigcup_{\substack{I \subseteq [k] \\ |I|=k-2}} \mathcal{C} \big( v, G^I \big) \text{.} \]
\end{claim}

\begin{definition}[The outer boundary of the set of vertices reachable with $k-2$ colors] \label{def:outer_boundary}
 Given an edge-colored graph $G=(V,\underline{E})$ with $k \in \mathbb{N}_+$ colors, a color $i \in [k]$, a number $h \in \{ 0, 1, \ldots, k-1 \}$ and a vertex $v \in V$, let us denote
 \[ \widetilde{N}_h(v) \colonequals \widetilde{N}_h(v,G) \colonequals \bigcup_{\underline{s} \in S_h} \widetilde{N}_{\underline{s}}(v) \text{,} \]
 \[ \widetilde{N}_{k-1}^{\setminus i}(v) \colonequals \widetilde{N}_{k-1}^{\setminus i}(v,G) \colonequals \bigcup_{\substack{\underline{s} \in S_{k-1}: \\ \mathrm{set}(\underline{s}) = [k] \setminus \{ i \}}} \widetilde{N}_{\underline{s}}(v) \text{,} \]
 and
 \[ \underline{b}(v) \colonequals \underline{b}(v,G) \colonequals \Bigg( \Big| \widetilde{N}_{k-1}^{\setminus i}(v) \Big| \Bigg)_{i \in [k]} \text{.} \]
 
 For a set of vertices $W \subseteq V$, let
 \[ \widetilde{N}_h(W) \colonequals \widetilde{N}_h(W,G) \colonequals \bigcup_{w \in W} \widetilde{N}_h(w) \text{.} \]
\end{definition}

In words: The set $\widetilde{N}_h(v)$ consists of those vertices that can be reached from $v$ with a path which uses exactly $h$ colors and the color of the last edge is new, but cannot be reached with a path of at most $h-1$ colors. The set $\widetilde{N}_{k-1}^{\setminus i}(v)$ consists of those vertices that can be reached from $v$ with a path which uses all of the colors in $[k] \setminus \{ i \}$ and the color of the last edge is new, but cannot be reached with a path of at most $k-2$ colors. The vector $\underline{b}(v)$ is of length $k$ and contains the cardinalities of the boundary sets $\widetilde{N}_{k-1}^{\setminus i}(v)$ for all $i \in [k]$.


\section{The connection between ECER graphs and ECBP trees} \label{section:connection_between_ECER_and_ECBP}

Section~\ref{section:connection_between_ECER_and_ECBP} contains some auxiliary results for proving the statements of Section~\ref{section:main_results}.

In Section~\ref{subsection:exploring}, we introduce some notation and state some basic results about the exploration of ECER graphs and ECBP trees. 

In Section~\ref{subsection:types}, we introduce the notion of \emph{color-avoiding type I vectors}: informally, in the case of ECER graphs, the $i$-th coordinate of the type I vector of a vertex $v$ is the indicator of the event that $v$ is in the $i$-avoiding giant component, while in the case of an ECBP trees, the $i$-th coordinate of the type I vector of a vertex $v$ is the indicator of the event that $v$ is $i$-avoiding connected to infinity for some color $i$. The main result of Section~\ref{subsection:types} states that the empirical frequency of a type I vector in a sequence of ECER graphs converges as $n \to \infty$ to the probability of the root having the same type I vector in an ECBP tree with the same parameter vector $\underline{\lambda}$.

In Section~\ref{subsection:types_and_color_avoiding_connectivity}, we further elaborate on the relations between the notion of type I vectors, color-avoiding connectivity in ECER graphs and color-avoiding friendship in ECBP trees.

\subsection{Exploration of ECER graphs and ECBP trees} \label{subsection:exploring}

The main goal of Section~\ref{subsection:exploring} is to formalize the idea that ECER graphs locally look like ECBP trees.

We use two types of exploration methods in edge-colored graphs: given a vertex $v$, in Definition~\ref{def:exploring_up_to_dist_d} we explore the vertices reachable from $v$ with at most $d$ edges, while in Definition~\ref{def:exploring_up_to_colordist_h} we explore the vertices reachable from $v$ with at most $h$ colors plus one more final edge that has a different color for some $d, h \in \mathbb{N}$. Note that if $h \le k-2$, then using the second techniques we explore a ``small'' part, due to our Assumption~\ref{asm:simplifying_assumption}.

We observe that even if we condition on what we explored in an ECER graph with any of these methods, the remaining graph is ECER, cf.\ Lemma~\ref{lemma:after_conditioning_still_ECER}. In Definition~\ref{def:good_events}, we collect some ``good'' properties of the explored and unexplored parts of ECER graphs and in Proposition~\ref{prop:good_events}, we prove (using known results about Erd\H{o}s--R\'{e}nyi random graphs) that these good properties do hold with high probability as $n \to \infty$.

\begin{definition}[Exploring up to distance $d$] \label{def:exploring_up_to_dist_d}
 Let $G = (V,E)$ be a graph, $v, w \in V$ a vertex, $W \subseteq V$ a set of vertices, and $d \in \mathbb{N}$. Let us define
 \[ R_d(v) \colonequals R_d(v,G) \colonequals \big\{ u \in V : \mathrm{dist}(u,v) = d \big\} \text{,} \]
 i.e.\ the set of vertices of $G$ at distance $d$ from the root, and let
 \[ R_d(W) \colonequals R_d(W,G) \colonequals \big\{ u \in V : \mathrm{dist}(u,W) = d \big\} \text{.} \]
 Let
 \[ R^{\le}_d(v) \colonequals R^{\le}_d(v,G) \colonequals \big\{ u \in V : \mathrm{dist}(u,v) \le d \big\} \text{,} \]
 i.e.\ the set of vertices of $G$ at distance at most $d$ from the root, and let
 \[ R^{\le}_d(W) \colonequals R^{\le}_d(W,G) \colonequals \bigcup_{w \in W} R^{\le}_d(w) \text{.} \]
 
 Let us define
 \[ G_d(v) \colonequals G \Big[ R^{\le}_d(v) \Big] - E \Big( G \big[ R_d(v) \big] \Big) \text{,} \]
 i.e. the subgraph of $G$ spanned by the vertices we can reach from $v$ using at most $d-1$ edges, plus the edges (along with their endpoints) that connect this spanned subgraph to the rest of the graph.
 
 We also define a similar subgraph with two roots,
 \[ G_d(v,w) \colonequals G \Big[ R^{\le}_d \big( \{ v,w \} \big) \Big] - E \Big( G \big[ R_d \big( \{ v,w \} \big) \big] \Big) \text{.} \]
 
 Let us define
 \[ G'_d(v) \colonequals G - \Big( R^{\le}_{d-1} (v) \Big) \text{,} \]
 i.e. the subgraph of the graph $G$ that can be obtained by deleting the vertices at distance at most $d-1$ from $v$.
 
 We also define a similar subgraph
 \[ G'_d(v,w) \colonequals G - \Big( R^{\le}_{d-1} \big( \{ v,w \} \big) \Big) \text{.} \]
 
 For an ECER graph $G_n \sim \mathcal{G}_n (V, \underline{\lambda})$ (where $n \in \mathbb{N}_+$ and $\lambda \in \mathbb{R}_+^k$ for some $k \in \mathbb{N}_+$) or for an ECBP tree $G_{\infty} \colonequals G_{\infty}(r) \sim \mathcal{G}_{\infty}(\underline{\lambda})$, we use the notation 
 \[ R_{n,d}(v) \colonequals R_d(v, G_n), \qquad R_{n,d}(W) \colonequals R_d(W, G_n), \qquad R_{\infty,d} \colonequals R_d(r, G_{\infty}) \text{,} \]
 \[ R^{\le}_{n,d}(v) \colonequals R^{\le}_d(v, G_n), \qquad R^{\le}_{n,d}(W) \colonequals R^{\le}_d(v, G_n), \qquad R^{\le}_{\infty,d} \colonequals R^{\le}_d(r, G_{\infty}) \text{,} \]
 and we use the notation $G_{n,d}(v)$, $G_{\infty,d}(r)$, $G_{n,d}(v,w)$, $G'_{n,d}(v)$, $G'_{n,d}(v,w)$ for the above defined subgraphs, respectively. In addition, for a color $i \in [k]$ or for a subset of colors $I \subseteq [k]$, we use the notation
 \[ R_{\infty,d}^I \colonequals R_d \big(r, G_{\infty}^I \big), \qquad R_{\infty,d}^{\setminus i} \colonequals R_{\infty,d}^{[k] \setminus \{ i \}} \text{.} \]
\end{definition}

\begin{definition}[Exploring until you see $h+1$ colors for the first time] \label{def:exploring_up_to_colordist_h}
 Let $k \in \mathbb{N}_+$, $h \in \{ 0, 1, \ldots, k-2 \}$ and $G = (V,\underline{E})$ be a random edge-colored graph with $k$ colors, let $v, w \in V$, and let us define the following.
 
 Let $\mathcal{F}_{v,h}$ denote the sigma-algebra
 \[ \mathcal{F}_{v,h} \colonequals \mathcal{F}_{v,h}(G) \colonequals \sigma \left( \mathds{1} \big[ \{ u_1, u_2 \} \in E_i \big] \; : \; u_1 \in \widetilde{R}^{\le}_h(v), \; u_2 \in V, \; i \in [k] \, \right) \text{,} \]
 and let $\mathcal{F}_{v,w,h}$ denote the sigma-algebra
 \[ \mathcal{F}_{v,w,h} \colonequals \mathcal{F}_{v,w,h}(G) \colonequals \sigma \big( \mathcal{F}_{v,h}, \mathcal{F}_{w,h} \big) \text{.} \]
 
 Let $\widetilde{G}_h(v)$ denote the following edge-colored, rooted subgraph of $G$:
 \[ \widetilde{G}_h(v) \colonequals G \left[ \widetilde{R}^{\le}_h(v) \cup \widetilde{N}_{h+1}(v) \right] - \underline{E} \bigg( G \Big[ \widetilde{N}_{h+1}(v) \Big] \bigg) \text{,} \]
 i.e. the subgraph of $G$ spanned by the vertices we can reach from $v$ using at most $h$ colors, plus the edges (along with their endpoints) that connect this spanned subgraph to the rest of the graph.
 
 We also define a similar edge-colored subgraph with two roots,
 \begin{equation*}
  \widetilde{G}_h(v,w) \colonequals G \left[ \widetilde{R}^{\le}_h \big( \{ v,w \} \big) \cup \widetilde{N}_{h+1} \big( \{ v,w \} \big) \right] - \underline{E} \bigg( G \Big[ \widetilde{N}_{h+1} \big( \{ v,w \} \big) \Big] \bigg) \text{.}
 \end{equation*}
 
 Let $\widetilde{G}'_h(v,w)$ denote the following edge-colored subgraph of $G$:
 \[ \widetilde{G}'_h(v,w) \colonequals G - \left( \widetilde{R}^{\le}_h \big( \{ v,w \} \big) \right) \text{,} \]
 i.e. the subgraph that can be obtained by deleting the vertices reachable from $v$ or $w$ with at most $h$ colors.
 
 For an ECER graph $G_n \sim \mathcal{G}_n (V, \underline{\lambda})$ (where $n \in \mathbb{N}_+$ and $\lambda \in \mathbb{R}_+^k$ for some $k \in \mathbb{N}_+$) or for an ECBP tree $G_{\infty}(r) \sim \mathcal{G}_{\infty}(\underline{\lambda})$, we use the notation $\widetilde{G}_{n,h}(v)$, $\widetilde{G}_{\infty,h}(r)$, $\widetilde{G}_{n,h}(v,w)$, $\widetilde{G}'_{n,h}(v,w)$, respectively, for the above subgraphs.
\end{definition}

\begin{claim}[Known information given $\mathcal{F}_{v,h}$] \label{claim:known_info_given_sigmaalg}
 Let $k\in \mathbb{N}_+$, $h \in \{ 0, 1, \ldots, k-2 \}$ and let $G = (V, \underline{E})$ be an edge-colored graph colored with $k$ colors. Then for any $v, w \in V$ and $\underline{s} \in S_{h+1}$, the sets $\widetilde{R}^{\le}_h(v,G)$ and $\widetilde{N}_{\underline{s}}(v,G)$ are all $\mathcal{F}_{v,h}$-measurable. In particular, $\rho(v)$ and $\underline{b}(v)$ are also $\mathcal{F}_{v,k-2}$-measurable, and $\widetilde{G}_h(v,w)$ is $\mathcal{F}_{v,w,h}$-measurable.
\end{claim}

The next two claims directly follow from the definitions of ECER graphs and ECBP trees.

\begin{claim}[Distribution of uncolored color-subgraphs of ECER graphs] \label{claim:GI_is_ER}
 Let $k,n \in \mathbb{N}_+$, $\underline{\lambda} \in \mathbb{R}_+^k$, and let $G \sim \mathcal{G}_n (V, \underline{\lambda})$. Then
 \[ G^I \sim \mathcal{G}_n (V, \lambda_I) \]
 for any $I \subseteq [k]$. In particular, $G^{\text{uc}} \sim \mathcal{G}_n (V, \lambda^{\text{uc}})$ and $G^{\setminus i} \sim \mathcal{G}_n (V, \lambda^{\setminus i})$ for any $i \in [k]$.
\end{claim}

\begin{claim}[Distribution of uncolored color-subtrees of ECBP trees] \label{claim:GIr_is_BP}
 Let $k \in \mathbb{N}_+$, $\underline{\lambda} \in \mathbb{R}_+^k$, and let $G_{\infty}(r) \sim \mathcal{G}_{\infty}(\underline{\lambda})$. Then $\big( \big| R_{\infty,d}^I \big| \big)_{d \in \mathbb{N}}$ is a branching process with offspring distribution $\mathrm{POI}(\lambda_I)$ for any $I \subseteq [k]$. In particular, $\big( \big| R_{\infty,d}^{\setminus i} \big| \big)_{d \in \mathbb{N}}$ is a branching process with offspring distribution $\mathrm{POI}(\lambda^{\setminus i})$ for any $i \in [k]$.
\end{claim}

\begin{lemma}[The unexplored graph is ECER] \label{lemma:after_conditioning_still_ECER}
 Let $k \ge 2$ be an integer, let $n \in \mathbb{N}_+$, $d \in \mathbb{N}$ and $h \in \{ 0, 1, \ldots, k-2 \}$, let $\underline{\lambda} \in \mathbb{R}_+^k$, $G \sim \mathcal{G}_n([n], \underline{\lambda})$ and $v, w \in [n]$. Then the conditional distribution of the edge-colored subgraph $G'_d(v,w)$ given the graph $G_d(v,w)$ is 
 \[ \mathcal{G}_n \Big( [n] \setminus R^{\le}_{d-1} \big( \{ v,w \}, G \big), \underline{\lambda} \Big) \text{,} \]
 and the conditional distribution of the edge-colored subgraph $\widetilde{G}'_h(v,w)$ given $\mathcal{F}_{v,w,h}$ is 
 \[ \mathcal{G}_n \Big( [n] \setminus \widetilde{R}^{\le}_h \big( \{ v,w \}, G \big), \underline{\lambda} \Big) \text{.} \]
\end{lemma}
\begin{proof}
 We only prove the second statement of the lemma since the first one can be obtained analogously.
 
 It follows from the definitions of $\mathcal{F}_{v,w,h}$ and $\widetilde{G}_h(v,w)$ that given $\mathcal{F}_{v,w,h}$, the edges of the subgraph $\widetilde{G}_h(v,w)$ are determined, and conversely, the subgraph $\widetilde{G}_h(v,w)$ determines every event in $\mathcal{F}_{v,w,h}$.
 
 Therefore, conditional on $\mathcal{F}_{v,w,h}$, the edges of color $i$ between the vertices in $[n] \setminus \widetilde{R}^{\le}_h \big( \{ v,w \}, G \big)$ are still present independently of each other with probability $p_i = 1 - \exp \left( -\lambda_i / n \right)$ for any $i \in [k]$. Therefore, the conditional distribution of $\widetilde{G}'_h(v,w)$, i.e.\ the subgraph of $G$ spanned by $[n] \setminus \widetilde{R}^{\le}_h \big( \{ v,w \}, G \big)$, is an ECER graph with parameter $n$ and color density parameter vector $\underline{\lambda}$ on the vertex set $[n] \setminus \widetilde{R}^{\le}_h \big( \{ v,w \}, G \big)$.
\end{proof}

\begin{proposition}[Colored local weak convergence] \label{prop:benjamini_schramm}
 Let $k \in \mathbb{N}_+$ and $\underline{\lambda} \in \mathbb{R}_+^k$, and let us assume that Assumption~\ref{asm:simplifying_assumption} holds. Let $G_n \sim \mathcal{G}_n([n],\underline{\lambda})$ and $v_1, v_2 \in [n]$, $v_1 \ne v_2$ for all $n \in \mathbb{N}_+$, and let $G_{\infty} \colonequals G_{\infty}(r) \sim \mathcal{G}_{\infty}(\underline{\lambda})$. Then the graph $\widetilde{G}_{\infty,k-2}(r)$ is almost surely finite and $\widetilde{G}_{n,k-2}(v_1,v_2)$ converges in distribution to two i.i.d.\ copies of $\widetilde{G}_{\infty,k-2}(r)$ in the sense that
 \[ \lim_{n \to \infty} \mathbb{P} \Big( \widetilde{G}_{n,k-2}(v_1,v_2) \simeq H_1 \oplus H_2 \Big) = \prod_{j \in [2]} \mathbb{P} \Big( \widetilde{G}_{\infty,k-2}(r) \simeq H_j \Big) \]
 for any pair of fixed finite, rooted, edge-colored graphs $H_1$ and $H_2$.
\end{proposition}

We prove Proposition~\ref{prop:benjamini_schramm} in the Appendix.
Note that during the proof we also obtain that for any $k \in \mathbb{N}_+$, $d \in \mathbb{N}$, and $\underline{\lambda} \in \mathbb{R}_+^k$, and for any $G_n \sim \mathcal{G}_n([n],\underline{\lambda})$ and $v_1, v_2 \in [n]$, $v_1 \ne v_2$ for all $n \in \mathbb{N}_+$, and for any pair of fixed finite, rooted, edge-colored graphs $F_1$ and $F_2$, we have
\begin{equation} \label{eq:lim_of_Ed_is_1}
 \lim_{n \to \infty} \mathbb{P} \big( \text{$G_{n,d}(v_1, v_2)$ is a forest with two components} \big) = 1
\end{equation}
and
\begin{equation} \label{eq:conv_of_edgecolored_balls_of_radius_d}
 \lim_{n \to \infty} \mathbb{P} \big( G_{n,d}(v_1, v_2) \simeq F_1 \oplus F_2 \big) = \prod_{j \in [2]} \mathbb{P} \big( G_{\infty,d}(r) \simeq F_j \big) \text{.}
\end{equation}

\begin{definition}[Probability of being $i$-avoiding connected to infinity] \label{def:theta_minus_i}
 Let $k \in \mathbb{N}_+$, $i \in [k]$ and let $\underline{\lambda} \in \mathbb{R}^k_+$. Let us denote by $\theta^{\setminus i}$ the survival probability of the branching process with offspring distribution $\mathrm{POI}(\lambda^{\setminus i})$.
\end{definition}

Note that if $G_{\infty} \colonequals G_{\infty}(r) \sim \mathcal{G}_{\infty}(\underline{\lambda})$, then by Claim~\ref{claim:GIr_is_BP},
\[ \theta^{\setminus i} = \mathbb{P} \left( r \stackrel{G_{\infty}^{\setminus i}}{\longleftrightarrow} \infty \right) \]
holds.

\begin{definition}[Supercritical indices] \label{def:set_of_supercrit_indices}
 Let $k \in \mathbb{N}_+$. Given $\underline{\lambda} \in \mathbb{R}_+^k$, we say that $i \in [k]$ is a \emph{supercritical index} if $\lambda^{\setminus i} > 1$, and let us define the set $I_{\underline{\lambda}}$ of \emph{supercritical indices} of $\underline{\lambda}$ by
 \[ I_{\underline{\lambda}} \colonequals \left\{ i \in [k] : \lambda^{\setminus i} > 1 \right\} \text{.} \]
\end{definition}

Note that $\underline{\lambda}$ is fully supercritical (cf.\ Definition~\ref{def:fully_supercrit}) if and only if $I_{\underline{\lambda}} = [k]$ and $\underline{\lambda}$ is fully critical-subcritical if and only if $I_{\underline{\lambda}} = \emptyset$. Also note that $\theta^{\setminus i} = 0$ holds if and only if $i \in [k] \setminus I_{\underline{\lambda}}$.

\begin{definition}[Good events for ECER graphs] \label{def:good_events}
 Let $k,n \in \mathbb{N}_+$ and $\underline{\lambda} \in \mathbb{R}_+^k$. Let $G \sim \mathcal{G}_n(V,\underline{\lambda})$, $v,w \in V$, $v \ne w$, $i \in I_{\underline{\lambda}}$, and let us define the following events.
 
 \begin{align*}
  A_1^{\setminus i} \colonequals & A_1^{\setminus i}(G) \colonequals \left\{ \text{$G^{\setminus i}$ has a unique largest component} \right\} \\[10pt]
  A_2^{\setminus i} \colonequals & A_2^{\setminus i}(G) \\
  \colonequals & \left\{ \text{the non-largest components of $G^{\setminus i}$ have size at most $n^{1/4}$} \right\} \\[10pt]
  A_3(v,w) \colonequals & A_3(v,w,G) \colonequals \Big\{ \text{$\widetilde{G}_{k-2}(v,w)$ is a forest with two components} \Big\} \\[10pt]
  A_4(v,w) \colonequals & A_4(v,w,G) \colonequals \bigg\{ \Big| \, \widetilde{R}^{\le}_{k-2} \big( \{ v,w \}, G \big) \Big| \le n^{1/4} \bigg\} \\[10pt]
  A_5^{\setminus i}(v,w) \colonequals & A_5^{\setminus i}(v,w,G) \colonequals \bigg\{ \Big| \, \widetilde{N}^{\setminus i}_{k-1} \big( \{ v,w \}, G \big) \Big| \le n^{1/4} \bigg\} \\[10pt]
  A_6^{\setminus i}(v,w) \colonequals & A_6^{\setminus i}(v,w,G) \colonequals \Bigg\{ \bigg| \, \mathcal{U}_{\text{max}} \bigg( \Big( \widetilde{G}'_{k-2}(v,w) \Big)^{\setminus i} \bigg) \bigg| \ge n^{3/4} \Bigg\} \\[10pt]
  A_7^{\setminus i}(v,w) \colonequals & A_7^{\setminus i}(v,w,G) \\
  \colonequals & \left\{ \text{$\Big( \widetilde{G}'_{k-2}(v,w) \Big)^{\setminus i}$ has a unique largest component} \right\} \\[10pt]
  A_8^{\setminus i}(v,w) \colonequals & A_8^{\setminus i}(v,w,G) \\
  \colonequals & \bigg\{ \text{the non-largest components of $\Big( \widetilde{G}'_{k-2}(v,w) \Big)^{\setminus i}$} \\ & \hspace{130pt} \text{have size at most $n^{1/4}$} \bigg\} \\[10pt]
  A_9^{\setminus i}(v,w) \colonequals & A_9^{\setminus i}(v,w,G) \colonequals \left\{ \mathcal{U}_{\text{max}} \Bigg( \Big( \widetilde{G}'_{k-2}(v,w) \Big)^{\setminus i} \Bigg) \subseteq \mathcal{U}_{\text{max}} \big( G_n^{\setminus i} \big) \right \} \\[15pt]
  A(v,w) \colonequals & A(v,w,G) \\
  \colonequals & \displaystyle \bigcap_{i \in I_{\underline{\lambda}}} \left( A_1^{\setminus i} \cap A_2^{\setminus i} \right) \cap A_3(v,w) \cap A_4(v,w) \cap \left( \bigcap_{j=5}^9 \bigcap_{i \in I_{\underline{\lambda}}} A_j^{\setminus i}(v,w) \right)
 \end{align*}
\end{definition}

\medskip

\begin{claim}[Implications of local tree structure] \label{claim:A_3_implies}
 Let $k,n \in \mathbb{N}_+$, $\underline{\lambda} \in \mathbb{R}_+^k$ and let $G \sim \mathcal{G}_n(V,\underline{\lambda})$ and $v,w \in V$, $v \ne w$. Then the event $A_3(v,w)$ is $\mathcal{F}_{v,w,k-2}$-measurable, and $A_3(v,w)$ implies that for any $i_1, i_2 \in [k]$, $i_1 \ne i_2$, the sets of vertices $\widetilde{N}_{k-1}^{\setminus i_1}(v)$, $\widetilde{N}_{k-1}^{\setminus i_2}(v)$, $\widetilde{N}_{k-1}^{\setminus i_1}(w)$, $\widetilde{N}_{k-1}^{\setminus i_2}(w)$ are pairwise disjoint.
\end{claim}

\begin{proposition}[Typical behaviour of ECER graphs] \label{prop:good_events}
 Let $k \in \mathbb{N_+}$, $\underline{\lambda} \in \mathbb{R}_+^k$, let $G_n \sim \mathcal{G}_n([n],\underline{\lambda})$ and $v,w \in [n]$, $v \ne w$ for all $n \in \mathbb{N}_+$. If Assumption~\ref{asm:simplifying_assumption} holds, then
 \[ \lim_{n \to \infty} \mathbb{P} \big( A_3(v,w,G_n) \big) = 1 \quad \quad \text{ and } \quad \quad \lim_{n \to \infty} \mathbb{P} \big( A(v,w,G_n) \big) = 1 \text{.} \]
\end{proposition}
\begin{proof} 
 Clearly, it is enough to show that for each $j \in \{ 5, 6, \ldots, 9 \}$ and $i \in I_{\underline{\lambda}}$, the probabilities of the events $A_1^{\setminus i}$, $A_2^{\setminus i}$, $A_3(v,w)$, $A_4(v,w)$, and $A_j^{\setminus i}$ goes to 1 as $n \to \infty$.
 
 Let $i \in I_{\underline{\lambda}}$ be an arbitrary color.
 
 From our Claim~\ref{claim:GI_is_ER} and Theorem~4.8 of~\cite{vdH}, it follows that
 \begin{equation} \label{eq:typical_no1}
  \frac{ \ \Big| \, \mathcal{U}_{\text{max}} \big( G_n^{\setminus i} \big) \Big| \ }{n} \stackrel{\mathbb{P}}{\longrightarrow} \theta^{\setminus i}, \qquad n \to \infty \text{.}
 \end{equation}
 Note that $\theta^{\setminus i}>0$ holds for any $i \in I_{\underline{\lambda}}$ and thus,
 \begin{equation} \label{eq:typical_no2}
  \lim_{n \to \infty} \mathbb{P} \bigg( \Big| \, \mathcal{U}_{\text{max}} \big( G_n^{\setminus i} \big) \Big| \ge n^{3/4} \bigg) = 1 \text{.}
 \end{equation}
 From~\eqref{eq:typical_no1} and Lemma~4.14 of~\cite{vdH}, it follows that
 \begin{equation} \label{eq:A1}
  \lim_{n \to \infty} \mathbb{P} \big( A_1^{\setminus i} \big) = 1
 \end{equation}
 and
 \begin{equation} \label{eq:A2}
  \lim_{n \to \infty} \mathbb{P} \big( A_2^{\setminus i} \big) \\
  = 1 \text{.}
 \end{equation}
 
 By Proposition~\ref{prop:benjamini_schramm}, if Assumption~\ref{asm:simplifying_assumption} holds, then 
 \[ \lim_{n \to \infty} \mathbb{P} \big( A_3(v,w,G_n) \big) = 1 \text{.} \]
 
 Now we prove that $\lim_{n \to \infty} \mathbb{P} \big( A_4(v,w,G_n) \big) = 1$. Let $I \subseteq [k]$ with $|I|=k-2$. Then by Claim~\ref{claim:GI_is_ER}, we have $G_n^I \sim \mathcal{G}_n([n],\lambda_I)$. By Assumption~\ref{asm:simplifying_assumption}, we can apply Theorem 4.2 of~\cite{vdH} to get that $\big| \mathcal{C}(v, G_n^I) \big|$ is stochastically dominated by the total number of offspring in a subcritical Poisson branching process with mean number of offspring at most $\lambda_I$. Therefore
 \[ \mathbb{E} \bigg( \Big| \mathcal{C} \big( v, G_n^I \big) \Big| \bigg) \le \frac{1}{1-\lambda_I} \text{.} \]
 Then Claim~\ref{claim:union_of_bags}, the union bound, and Markov's inequality imply that 
 \begin{equation} \label{eq:A4}
  \lim_{n \to \infty} \mathbb{P} \big( A_4(v,w,G_n) \big) = 1 \text{.}
 \end{equation}
 
 The probability of $A_5^{\setminus i}(v,w,G_n)$ is estimated similarly: clearly,
 \[ \widetilde{N}^{\setminus i}_{k-1}(v,G_n) \subseteq N \Big( \widetilde{R}^{\le}_{k-2}(v,G_n), G_n \Big) \text{,} \]
 thus again the expectation of $\big| \widetilde{N}^{\setminus i}_{k-1}(v,G_n) \big|$ can be bounded by a constant (that does not depend on $n$) and it follows from the union bound and Markov's inequality that 
 \[ \lim_{n \to \infty} \mathbb{P} \big( A_5^{\setminus i}(v,w,G_n) \big) = 1 \text{.} \]
 
 By~\eqref{eq:A4}, we obtain 
 \[ \frac{ \ \bigg| V \bigg(\Big( \widetilde{G}'_{n,k-2}(v,w) \Big)^{\setminus i} \bigg) \bigg| \ }{n} = \frac{ \ \Big| V \left( \widetilde{G}'_{n,k-2}(v,w) \right) \Big| \ }{n} \stackrel{\mathbb{P}}{\longrightarrow} 1, \qquad n \to \infty \text{.} \]
 From this, our Lemma~\ref{lemma:after_conditioning_still_ECER}, and Theorem 4.8 of~\cite{vdH}, it follows
 \begin{equation} \label{eq:law_of_large_numbers_giant}
  \!\! \frac{ \, \bigg| \, \mathcal{U}_{\text{max}} \bigg( \! \Big( \widetilde{G}'_{n,k-2}(v,w) \Big)^{ \!\! \setminus i} \bigg) \bigg| \, }{\bigg| \, V \bigg( \Big( \widetilde{G}'_{n,k-2}(v,w) \Big)^{\setminus i} \bigg) \bigg|} \! \cdot \! \frac{ \, \bigg| \, V \bigg( \!  \Big( \widetilde{G}'_{n,k-2}(v,w) \Big)^{ \!\! \setminus i} \bigg) \bigg| \, }{n} \stackrel{\mathbb{P}}{\longrightarrow} \theta^{\setminus i}, \qquad n \to \infty \text{.}
 \end{equation}
 Thus,
 \begin{equation} \label{eq:A6}
  \lim_{n \to \infty} \mathbb{P} \big( A_6^{\setminus i}(v,w,G_n) \big) = 1 \text{.}
 \end{equation}
 From~\eqref{eq:law_of_large_numbers_giant} and Lemma~4.14 of~\cite{vdH}, it follows that
 \begin{equation} \label{eq:A7}
  \lim_{n \to \infty} \mathbb{P} \big( A_7^{\setminus i}(v,w,G_n) \big) = 1 \text{.}
 \end{equation}
 and
 \[ \lim_{n \to \infty} \mathbb{P} \big( A_8^{\setminus i}(v,w,G_n) \big) = 1 \text{.} \]
 The above observations (namely,~\eqref{eq:typical_no2},~\eqref{eq:A1},~\eqref{eq:A2}, and~\eqref{eq:A6},~\eqref{eq:A7}) imply
 \[ \lim_{n \to \infty} \mathbb{P} \big( A_9^{\setminus i}(v,w,G_n) \big) = 1 \text{.} \]
\end{proof}

\subsection{Types} \label{subsection:types}

In Section~\ref{subsection:types}, we define the color-avoiding type I vectors of the vertices, first in ECBP trees, then in ECER graphs. The main result of this section, Theorem~\ref{thm:conv_of_typeI}, connects the notion of these type I vectors. 

\begin{definition}[Color-avoiding type I vectors in ECBP trees] \label{def:typeI_in_ECBP}
 Let $k \in \mathbb{N}_+$, $\underline{\lambda} \in \mathbb{R}_+^k$, and let $G_{\infty} \colonequals G_{\infty}(r) \sim \mathcal{G}_{\infty}(\underline{\lambda})$. The \emph{(color-avoiding) type I vector} of a vertex $v$ in the ECBP tree $G_{\infty}$ is
 \[ \underline{t}^*(v) \colonequals \underline{t}^*(v, G_{\infty}) \colonequals \bigg( \mathds{1} \Big[ v \stackrel{G_{\infty}^{\setminus i}}{\longleftrightarrow} \infty \Big] \bigg)_{i \in I_{\underline{\lambda}}} \text{.} \]
\end{definition}

In words: for any $i \in I_{\underline{\lambda}}$, the $i$-th coordinate of the type I vector of a vertex $v$ in an ECBP tree is 1 if $v$ is $i$-avoiding connected to infinity through its descendants, otherwise it is 0. The reason why we only use the indices in $I_{\underline{\lambda}}$ is that we expect the branching process $G_{\infty}$ to survive without the edges of color $i$ if and only if $i \in I_{\underline{\lambda}}$.

\begin{definition}[Probability of type I vectors in ECBP trees] \label{def:pstar_gamma}
 Let $k \in \mathbb{N}_+$, $\underline{\lambda} \in \mathbb{R}_+^k$, and let $G_{\infty}(r) \sim \mathcal{G}_{\infty}(\underline{\lambda})$. Given any $\underline{\gamma} \in \left\{ 0,1 \right\}^{I_{\underline{\lambda}}}$, let us denote
 \[ p^*(\underline{\gamma}) \colonequals p^*(\underline{\gamma},\underline{\lambda}) \colonequals \mathbb{P} \big( \underline{t}^*(r)=\underline{\gamma} \big) \text{.} \]
\end{definition}

\begin{definition}[Color-avoiding type I vectors in ECER graphs] \label{def:typeI_in_ECER}
 Let $k,n \in \mathbb{N}_+$, $\underline{\lambda} \in \mathbb{R}_+^k$, and let $G \sim \mathcal{G}_n (V, \underline{\lambda})$. The \emph{(color-avoiding) type I vector} of a vertex $v$ in the ECER graph $G$ is
 \[ \underline{t}(v) \colonequals \underline{t}(v,G) \colonequals \bigg( \mathds{1} \left[ v \in \mathcal{U}_{\text{max}} \big( G^{\setminus i} \big) \right] \bigg)_{i \in I_{\underline{\lambda}}} \text{.} \]
\end{definition}

In words: for any $i \in I_{\underline{\lambda}}$, the $i$-th coordinate of the type of a vertex $v$ in an edge-colored finite graph is 1 if $v$ is in one of the largest $i$-avoiding components, otherwise it is 0. Similarly as before, the reason why we only use the indices in $I_{\underline{\lambda}}$ is that we expect $\mathcal{U}_{\text{max}}\big(G^{\setminus i}\big)$ to be a giant component if and only if $i \in I_{\underline{\lambda}}$, so the sequence of indicators $\mathds{1} \left[ v \in \mathcal{U}_{\text{max}} \big(G^{\setminus i}\big) \right]$ does not converge to zero in probability as $n \to \infty$ if and only if $i \in I_{\underline{\lambda}}$.

We are now ready to state the main result of Section~\ref{subsection:types}.

\begin{theorem}[Convergence of the empirical density of type I vectors in ECER graphs] \label{thm:conv_of_typeI}
 Let $k \in \mathbb{N}_+$, $\underline{\lambda} \in \mathbb{R}_+^k$, $\underline{\gamma} \in \{ 0,1 \}^{I_{\underline{\lambda}}}$, and let $G_n \sim \mathcal{G}_n([n], \underline{\lambda})$ for all $n \in \mathbb{N}_+$. If Assumption~\ref{asm:simplifying_assumption} holds, then
 \[ \frac{1}{n} \sum_{v \in [n]} \mathds{1} \big[ \, \underline{t}(v, G_n) =\underline{\gamma} \, \big] \; \stackrel{\mathbb{P}}{\longrightarrow} \; p^*(\underline{\gamma}), \qquad n \to \infty \text{.} \]
\end{theorem}

The rest of Section~\ref{subsection:types} is devoted to the proof of Theorem~\ref{thm:conv_of_typeI}. After some further definitions and observations, the key idea of the proof is stated in Lemma~\ref{lemma:connection_between_typeI_and_typeII_in_ECER}.

\begin{definition}[The known data for ECBP trees] \label{def:data_for_ECBP}
 Let $k \in \mathbb{N}_+$, $\underline{\lambda} \in \mathbb{R}_+^k$, $\underline{\gamma} \in \left\{ 0,1 \right\}^{I_{\underline{\lambda}}}$, $\underline{\beta} \in \mathbb{N}^k$, $m \in \mathbb{N}_+$, and let $G_{\infty} \colonequals G_{\infty}(r) \sim \mathcal{G}_{\infty}(\underline{\lambda})$. Let us define the event
 \[ D^* \big( \underline{\beta}, m \big) \colonequals D^* \big( \underline{\beta}, m, G_{\infty} \big) \colonequals \left\{ \, \underline{b}(r) = \underline{\beta}, \, \rho(r) = m \right\} \text{.} \]
\end{definition}

\begin{definition}[Probability of type I vectors in ECBP trees given some data] \label{def:pstar_gamma_cond_Dstar}
 Let $k \in \mathbb{N}_+$, $\underline{\lambda} \in \mathbb{R}_+^k$, $\underline{\gamma} \in \left\{ 0,1 \right\}^{I_{\underline{\lambda}}}$, $\underline{\beta} \in \mathbb{N}^k$, $m \in \mathbb{N}_+$, and let $G_{\infty} \colonequals G_{\infty}(r) \sim \mathcal{G}_{\infty}(\underline{\lambda})$. Let us define
 \[ p^* \Big( \underline{\gamma} \Bigm| D^* \big( \underline{\beta}, m \big) \Big) \colonequals \mathbb{P} \Big( \underline{t}^*(r) = \underline{\gamma} \Bigm| D^*(\underline{\beta}, m) \Big) \text{.} \]
\end{definition}

Note that by the following claim, $p^* \big( \underline{\gamma} \bigm| D^* (\underline{\beta}, m) \big)$ actually does not depend on $m$.

\begin{proposition}[Probability of type I vectors in ECBP trees given some data] \label{prop:pstar_gamma_cond_Dstar_expressed_as_a_product}
 Let $k \in \mathbb{N}_+$, $\underline{\lambda} \in \mathbb{R}_+^k$, $\underline{\gamma} \in \left\{ 0,1 \right\}^{I_{\underline{\lambda}}}$, $\underline{\beta} \in \mathbb{N}^k$, $m \in \mathbb{N}_+$, and let $G_{\infty} \colonequals G_{\infty}(r) \sim \mathcal{G}_{\infty}(\underline{\lambda})$. Then
 \[ p^* \Big( \underline{\gamma} \Bigm| D^* \big( \underline{\beta}, m \big) \Big) = g_{\underline{\gamma}, \underline{\beta}} \Big( \big( \theta^{\setminus i} \big)_{i \in I_{\underline{\lambda}}} \, \Big) \text{,} \]
 where
 \begin{multline*}
  g_{\underline{\gamma}, \underline{\beta}}: ~ [0,1]^{I_{\underline{\lambda}}} \to \mathbb{R} \\
  (x_i)_{i \in I_{\underline{\lambda}} } \mapsto \prod_{i \in I_{\underline{\lambda}}} \Big( \big( 1-(1-x_i)^{\beta_i} \big) \gamma_i + (1-x_i)^{\beta_i} (1-\gamma_i) \Big) \text{.}
 \end{multline*}
\end{proposition}
\begin{proof}
 It is not difficult to see that for any $i \in I_{\underline{\lambda}}$, the root $r$ is $i$-avoiding connected to infinity if and only if at least one vertex in $\widetilde{N}_{k-1}^{\setminus i}(r)$ is $i$-avoiding connected to infinity. Note that conditional on $D^* \big( \beta, m \big)$, the events 
 \[ \bigcup_{v \in \widetilde{N}_{k-1}^{\setminus i} (r)} \Big\{v \stackrel{G_{\infty}^{\setminus i}}{\longleftrightarrow} \infty \Big\} \]
 for all $i \in I_{\underline{\lambda}}$ are conditionally independent since the sets $\widetilde{N}_{k-1}^{\setminus i} (r)$ for all $i \in I_{\underline{\lambda}}$ are disjoint. Thus, by the definition of $\theta^{\setminus i}$ for any $i \in I_{\underline{\lambda}}$, we obtain
 \[ \mathbb{P} \Big( \underline{t}^*(r) = \underline{\gamma} \bigm| D^* \big( \beta, m \big) \Big) = g_{\underline{\gamma}, \underline{\beta}} \Big( \big( \theta^{\setminus i} \big)_{i \in I_{\underline{\lambda}}} \, \Big) \text{.} \]
\end{proof}

\begin{claim}[Total law of probability for type I vectors in ECBP trees] \label{claim:total_law_of_probab_for_typeI_in_ECBP}
 Let $k \in \mathbb{N}_+$, $\underline{\lambda} \in \mathbb{R}_+^k$, $\underline{\gamma}_1, \underline{\gamma}_2 \in \left\{ 0,1 \right\}^{I_{\underline{\lambda}}}$, and let $G_{\infty}(r) \sim \mathcal{G}_{\infty}(\underline{\lambda})$. If Assumption~\ref{asm:simplifying_assumption} holds, then
 \[ p^*(\underline{\gamma}_1) \cdot p^*(\underline{\gamma}_2) = \sum_{\substack{\underline{\beta}_1, \underline{\beta}_2 \in \mathbb{N}^k, \\ m_1, m_2 \in \mathbb{N}_+}} \prod_{j \in [2]} p^* \Big( \underline{\gamma}_j \Bigm| D^* \big( \underline{\beta}_j, m_j \big) \Big) \cdot \mathbb{P} \left( D^* \big( \underline{\beta}_j, m_j \big) \right) \text{.} \]
\end{claim}
\begin{proof}
 It follows from Claim~\ref{claim:finiteness_of_rhor} and the total law of probability.
\end{proof}

Now we turn to ECER graphs. Note that the notion of color-avoiding type I vectors in ECER graphs was already introduced in Definition~\ref{def:typeI_in_ECER}.

\begin{definition}[Color-avoiding type II events in ECER graphs] \label{def:typeII_in_ECER}
 Let $k,n \in \mathbb{N}_+$, $\underline{\lambda} \in \mathbb{R}_+^k$, $i \in I_{\underline{\lambda}}$ and $\underline{\gamma} \in \left\{ 0,1 \right\}^{I_{\underline{\lambda}}}$ and let $G \sim \mathcal{G}_n(V,\underline{\lambda})$, and $v, w \in V$, $v \ne w$. Let us define the event
 \[ T^{\setminus i}(v,w) \colonequals T^{\setminus i}(v,w,G) \colonequals \Bigg\{ \, \widetilde{N}_{k-1}^{\setminus i}(v,G) \cap \mathcal{U}_{\text{max}} \bigg( \Big( \widetilde{G}'_{k-2}(v,w) \Big)^{\setminus i} \bigg) \neq \emptyset \, \Bigg\} \text{,} \]
 and the \emph{(color-avoiding) type II event}
 \[ T \big( v, w, \underline{\gamma} \big) \colonequals T \big( v, w, \underline{\gamma}, G \big) \colonequals \left( \bigcap_{\substack{i \in I_{\underline{\lambda}}: \\ \gamma_i=1}} T^{\setminus i}(v,w) \right) \cap \left( \bigcap_{\substack{i \in I_{\underline{\lambda}}: \\ \gamma_i=0}} \left( T^{\setminus i}(v,w) \right)^c \right) \text{.} \]
\end{definition}

 The idea is that if the event $A(v,w)$ occurs -- and if Assumption~\ref{asm:simplifying_assumption} holds, then it occurs with high probability by Proposition~\ref{prop:good_events} --, then by our next lemma we can replace the complicated event $\big\{ v \in \mathcal{U}_{\text{max}} \big( G^{\setminus i} \big) \big\}$ for type I vectors with the simpler type II subevent $T^{\setminus i}(v,w)$ for any $i \in I_{\underline{\lambda}}$. The reason why $T^{\setminus i}(v,w)$ is simpler to handle is that $\widetilde{N}_{k-1}^{\setminus i}(v)$ is $\mathcal{F}_{v,w,k-2}$-measurable, and $\big( \widetilde{G}'_{k-2}(v,w) \big)^{\setminus i}$ is an ECER graph given the sigma-algebra $\mathcal{F}_{v,w,k-2}$ by Lemma~\ref{lemma:after_conditioning_still_ECER} and Claim~\ref{claim:GI_is_ER}.

\begin{lemma}[Types I and II agree with high probability in ECER graphs] \label{lemma:connection_between_typeI_and_typeII_in_ECER}
 Let $k,n \in \mathbb{N}_+$, $\underline{\lambda} \in \mathbb{R}_+^k$, $\underline{\gamma} \in \left\{ 0,1 \right\}^{I_{\underline{\lambda}}}$, and let $G \sim \mathcal{G}_n(V,\underline{\lambda})$, and $v, w \in V$, $v \ne w$. If $n$ is large enough and the event $A(v,w)$ occurs, then
 \[ \big\{ \underline{t}(v) = \underline{\gamma} \big\} = T \big( v,w,\underline{\gamma} \big) \text{.} \]
\end{lemma}
\begin{proof}
 Throughout the proof, we assume that the event $A(v,w,G)$ occurs. Let $i \in I(\underline{\lambda})$ be arbitrary. First, note that $\big\{ \big( \underline{t}(v,G) \big)_i = 1 \big\} = \big\{ v \in \mathcal{U}_{\text{max}} \big( G^{\setminus i} \big) \big\}$.
 
 It is enough to show that the events $T^{\setminus i}(v,w,G)$ and $\big\{ v \in \mathcal{U}_{\text{max}} \big( G^{\setminus i} \big) \big\}$ imply one another.
 
 Assume that the event $T^{\setminus i}(v,w,G)$ holds, i.e., there exists a vertex $u \in V$ such that
 \[ u \in \widetilde{N}_{k-1}^{\setminus i}(v,G) \cap \mathcal{U}_{\text{max}} \bigg( \Big( \widetilde{G}'_{k-2}(v,w) \Big)^{\setminus i} \bigg) \text{.} \]
 Then the event $A(v,w,G)$ -- in particular $A_9^{\setminus i}(v,w,G)$ -- implies $u \in \mathcal{U}_{\text{max}} \big( G^{\setminus i} \big)$. By the choice of $u$ and the definition of $\widetilde{N}_{k-1}^{\setminus i}(v,G)$, $u$ and $v$ are $i$-avoiding connected, which implies $v \in \mathcal{U}_{\text{max}} \big( G^{\setminus i} \big)$.

 Now assume that the event $T^{\setminus i}(v,w,G)$ does not hold. First, we show that
 \[ \mathcal{C}(v, G^{\setminus i}) \subseteq \widetilde{R}^{\le}_{k-2} \big( \{ v,w \}, G \big) \cup \left( \bigcup_{u \in \widetilde{N}_{k-1}^{\setminus i}(v,G)} \mathcal{C} \left( u, \Big( \widetilde{G}'_{k-2}(v,w) \Big)^{\setminus i} \right) \right) \text{.} \]
 Let $v' \in \mathcal{C}(v, G^{\setminus i})$. If $v' \notin \widetilde{R}^{\le}_{k-2} \big( \{ v,w \}, G \big)$, then any $i$-avoiding $v$-$v'$ path (which must exist) contains a vertex $u \in \widetilde{N}_{k-1}^{\setminus i}(v,G)$, so $v' \in \mathcal{C} \Big( u, \big( \widetilde{G}'_{k-2}(v,w) \big)^{\setminus i} \Big)$.
 
 From the above, it follows that
 \[ \big| \mathcal{C}(v, G^{\setminus i}) \big| \le \left| \widetilde{R}^{\le}_{k-2} \big( \{ v,w \}, G \big) \right| + \sum_{u \in \widetilde{N}_{k-1}^{\setminus i}(v,G)} \Bigg| \mathcal{C} \bigg( u,\Big( \widetilde{G}'_{k-2}(v,w) \Big)^{\setminus i} \bigg) \Bigg| \text{.} \]
 Note that the event $\big( T^{\setminus i}(v,w,G) \big)^c$ implies that if $u \in \widetilde{N}_{k-1}^{\setminus i}(v,G)$, then
 \[ \mathcal{C} \left( u, \Big( \widetilde{G}'_{k-2}(v,w) \Big)^{\setminus i} \right) \cap \mathcal{U}_{\text{max}} \bigg( \Big( \widetilde{G}'_{k-2}(v,w) \Big)^{\setminus i} \bigg) = \emptyset \text{.} \]
 Also note that the event $A(v,w,G)$ -- in particular $A_4(v,w,G)$, $A_5^{\setminus i}(v,w,G)$, and $A_8^{\setminus i}(v,w,G)$ -- implies
 \[ \big| \mathcal{C}(v, G^{\setminus i}) \big| \le n^{1/4} + n^{1/4} \cdot n^{1/4} \text{.} \]
 Therefore, if $n$ is large enough, then $\big| \mathcal{C}(v, G^{\setminus i}) \big| < n^{3/4}$ holds, so the event $A(v,w,G)$ -- in particular $A_1^{\setminus i}$, $A_7^{\setminus i}(v,w,G)$, and $A_9^{\setminus i}(v,w,G)$ -- implies
 \[ \mathcal{C}(v,G) \cap \mathcal{U}_{\text{max}} \big( G^{\setminus i} \big) = \emptyset \text{,} \]
 from which $v \notin \mathcal{U}_{\text{max}}\big(G^{\setminus i} \big)$ follows.
\end{proof}

\begin{definition}[The known data for ECER graphs] \label{def:data_for_ECER}
 Let $k,n \in \mathbb{N}_+$, $\underline{\lambda} \in \mathbb{R}_+^k$, $\underline{\beta}_1, \underline{\beta}_2 \in \mathbb{N}^k$, $m_1, m_2 \in \mathbb{N}_+$, and let $G \sim \mathcal{G}_n(V,\underline{\lambda})$, and $v_1, v_2 \in V$, $v_1 \ne v_2$. Let us define the event
 \begin{multline*}
  D \Big( \big( v_j, \underline{\beta}_j, m_j \big)_{j \in [2]} \Big) \colonequals D \Big( \big( v_j, \underline{\beta}_j, m_j \big)_{j \in [2]}, G \Big) \\
  \colonequals A_3(v_1,v_2) \cap \left\{ \underline{b}(v_1) = \underline{\beta}_1, \underline{b}(v_2) = \underline{\beta}_2, \rho(v_1) = m_1, \rho(v_2) = m_2 \right\} \text{.}
 \end{multline*}
\end{definition}

\begin{definition}[Probability of type II events in ECER graphs given some data] \label{def:p_gamma1_gamma2_cond_D}
 Let $k,n \in \mathbb{N}_+$, $\underline{\lambda} \in \mathbb{R}_+^k$, $\underline{\gamma}_1, \underline{\gamma}_2 \in \left\{ 0,1 \right\}^{I_{\underline{\lambda}}}$, $\underline{\beta}_1, \underline{\beta}_2 \in \mathbb{N}^k$, $m_1, m_2 \in \mathbb{N}_+$, and let $G \sim \mathcal{G}_n(V,\underline{\lambda})$, and $v_1, v_2 \in V$, $v_1 \ne v_2$. Let us define
 \begin{multline*}
  p \Big( \underline{\gamma}_1, \underline{\gamma}_2 \Bigm| D \big( (v_j, \underline{\beta}_j, m_j)_{j \in [2]} \big) \Big) \colonequals p \Big( \underline{\gamma}_1, \underline{\gamma}_2, G \Bigm| D \big( (v_j, \underline{\beta}_j, m_j)_{j \in [2]} \big) \Big) \\
  \colonequals \mathbb{P} \Big( T \big( v_1, v_2, \underline{\gamma}_1 \big) \cap T \big( v_2, v_1, \underline{\gamma}_2 \big) \Bigm| D \big( (v_j, \underline{\beta}_j, m_j)_{j \in [2]} \big) \Big) \text{.}
 \end{multline*}
\end{definition}

Note that by Lemma~\ref{lemma:connection_between_typeI_and_typeII_in_ECER}, $p \big( \underline{\gamma}_1, \underline{\gamma}_2, G \bigm| D \big( (v_j, \underline{\beta}_j, m_j)_{j \in [2]} \big) \big)$ can be viewed as an analogue of $p^* \big( \underline{\gamma}_1 \bigm| D^*(\underline{\beta}_1, m_1) \big) \cdot p^* \big( \underline{\gamma}_2 \bigm| D^*(\underline{\beta}_2, m_2) \big)$.

The next claim follows from the total law of probability.

\begin{claim}[Total law of probability for type II events in ECER graphs] \label{claim:total_law_of_probab_for_typeII_in_ECER}
 Let $k,n \in \mathbb{N}_+$, $\underline{\lambda} \in \mathbb{R}_+^k$, $\underline{\gamma}_1, \underline{\gamma}_2 \in \left\{ 0,1 \right\}^{I_{\underline{\lambda}}}$, and let $G \sim \mathcal{G}_n(V,\underline{\lambda})$ and $v_1, v_2 \in V$, $v_1 \ne v_2$. Then
 \begin{multline*}
  \mathbb{P} \left( T \big( v_1, v_2, \underline{\gamma}_1 \big) \cap T \big( v_2, v_1, \underline{\gamma}_2 \big) \cap A_3(v,w) \right) \\
  = \sum_{\substack{ \underline{\beta}_1, \underline{\beta}_2 \in \mathbb{N}^k, \\ m_1, m_2 \in \mathbb{N}_+}} p \Big( \underline{\gamma}_1, \underline{\gamma}_2 \Bigm| D \big( (v_j, \underline{\beta}_j, m_j)_{j \in [2]} \big) \Big) \cdot \mathbb{P} \left( D \big( (v_j, \underline{\beta}_j, m_j)_{j \in [2]} \big) \right) \text{.}
 \end{multline*}
\end{claim}

In Proposition~\ref{prop:connection_between_typeII_and_typeIII_in_ECER} below, we want to give an alternative formula for the probabilities $p \big( \underline{\gamma}_1, \underline{\gamma}_2 \bigm| D \big( (u_j, \underline{\beta}_j, m_j)_{j \in [2]} \big) \big)$ which involves less conditioning. In order to do so, we need the next definition.

\begin{definition}[Color-avoiding type III events in ECER graphs] \label{def:typeIII_in_ECER}
 Let $k \in \mathbb{N}_+$, $\underline{\lambda} \in \mathbb{R}_+^k$ and $\underline{\gamma} \in \left\{ 0,1 \right\}^{I_{\underline{\lambda}}}$. Let $n, m \in \mathbb{N}_+$ and $\underline{\beta} = (\beta_1, \ldots, \beta_k) \in \mathbb{N}^k$ so that
 \[ \sum_{i \in [k]} \beta_i \le n-m \text{.} \]
 Let $G \sim \mathcal{G}_n([n-m], \underline{\lambda})$, and for each $i \in [k]$ and $j \in [\beta_i]$ let us fix the vertices $v_{i,j} \in [n-m]$, and let 
 \[ \underline{v} = (v_{i,j})_{i \in [k], \, j \in [\beta_i]} \text{.} \]
 For any $i \in I_{\underline{\lambda}}$, let us define the event
 \[ \widehat{T}^{\setminus i}(\underline{v}, \underline{\beta}, m) \colonequals \widehat{T}^{\setminus i}(\underline{v}, \underline{\beta}, m, G) \colonequals \left\{ \, \exists j \in [\beta_i] \, : \,  v_{i,j} \in \mathcal{U}_{\text{max}} \big( G^{\setminus i} \big) \, \right\} \text{,} \]
 and the \emph{(color-avoiding) type III event}
 \begin{multline*}
  \widehat{T} \big( \underline{v}, \underline{\gamma}, \underline{\beta}, m \big) \colonequals \widehat{T} \big( \underline{v}, \underline{\gamma}, \underline{\beta}, m, G \big) \\
  \colonequals \left( \bigcap_{\substack{i \in I_{\underline{\lambda}}: \\ \gamma_i=1}} \widehat{T}^{\setminus i}(\underline{v}, \underline{\beta}, m) \right) \cap \left( \bigcap_{\substack{i \in I_{\underline{\lambda}}: \\ \gamma_i=0}} \left( \widehat{T}^{\setminus i}(\underline{v}, \underline{\beta}, m) \right)^c \right) \text{.}
 \end{multline*}
\end{definition}

\begin{proposition}[Types II and III agree in probability in ECER graphs] \label{prop:connection_between_typeII_and_typeIII_in_ECER}
 Let $k \in \mathbb{N}_+$, $\underline{\lambda} \in \mathbb{R}_+^k$ and $\underline{\gamma}_1, \underline{\gamma}_2 \in \left\{ 0,1 \right\}^{I_{\underline{\lambda}}}$. Let $n, m_1, m_2 \in \mathbb{N}_+$ and $\underline{\beta}_1 = (\beta_{1,1}, \ldots, \beta_{1,k}), \underline{\beta}_2 = (\beta_{2,1}, \ldots, \beta_{2,k}) \in \mathbb{N}^k$ so that
 \[ \sum_{i \in [k]} \beta_{1,i} + \sum_{i \in [k]} \beta_{2,i} \le n - (m_1 + m_2) \text{.} \]
 Let $G \sim \mathcal{G}_n([n], \underline{\lambda})$ and $H \sim \mathcal{G}_n \big( [ n - (m_1 + m_2) ], \underline{\lambda} \big)$, let us fix the vertices $u_1, u_2 \in [n]$, $u_1 \ne u_2$ and for each $i \in [k]$, $j_1 \in [\beta_{1,i}]$ and $j_2 \in [\beta_{2,i}]$, let us fix distinct vertices $v_{i,j_1}, w_{i,j_2} \in \big[ n - (m_1 + m_2) \big]$, and let 
 \[ \underline{v} = (v_{i,j_1})_{i \in [k], \, j_1 \in [\beta_{1,i}]}, \qquad \underline{w} = (w_{i,j_2})_{i \in [k], \, j_2 \in [\beta_{2,i}]} \text{.} \]
 Then
 \begin{multline*}
  p \Big( \underline{\gamma}_1, \underline{\gamma}_2 \Bigm| D \big( (u_j, \underline{\beta}_j, m_j)_{j \in [2]} \big) \Big) \\
  = \mathbb{P} \left( \widehat{T} \big( \underline{v}, \underline{\gamma}_1, \underline{\beta}_1, m_1, H \big) \cap \widehat{T} \big( \underline{w}, \underline{\gamma}_2, \underline{\beta}_2, m_2, H \big) \right) \text{.}
 \end{multline*}
\end{proposition}
\begin{proof}
 Let us condition on the sigma-algebra $\mathcal{F}_{u_1, u_2, k-2}(G)$ and assume that the $\mathcal{F}_{u_1, u_2, k-2}(G)$-measurable event $D \big( (u_j, \underline{\beta}_j, m_j)_{j \in [2]} \big)$ occurs. Then by the definition of $A_3(u_1, u_2, G)$, we have
 \[ \Big| \widetilde{R}^{\le}_{k-2} \big( \{ u_1, u_2 \}, G \big) \Big| = \rho(u_1, G) + \rho(u_2, G) = m_1 + m_2. \]
 Thus by Lemma~\ref{lemma:after_conditioning_still_ECER}, given $\mathcal{F}_{u_1, u_2, k-2}(G)$, the conditional distribution of the graph $\widetilde{G}'_{k-2}(u_1, u_2)$ is the same as the distribution of $H$. In addition, our assumption that $A_3(u_1, u_2, G)$ holds implies that the sets of vertices $\widetilde{N}_{k-1}^{\setminus i_1}(u_1, G)$, $\widetilde{N}_{k-1}^{\setminus i_2}(u_2, G)$ for all $i_1, i_2 \in [k]$ are pairwise disjoint (cf.\ Claim~\ref{claim:A_3_implies}). Let us denote
 \[ \widetilde{N}_{k-1}^{\setminus i}(u_1, G) = \Big\{ v'_{ij_1} : j_1 \in \Big[ \big| \big( b(u_1) \big)_i \big| \Big] \Big\} \]
 and
 \[ \widetilde{N}_{k-1}^{\setminus i}(u_2, G) = \Big\{ w'_{ij_2} : j_2 \in \Big[ \big| \big( b(u_2) \big)_i \big| \Big] \Big\} \]
 for all $i \in [k]$, and let us use the notation
 \[ \underline{v}' = (v'_{i,j_1})_{i \in [k], \, j_1 \in \left[\rule{0cm}{8pt}\right. \! \left|\rule{0cm}{6pt}\right. \! \left(\rule{0cm}{6pt}\right. \! b(u_1) \! \left.\rule{0cm}{6pt}\right)_i \! \left.\rule{0cm}{6pt}\right| \! \left.\rule{0cm}{8pt}\right]}, \qquad \underline{w}' = (w'_{i,j_2})_{i \in [k], \, j_2 \in \left[\rule{0cm}{8pt}\right. \! \left|\rule{0cm}{6pt}\right. \! \left(\rule{0cm}{6pt}\right. \! b(u_2) \! \left.\rule{0cm}{6pt}\right)_i \! \left.\rule{0cm}{6pt}\right| \! \left.\rule{0cm}{8pt}\right]} \text{.} \]
 By the definitions of type II and type III events, we have
 \begin{multline*}
  \mathbb{P} \left( T \big( u_1, u_2, \underline{\gamma}_1, G \big) \cap T \big( u_2, u_1, \underline{\gamma}_2, G \big) \middle| \, D \big( v, w, \underline{\beta}_1, \underline{\beta}_2, m_1, m_2 \big) \right) \\
  = \mathbb{P} \left( \widehat{T} \big( \underline{v}', \underline{\gamma}_1, \underline{\beta}_1, m_1, \widetilde{G}'_{k-2}(u_1, u_2) \big) \cap \widehat{T} \big( \underline{w}', \underline{\gamma}_2, \underline{\beta}_2, m_2, \widetilde{G}'_{k-2}(u_1, u_2) \big) \right) \text{.} 
 \end{multline*}
 Thus by the definition of $p \big( \underline{\gamma}_1, \underline{\gamma}_2 \bigm| D \big( (u_j, \underline{\beta}_j, m_j)_{j \in [2]} \big) \big)$, we proved the required statement.
\end{proof}

\begin{lemma}[Data and types: connections between ECBP trees and ECER graphs] \label{lemma:data_type_connection_between_ECBP_and_ECER}
 Let $k \in \mathbb{N}_+$, $\underline{\lambda} \in \mathbb{R}_+^k$, $\underline{\gamma}_1, \underline{\gamma}_2 \in \left\{ 0,1 \right\}^{I_{\underline{\lambda}}}$, $\underline{\beta}_1, \underline{\beta}_2 \in \mathbb{N}^k$, $m_1, m_2 \in \mathbb{N}_+$, and let $G_n \sim \mathcal{G}_n([n], \underline{\lambda})$ and $u_1, u_2 \in [n]$, $u_1 \ne u_2$ for all $n \in \mathbb{N}_+$. If Assumption~\ref{asm:simplifying_assumption} holds, then
 \begin{equation} \label{eq:data_type_connection_no1}
  \lim_{n \to \infty} \mathbb{P} \Big( D \big( (u_j, \underline{\beta}_j, m_j)_{j \in [2]}, G_n \big) \Big) = \mathbb{P} \left( D^* \big( \underline{\beta}_1, m_1 \big) \right) \cdot \mathbb{P} \left( D^* \big( \underline{\beta}_2, m_2 \big) \right)
 \end{equation}
 and
 \begin{multline} \label{eq:data_type_connection_no2}
  \lim_{n \to \infty} p \Big( \underline{\gamma}_1, \underline{\gamma}_2, G_n \Bigm| D \big( (u_j, \underline{\beta}_j, m_j)_{j \in [2]} \big) \Big) \\
  = p^* \Big( \underline{\gamma}_1 \Bigm| D^* \big( \underline{\beta}_1, m_1 \big) \Big) \cdot p^* \Big( \underline{\gamma}_2 \Bigm| D^* \big( \underline{\beta}_2, m_2 \big) \Big) \text{.}
  \end{multline}
 \end{lemma}

\begin{proof}
 Since Assumption~\ref{asm:simplifying_assumption} holds, Propositions~\ref{prop:benjamini_schramm} and~\ref{prop:good_events} imply~\eqref{eq:data_type_connection_no1}.
 
 Now we prove~\eqref{eq:data_type_connection_no2}. By Proposition~\ref{prop:pstar_gamma_cond_Dstar_expressed_as_a_product}, we have
 \begin{multline*}
  p^* \Big( \underline{\gamma}_1 \Bigm| D^* \big( \underline{\beta}_1, m_1 \big) \Big) \cdot p^* \Big( \underline{\gamma}_2 \Bigm| D^* \big( \underline{\beta}_2, m_2 \big) \Big) \\
  = g_{\underline{\gamma}_1, \underline{\beta}_1} \Big( \big( \theta^{\setminus i} \big)_{i \in I_{\underline{\lambda}}} \, \Big) \cdot g_{\underline{\gamma}_2, \underline{\beta}_2} \Big( \big( \theta^{\setminus i} \big)_{i \in I_{\underline{\lambda}}} \, \Big) \text{.}
 \end{multline*}
 Now we reformulate the left-hand side of~\eqref{eq:data_type_connection_no2}.
 
 Let $n \in \mathbb{N}_+$ be large enough such that 
 \[ \sum_{i \in [k]} \beta_{1,i} + \sum_{i \in [k]} \beta_{2,i} \le n - (m_1 + m_2) \text{.} \]
 Let $H_n \sim \mathcal{G}_n \big( [ n - (m_1 + m_2) ], \underline{\lambda} \big)$. For each $i \in [k]$ and $j_1 \in [\beta_{1,i}]$, $j_2 \in [\beta_{2,i}]$, choose the vertices $v_{i,j_1}$ and $w_{i,j_2}$ independently and uniformly at random from $\big[ n-(m_1+m_2) \big]$, and let
 \[ \underline{v} = (v_{i,j_1})_{i \in [k], j_1 \in [\beta_{1,i_1}]}, \quad \underline{w} = (w_{i,j_2})_{i \in [k], j_2 \in [\beta_{2,i}]} \text{.} \]
 By Proposition~\ref{prop:connection_between_typeII_and_typeIII_in_ECER}, we have
 \begin{multline*}
  p \Big( \underline{\gamma}_1, \underline{\gamma}_2, G_n \Bigm| D \big( (u_j, \underline{\beta}_j, m_j)_{j \in [2]} \big) \Big) \\
  = \mathbb{P} \left( \widehat{T} \big( \underline{v}, \underline{\gamma}_1, \underline{\beta}_1, m_1, H_n \big) \cap \widehat{T} \big( \underline{w}, \underline{\gamma}_2, \underline{\beta}_2, m_2, H_n \big) \middle| E(\underline{v}, \underline{w}, H_n) \right) \text{,}
 \end{multline*}
 where recall that
 \[ E(\underline{v}, \underline{w}, H_n) = \{ \text{the vertices in $\underline{v}$ and $\underline{w}$ are all different in $H_n$} \} \text{.} \]
 It is easy to see that
 \[ \lim_{n \to \infty} \mathbb{P} \big( E(\underline{v}, \underline{w}, H_n) \big) = 1 \text{,} \]
 thus
 \begin{multline*}
  \lim_{n \to \infty} \mathbb{P} \left( \widehat{T} \big( \underline{v}, \underline{\gamma}_1, \underline{\beta}_1, m_1, H_n \big) \cap \widehat{T} \big( \underline{w}, \underline{\gamma}_2, \underline{\beta}_2, m_2, H_n \big) \middle| E(\underline{v}, \underline{w}, H_n) \right) \\
  = \lim_{n \to \infty} \mathbb{P} \left( \widehat{T} \big( \underline{v}, \underline{\gamma}_1, \underline{\beta}_1, m_1, H_n \big) \cap \widehat{T} \big( \underline{w}, \underline{\gamma}_2, \underline{\beta}_2, m_2, H_n \big) \right) \text{.}
 \end{multline*}
 Let
 \[ \theta^{\setminus i}_n \colonequals \frac{\Big| \, \mathcal{U}_{\text{max}} \big( H_n^{\setminus i} \big) \Big|}{n-(m_1+m_2)} \]
 for all $i \in I_{\underline{\lambda}}$.
 Since the vertices in $\underline{v}$ and $\underline{w}$ were chosen independently and uniformly at random, it is not difficult to see that
 \begin{multline*}
  \mathbb{P} \left( \widehat{T} \big( \underline{v}, \underline{\gamma}_1, \underline{\beta}_1, m_1, H_n \big) \cap \widehat{T} \big( \underline{w}, \underline{\gamma}_2, \underline{\beta}_2, m_2, H_n \big) \middle| \ \sigma \big( \theta^{\setminus i}_n : i \in [k] \big) \right) \\
  = g_{\underline{\gamma}_1, \underline{\beta}_1} \Big( \big( \theta^{\setminus i}_n \big)_{i \in I_{\underline{\lambda}}} \, \Big) \cdot g_{\underline{\gamma}_2, \underline{\beta}_2} \Big( \big( \theta^{\setminus i}_n \big)_{i \in I_{\underline{\lambda}}} \, \Big) \text{,}
 \end{multline*}
 thus
 \begin{multline*}
  \mathbb{P} \left( \widehat{T} \big( \underline{v}, \underline{\gamma}_1, \underline{\beta}_1, m_1, H_n \big) \cap \widehat{T} \big( \underline{w}, \underline{\gamma}_2, \underline{\beta}_2, m_2, H_n \big) \right) \\
  = \mathbb{E} \left[ g_{\underline{\gamma}_1, \underline{\beta}_1} \Big( \big( \theta^{\setminus i}_n \big)_{i \in I_{\underline{\lambda}}} \, \Big) \cdot g_{\underline{\gamma}_2, \underline{\beta}_2} \Big( \big( \theta^{\setminus i}_n \big)_{i \in I_{\underline{\lambda}}} \, \Big) \right] \text{.}
 \end{multline*}
 
 Similarly to~\eqref{eq:law_of_large_numbers_giant}, one can show that
 \[ \theta^{\setminus i}_n \stackrel{\mathbb{P}}{\longrightarrow} \theta^{\setminus i}, \qquad n \to \infty \text{.} \]
 Since $g$ is continuous,
 \begin{multline*} 
  g_{\underline{\gamma}_1, \underline{\beta}_1} \Big( \big( \theta^{\setminus i}_n \big)_{i \in I_{\underline{\lambda}}} \, \Big) \cdot g_{\underline{\gamma}_2, \underline{\beta}_2} \Big( \big( \theta^{\setminus i}_n \big)_{i \in I_{\underline{\lambda}}} \, \Big) \\
  \stackrel{\mathbb{P}}{\longrightarrow} g_{\underline{\gamma}_1, \underline{\beta}_1} \Big( \big( \theta^{\setminus i} \big)_{i \in I_{\underline{\lambda}}} \, \Big) \cdot g_{\underline{\gamma}_2, \underline{\beta}_2} \Big( \big( \theta^{\setminus i} \big)_{i \in I_{\underline{\lambda}}} \, \Big), \qquad n \to \infty \text{,}
 \end{multline*}
 and since $g$ is bounded, we obtain
 \begin{multline*}
  \lim_{n \to \infty} \mathbb{E} \left[ g_{\underline{\gamma}_1, \underline{\beta}_1} \Big( \big( \theta^{\setminus i}_n \big)_{i \in I_{\underline{\lambda}}} \, \Big) \cdot g_{\underline{\gamma}_2, \underline{\beta}_2} \Big( \big( \theta^{\setminus i}_n \big)_{i \in I_{\underline{\lambda}}} \, \Big) \right] \\
  = g_{\underline{\gamma}_1, \underline{\beta}_1} \Big( \big( \theta^{\setminus i} \big)_{i \in I_{\underline{\lambda}}} \, \Big) \cdot g_{\underline{\gamma}_2, \underline{\beta}_2} \Big( \big( \theta^{\setminus i} \big)_{i \in I_{\underline{\lambda}}} \, \Big) \text{,}
 \end{multline*}
 thus the statement of the lemma follows.
\end{proof}

We are now ready to prove the main result of Section~\ref{subsection:types}.

\begin{proof}[Proof of Theorem~\ref{thm:conv_of_typeI}]
 Clearly, it suffices to show that
 \[ \lim_{n \to \infty} \mathbb{E} \left( \frac{1}{n} \sum_{v \in [n] } \mathds{1} \big[ \, \underline{t}(v, G_n) = \underline{\gamma} \, \big] \right) = p^*(\underline{\gamma}) \]
 and
 \[ \lim_{n \to \infty} \mathrm{Var} \left( \frac{1}{n} \sum_{v \in [n]} \mathds{1} \big[ \, \underline{t}(v, G_n) = \underline{\gamma} \, \big] \right) = 0 \text{.} \]
 
 In order to prove this, it is enough to show that for any $v_1, v_2 \in [n]$, $v_1 \ne v_2$ and for any $\underline{\gamma}_1, \underline{\gamma}_2 \in \{0,1\}^{I_{\underline{\lambda}}}$, we have
 \begin{equation} \label{eq:typeI_of_one_vertex_conv_in_distr}
  \lim_{n \to \infty} \mathbb{P} \left( \underline{t}(v_1, G_n) = \underline{\gamma}_1 \right) = p^*(\underline{\gamma}_1)
 \end{equation}
 and
 \begin{equation} \label{eq:typeI_of_two_vertices_conv_in_distr}
  \lim_{n \to \infty} \mathbb{P} \left( \underline{t}(v_1, G_n) = \underline{\gamma}_1, \underline{t}(v_2, G_n) = \underline{\gamma}_2 \right) = p^* (\underline{\gamma}_1) \cdot p^* (\underline{\gamma}_2)
 \end{equation}
 Note that~\eqref{eq:typeI_of_one_vertex_conv_in_distr} follows from~\eqref{eq:typeI_of_two_vertices_conv_in_distr} by the total law of probability. So it is enough to prove~\eqref{eq:typeI_of_two_vertices_conv_in_distr}.
 
 Let us denote
 \[ q^* \Big( \big( \underline{\gamma}_j, \underline{\beta}_j, m_j \big)_{j \in [2]} \Big) \colonequals \prod_{j \in [2]} p^* \Big( \underline{\gamma}_j \Bigm| D^* \big( \underline{\beta}_j, m_j \big) \Big) \cdot \mathbb{P} \left( D^* \big( \underline{\beta}_j, m_j \big) \right) \text{.} \]
 
 By Claim~\ref{claim:total_law_of_probab_for_typeI_in_ECBP}, we have
 \begin{equation} \label{eq:for_RHS}
  p^* (\underline{\gamma}_1) \cdot p^* (\underline{\gamma}_2) = \sum_{\substack{\underline{\beta}_1, \underline{\beta}_2 \in \mathbb{N}^k, \\ m_1, m_2 \in \mathbb{N}_+}} \, q^* \Big( \big( \underline{\gamma}_j, \underline{\beta}_j, m_j \big)_{j \in [2]} \Big) \text{.}
 \end{equation}
 
 Now we reformulate the left-hand side of~\eqref{eq:typeI_of_two_vertices_conv_in_distr}. Let us denote
 \begin{multline*}
  q_n \Big( \big( \underline{\gamma}_j, \underline{\beta}_j, m_j \big)_{j \in [2]} \Big) \\
  \colonequals p \Big( \underline{\gamma}_1, \underline{\gamma}_2, G_n \Bigm| D \big( (v_j, \underline{\beta}_j, m_j)_{j \in [2]} \big) \Big) \cdot \mathbb{P} \left( D \big( (v_j, \underline{\beta}_j, m_j)_{j \in [2]}, G_n \big) \right) \text{.}
 \end{multline*}
 By Proposition~\ref{prop:good_events}, Lemma~\ref{lemma:connection_between_typeI_and_typeII_in_ECER}, Proposition~\ref{prop:good_events}, and Claim~\ref{claim:total_law_of_probab_for_typeII_in_ECER}, respectively,
 \begin{multline} \label{eq:for_LHS}
  \lim_{n \to \infty} \mathbb{P} \left( \underline{t}(v_1, G_n) = \underline{\gamma}_1, \underline{t}(v_2, G_n) = \underline{\gamma}_2 \right) \\
  = \lim_{n \to \infty} \mathbb{P} \left( \underline{t}(v_1, G_n) = \underline{\gamma}_1, \underline{t}(v_2, G_n) = \underline{\gamma}_2, A(v_1, v_2, G_n) \right) \\
  = \lim_{n \to \infty} \mathbb{P} \left( T \big( v_1, v_2, \underline{\gamma}_1, G_n \big) \cap T \big( v_2, v_1, \underline{\gamma}_2, G_n \big) \cap A(v_1, v_2, G_n) \right) \\
  = \lim_{n \to \infty} \mathbb{P} \left( T \big( v_1, v_2, \underline{\gamma}_1, G_n \big) \cap T \big( v_2, v_1, \underline{\gamma}_2, G_n \big) \cap A_3(v_1, v_2, G_n) \right) \\
  = \lim_{n \to \infty} \sum_{\substack{\underline{\beta}_1, \underline{\beta}_2 \in \mathbb{N}^k, \\ m_1, m_2 \in \mathbb{N}_+}} q_n \Big( \big( \underline{\gamma}_j, \underline{\beta}_j, m_j \big)_{j \in [2]} \Big) \text{.}
 \end{multline}
 
 By Lemma~\ref{lemma:data_type_connection_between_ECBP_and_ECER}, we have
 \[ \lim_{n \to \infty} q_n \Big( \big( \underline{\gamma}_j, \underline{\beta}_j, m_j \big)_{j \in [2]} \Big) = q^* \Big( \big( \underline{\gamma}_j, \underline{\beta}_j, m_j \big)_{j \in [2]} \Big) \text{,} \]
 and clearly,
 \[ \lim_{n \to \infty} \sum_{\substack{\underline{\gamma}_1, \underline{\gamma}_2 \in \{0,1\}^{I_{\underline{\lambda}}}, \\ \underline{\beta}_1, \underline{\beta}_2 \in \mathbb{N}^k, \\ m_1, m_2 \in \mathbb{N}_+}} q_n \Big( \big( \underline{\gamma}_j, \underline{\beta}_j, m_j \big)_{j \in [2]} \Big) = \lim_{n \to \infty} \mathbb{P} \big( A_3(v_1, v_2, G_n) \big) = 1 \]
 and
 \[ \sum_{\substack{\underline{\gamma}_1, \underline{\gamma}_2 \in \{0,1\}^{I_{\underline{\lambda}}}, \\ \underline{\beta}_1, \underline{\beta}_2 \in \mathbb{N}^k, \\ m_1, m_2 \in \mathbb{N}_+}} q^* \Big( \big( \underline{\gamma}_j, \underline{\beta}_j, m_j \big)_{j \in [2]} \Big) = 1 \text{.} \]
 Thus we can use Scheff\'{e}'s lemma to infer
 \[ \lim_{n \to \infty} \sum_{\substack{\underline{\gamma}_1, \underline{\gamma}_2 \in \{0,1\}^{I_{\underline{\lambda}}}, \\ \underline{\beta}_1, \underline{\beta}_2 \in \mathbb{N}^k, \\ m_1, m_2 \in \mathbb{N}_+}} \Bigg| \, q_n \Big( \big( \underline{\gamma}_j, \underline{\beta}_j, m_j \big)_{j \in [2]} \Big) - q^* \Big( \big( \underline{\gamma}_j, \underline{\beta}_j, m_j \big)_{j \in [2]} \Big) \Bigg| = 0 \text{.} \]
 In particular,
 \[ \lim_{n \to \infty} \sum_{\substack{\underline{\beta}_1, \underline{\beta}_2 \in \mathbb{N}^k, \\ m_1, m_2 \in \mathbb{N}_+}} \Bigg| \, q_n \Big( \big( \underline{\gamma}_j, \underline{\beta}_j, m_j \big)_{j \in [2]} \Big) - q^* \Big( \big( \underline{\gamma}_j, \underline{\beta}_j, m_j \big)_{j \in [2]} \Big) \Bigg| = 0 \text{.} \]
 Putting this together with \eqref{eq:for_RHS} and \eqref{eq:for_LHS}, we obtain~\eqref{eq:typeI_of_two_vertices_conv_in_distr}, therefore the proof of the theorem is complete.
\end{proof}

Note that with some necessary modifications, the following version of the theorem could also be proved.
 
Let $k \in \mathbb{N}_+$, $m \in \mathbb{N}$, $\underline{\lambda} \in \mathbb{R}_+^k$, $\underline{\gamma} \in \{ 0,1 \}^{I_{\underline{\lambda}}}$, and let $G_n \sim \mathcal{G}_n([n-m], \underline{\lambda})$ for all $n \in \mathbb{N}_+$ with $n > m$. If Assumption~\ref{asm:simplifying_assumption} holds, then
\begin{equation} \label{eq:conv_of_typeI_version2}
 \frac{1}{n-m} \sum_{v \in [n-m]} \mathds{1} \big[ \, \underline{t}(v, G_n) =\underline{\gamma} \, \big] \; \stackrel{\mathbb{P}}{\longrightarrow} \; p^*(\underline{\gamma}), \qquad n \to \infty \text{.}
\end{equation}

\subsection{Types and color-avoiding connectivity} \label{subsection:types_and_color_avoiding_connectivity}

The goal of Section~\ref{subsection:types_and_color_avoiding_connectivity} is to further elaborate on the relation between the notion of types and color-avoiding connectivity in ECER graphs. Informally, Lemma~\ref{lemma:asymptotic_indep_of_typeI} states that the type I vectors of $k$ distinct vertices are close to being independent if $n$ is big (note that the case $k=2$ was settled in \eqref{eq:typeI_of_two_vertices_conv_in_distr}). Roughly speaking, Lemma~\ref{lemma:bad_events} states that with high probability a pair of randomly picked vertices in an ECER graph are $i$-avoiding connected if and only if both of them are in the $i$-avoiding giant component. Informally, Lemma~\ref{lemma:types_trickle_down_the_same_way} states that if we only explore a ball of radius $d$ in an ECBP tree or in an ECER graph, but we know the type I vectors of the vertices on the boundary, then we can infer all information about the color-avoiding connectivity or friendship between the explored vertices and the root, moreover, the effect of type I vectors of vertices at level $d$ is the same in ECBP trees and ECER graphs.

\begin{lemma}[Asymptotic independence of type I vectors in ECER graphs] \label{lemma:asymptotic_indep_of_typeI}
 Let $k, \beta \in \mathbb{N}_+$ and $\underline{\lambda} \in \mathbb{R}_+^k$. Let $G_n \sim \mathcal{G}_n([n], \underline{\lambda})$ for all $n \in \mathbb{N}_+$ with $n \ge \beta$, and let us fix the distinct vertices $v_1, \ldots, v_{\beta} \in [n]$. For any $j \in [\beta]$, let $\underline{\gamma}_j \in \{ 0, 1 \}^{I_{\underline{\lambda}}}$. If Assumption~\ref{asm:simplifying_assumption} holds, then
 \[ \lim_{n \to \infty} \mathbb{P} \Big( \forall j \in [\beta]: \, \underline{t}(v_j, G_n) = \underline{\gamma}_j \Big) = \prod_{j \in [\beta]} p^*(\underline{\gamma}_j) \text{.} \]
\end{lemma}
\begin{proof}
 Since the vertices $v_1, \ldots, v_{\beta} \in [n]$ were chosen to be distinct, the vertex exchangeability of ECER graphs implies that $\mathbb{P} \big( \forall j \in [\beta]: \, \underline{t}(v_j, G_n) = \underline{\gamma}_j \big)$ does not depend on the choice of $v_1, \ldots, v_{\beta}$. Let us choose the vertices $w_1, \ldots, w_{\beta}$ independently and uniformly at random from $[n]$, and let
 \[ E(w_1, \ldots, w_{\beta}, G_n) \colonequals \{ \text{the vertices $w_1, \ldots, w_{\beta}$ are all different in $G_n$} \} \text{.} \]
 It is easy to see that
 \[ \lim_{n \to \infty} \mathbb{P} \big( E(w_1, \ldots, w_{\beta}, G_n) \big) = 1 \text{,} \]
 thus
 \begin{multline*}
  \lim_{n \to \infty} \mathbb{P} \Big( \forall j \in [\beta]: \, \underline{t}(v_j, G_n) = \underline{\gamma}_j \Big) \\
  = \lim_{n \to \infty} \mathbb{P} \Big( \forall j \in [\beta]: \, \underline{t}(w_j, G_n) = \underline{\gamma}_j \Bigm| E(w_1, \ldots, w_{\beta}, G_n) \Big) \\
  = \lim_{n \to \infty} \mathbb{P} \Big( \forall j \in [\beta]: \, \underline{t}(w_j, G_n) = \underline{\gamma}_j \Big) \text{.}
 \end{multline*}
 For any $\underline{\gamma} \in \{ 0,1 \}^{I_{\underline{\lambda}}}$, let
 \[ V (\underline{\gamma}, G_n) \colonequals \big\{ v \in [n] : \underline{t}(v, G_n) = \underline{\gamma} \big\} \text{.} \]
 Then
 \begin{equation*}
  \mathbb{P} \Big( \forall j \in [\beta]: \, \underline{t}(w_j, G_n) = \underline{\gamma}_j \Big) 
  = \mathbb{E} \left( \prod_{j \in [\beta]} \frac{ \big| V(\underline{\gamma}_j, G_n) \big| }{n} \right) \text{.}
 \end{equation*}
 By Theorem~\ref{thm:conv_of_typeI}, we have
 \[ \prod_{j \in [\beta]} \frac{ \big| V(\underline{\gamma}_j, G_n) \big| }{n} \stackrel{\mathbb{P}}{\longrightarrow} \prod_{j \in [\beta]} p^*(\underline{\gamma}_j), \qquad n \to \infty \text{,} \]
 therefore, since $0 \leq |V(\underline{\gamma}, G_n)|/n \leq 1 $, we obtain
 \[ \lim_{n \to \infty} \mathbb{E} \left( \prod_{j \in [\beta]} \frac{ \big| V(\underline{\gamma}_j, G_n) \big| }{n} \right) = \prod_{j \in [\beta]} p^*(\underline{\gamma}_j) \text{,} \]
 and hence the statement of the lemma follows.
\end{proof}

Note that with some necessary modifications (e.g., using \eqref{eq:conv_of_typeI_version2} instead of Theorem~\ref{thm:conv_of_typeI}), the following version of the lemma could also be proved.
 
Let $k, \beta \in \mathbb{N}_+$, $m \in \mathbb{N}$ and $\underline{\lambda} \in \mathbb{R}_+^k$. Let $G_n \sim \mathcal{G}_n([n-m], \underline{\lambda})$ for all $n \in \mathbb{N}_+$ with $n \ge m+\beta$, and let us fix the distinct vertices $v_1, \ldots, v_{\beta} \in [n-m]$. For any $j \in [\beta]$, let $\underline{\gamma}_j \in \{ 0, 1 \}^{I_{\underline{\lambda}}}$. If Assumption~\ref{asm:simplifying_assumption} holds, then
\begin{equation} \label{eq:asymptotic_indep_of_typeI_version2}
 \lim_{n \to \infty} \mathbb{P} \Big( \forall j \in [\beta]: \, \underline{t}(v_j, G_n) = \underline{\gamma}_j \Big) = \prod_{j \in [\beta]} p^*(\underline{\gamma}_j) \text{.}
\end{equation}

Next we show that color-avoiding connectivity of two fixed vertices in an ECER graph is determined by the type I vectors of the vertices with high probability.

\begin{definition}[Bad events for in ECER graphs] \label{def:bad_events}
 Let $k \in \mathbb{N}_+$, $\underline{\lambda} \in \mathbb{R}_+^k$. Let $G \sim \mathcal{G}_n(V,\underline{\lambda})$, let $m \in \mathbb{N}_+$ with $m < |V|$ and let us fix the distinct vertices $v_1, \ldots, v_m \in V$, and let $\underline{v} = (v_1, \ldots, v_m)$. Let us define the events
 \[ B_1 (\underline{v}) \colonequals B_1 (\underline{v}, G) \colonequals \bigcup_{i \in [k] \setminus I_{\underline{\lambda}}} \; \bigcup_{\substack{j_1, j_2 \in [m]: \\ j_1 \ne j_2}} \Big\{ v_{j_1} \stackrel{G^{\setminus i}}{\longleftrightarrow} v_{j_2} \Big\} \text{,} \]
 \[ B_2 (\underline{v}) \colonequals B_2 (\underline{v}, G) \colonequals \bigcup_{i \in I_{\underline{\lambda}}} \; \bigcup_{\substack{j_1, j_2 \in [m]: \\ j_1 \ne j_2}} \! \Big\{ \big( t(v_{j_1}) \big)_i = \big( t(v_{j_2}) \big)_i = 1 \Big\} \triangle \Big\{ v_{j_1} \stackrel{G^{\setminus i}}{\longleftrightarrow} v_{j_2} \Big\} \text{,} \]
 and
 \[ B (\underline{v}) \colonequals B (\underline{v}, G) \colonequals B_1 (\underline{v}) \cup B_2 (\underline{v}) \text{,} \]
 where $\triangle$ denotes the symmetric difference of two sets.
\end{definition}

\begin{lemma}[Typical behaviour of ECER graphs] \label{lemma:bad_events}
 Let $k, m \in \mathbb{N}_+$, and $\underline{\lambda} \in \mathbb{R}_+^k$. Let $G_n \sim \mathcal{G}_n([n], \underline{\lambda})$ for all $n \in \mathbb{N}_+$ with $n \ge m$, and let us fix the distinct vertices $v_1, \ldots, v_m \in [n]$, and let $\underline{v} = (v_1, \ldots, v_m)$. Then
 \[ \lim_{n \to \infty} \mathbb{P} \big( B(\underline{v}, G_n) \big) = 0 \text{.} \]
\end{lemma}


\begin{proof}
 First we show that $\lim_{n \to \infty} \mathbb{P} \big( B_1(\underline{v}, G_n) \big) = 0$. Let $i \in [k] \setminus I_{\underline{\lambda}}$ and $j_1, j_2 \in [m]$, $j_1 \ne j_2$. It is enough to show that
 \begin{equation} \label{eq:vj1_and_vj2_not_i_avoiding_connected}
 \lim_{n \to \infty} \mathbb{P} \left( v_{j_1} \stackrel{G_n^{\setminus i}}{\longleftrightarrow} v_{j_2} \right) = 0
 \end{equation}
 holds. By the definition of $I_{\underline{\lambda}}$, we have $\lambda^{\setminus i} \le 1$, thus by Claim~\ref{claim:GI_is_ER}, the graph $G_n^{\setminus i}$ is a critical or subcritical Erd\H{o}s--R\'{e}nyi random graph. Thus by Theorem 5.1 of~\cite{vdH}, we have
 \begin{equation} \label{eq:no_i_avoiding_giant}
  \lim_{n \to \infty} \mathbb{P} \Big( \big| \, \mathcal{U}_{\text{max}} \big( G_n^{\setminus i} \big) \big| \ge n^{3/4} \Big) = 0.
 \end{equation}
 By the vertex exchangeability of the Erd\H{o}s--R\'{e}nyi random graph, we obtain
 \begin{multline} \label{eq:exchangeability}
  \mathbb{P} \left( v_{j_1} \stackrel{G_n^{\setminus i}}{\longleftrightarrow} v_{j_2} \right) = \mathbb{P} \Big( v_{j_2} \in \mathcal{C} \big( v_{j_1}, G_n^{\setminus i} \big) \Big) \\
  = \frac{1}{n-1} \sum_{v \in [n] \setminus \{v_j\} } \mathbb{P} \Big( v \in \mathcal{C} \big( v_{j_1}, G_n^{\setminus i} \big) \Big) = \frac{\mathbb{E} \Big( \, \Big| \mathcal{C} \big( v_{j_1}, G_n^{\setminus i} \big) \Big| \, \Big) - 1}{n-1} \\
  \le \frac{n^{3/4}-1}{n-1} + \mathbb{P} \Big( \, \Big| \, \mathcal{U}_{\text{max}} \big( G_n^{\setminus i} \big) \Big| \ge n^{3/4} \Big) \text{.}
 \end{multline}
 Thus~\eqref{eq:vj1_and_vj2_not_i_avoiding_connected} follows from~\eqref{eq:no_i_avoiding_giant} and~\eqref{eq:exchangeability}.
 
 \medskip
 
 Next, we show that $\lim_{n \to \infty} \mathbb{P} \big( B_2(\underline{v}, G_n) \big) = 0$.  Let us choose $i \in I_{\underline{\lambda}}$ and $j_1, j_2 \in [m]$, $j_1 \ne j_2$ arbitrarily. It is enough to show that
 \[ \lim_{n \to \infty} \mathbb{P} \left( \Big\{ \big( t(v_{j_1}) \big)_i = \big( t(v_{j_2}) \big)_i = 1 \Big\} \triangle \Big\{ v_{j_1} \stackrel{G^{\setminus i}}{\longleftrightarrow} v_{j_2} \Big\} \right) = 0 \text{.} \]
 Since $i \in I_{\underline{\lambda}}$, by Proposition~\ref{prop:good_events}, it is enough to show that
 \begin{multline*}
  \! \lim_{n \to \infty} \mathbb{P} \Bigg( A_1^{\setminus i}(G_n) \cap A_2^{\setminus i}(G_n) \cap \bigg( \Big\{ \big( t(v_{j_1}) \big)_i = \big( t(v_{j_2}) \big)_i = 1 \Big\} \triangle \Big\{ v_{j_1} \stackrel{G^{\setminus i}}{\longleftrightarrow} v_{j_2} \Big\} \bigg) \Bigg) \\
  = 0 \text{.}
 \end{multline*}
 Thus, by the union bound it is enough to show that
 \begin{multline} \label{eq:bad_no1}
  \! \lim_{n \to \infty} \mathbb{P} \left( A_1^{\setminus i}(G_n) \cap A_2^{\setminus i}(G_n) \cap \Big\{ \big( t(v_{j_1}) \big)_i = \big( t(v_{j_2}) \big)_i = 1 \Big\} \cap \Big\{ v_{j_1} \stackrel{G^{\setminus i}}{\longleftrightarrow} v_{j_2} \Big\}^c \right) \\
  = 0
 \end{multline}
 and
 \begin{multline} \label{eq:bad_no2}
  \! \lim_{n \to \infty} \mathbb{P} \left( A_1^{\setminus i}(G_n) \cap A_2^{\setminus i}(G_n) \cap \Big\{ \big( t(v_{j_1}) \big)_i = \big( t(v_{j_2}) \big)_i = 1 \Big\}^c \cap \Big\{ v_{j_1} \stackrel{G^{\setminus i}}{\longleftrightarrow} v_{j_2} \Big\} \right) \\
  = 0 \text{.}
 \end{multline}
 
 If $A_1^{\setminus i}(G_n) \cap \big\{ \big( t(v_{j_1}) \big)_i = \big( t(v_{j_2}) \big)_i = 1 \big\}$ holds, then both $v_{j_1}$ and $v_{j_2}$ are in the unique largest $i$-avoiding connected component of $G_n$, i.e.\ $v_{j_1}$ and $v_{j_2}$ are $i$-avoiding connected, thus~\eqref{eq:bad_no1} follows.
 
 On the other hand, if $A_2^{\setminus i}(G_n) \cap \big\{ \big( t(v_{j_1}) \big)_i = \big( t(v_{j_2}) \big)_i = 1 \big\}^c$ holds, then without loss of generality we may assume that $\big( t(v_{j_1}) \big)_i = 0$. Then $\big| \mathcal{C} \big( v_{j_1}, G_n^{\setminus i} \big) \big| \le n^{1/4}$, thus similarly to~\eqref{eq:no_i_avoiding_giant} and~\eqref{eq:exchangeability}, we have
 \[ \lim_{n \to \infty} \mathbb{P} \left( v_{j_1} \stackrel{G^{\setminus i}}{\longleftrightarrow} v_{j_2} \, \middle| \, A_2^{\setminus i}(G_n),\  \big( t(v_{j_1}) \big)_i=0\,  \right) = 0  \text{,} \]
 so~\eqref{eq:bad_no2} follows.
\end{proof}

\begin{lemma}[The boundary effect of type I vectors is the same in ECBP trees and ECER graphs] \label{lemma:types_trickle_down_the_same_way}
 Let $k, d \in \mathbb{N}_+$, $m \in \mathbb{N}$ and $\underline{\lambda} \in \mathbb{R}_+^k$.  Let $G_\infty \sim  \mathcal{G}_\infty( \underline{\lambda})$ and let $G_n \sim \mathcal{G}_n([n-m],\underline{\lambda})$ for all $n \in \mathbb{N}_+$ with $n \ge m$. Let $F$ denote a possible realization of $G_{\infty,d}(r)$ and let $\Gamma \!: R_d(r,F) \to \{0,1\}^{I_{\underline{\lambda}}}$ be an assignment of type I vectors to the vertices on the $d$-th level of $F$. Let $u \in V(F)$ and $v \in [n-m]$, and let $\underline{w}$ denote a possible arrangement of the set $R_d(r,F)$ of vertices into a vector. Let
 \begin{align*}
  q^* & \colonequals \mathbb{P} \left( \, u \in \widetilde{\mathcal{C}}^*(r,G_\infty) \, \middle| \, G_{\infty,d}(r) \simeq F, \, \forall w \in R_d(r,F):\, \underline{t}^*(w,G_\infty) = \Gamma(w) \right) \text{,} \\
  q_n & \colonequals \mathbb{P} \bigg( \, u \in \widetilde{\mathcal{C}}(v,G_n) \, \Bigm| \, G_{n,d}(v) \simeq F, \, \forall w \in R_d(r,F): \, \underline{t} \big( w,G'_{n,d}(v) \big) = \Gamma(w) \text{,} \\
       & \hspace{97pt} \Big( B \big( \underline{w}, G'_{n,d}(v) \big) \Big)^c \bigg) \text{,}
 \end{align*}
 where we identified the vertices of $F$, $G_{n,d}(v)$ and $G_{\infty,d}(r)$ under the above isomorphisms so that the root $r$ of $F$ is identified with $v$, noting that this also identifies $R_d(r,F)$ with a subset of the vertices of $G'_{n,d}(v)$. Then
 \begin{enumerate}[(i)]
  \item \label{item:zero_one_law} both $q^*$ and $q_n$ can only take the values $0$ or $1$, and
  \item \label{item:q_star_q_n_same} we have $q^* = q_n$.
 \end{enumerate}
  \end{lemma}
\begin{proof}
 Let us consider realizations of $G_{\infty}$ and $G_n$ for which all of the conditions appearing in the definitions of $q^*$ and $q_n$, respectively, hold. Note that our goal is to show that for any $i \in [k]$, the vertex $u$ is $i$-avoiding connected to $v$ in $G_n$ if and only if it is an $i$-avoiding friend of $r$ in $G_{\infty}$.

 \medskip
 
 \textit{Case 1:} $i \in [k] \setminus I_{\underline{\lambda}}$.
 
 Since the complement of the bad event $B_1 \big( \underline{w}, G'_{n,d}(v) \big)$ occurs, the vertices of $R_d(r,F)$ are not $i$-avoiding connected to each other in $G'_{n,d}(v)$. Thus any $i$-avoiding $v$-$u$ path must stay entirely in $G_{n,d}(v)$, so $v$ and $u$ are $i$-avoiding connected in $G_n$ if and only if the unique $r$-$u$ path in $F$ does not contain any edge of color $i$.
 
 In addition,  $i \in [k] \setminus I_{\underline{\lambda}}$ implies that the root $r$ of $G_\infty$ is almost surely not $i$-avoiding connected to infinity, thus by Definition~\ref{def:friends}, the vertices $r$ and $u$ are $i$-avoiding friends if and only if the unique $r$-$u$ path in $F$ does not contain any edge of color $i$. Thus in this case, the statements of the lemma hold.
 
 
 \medskip

 \textit{Case 2:} $i \in I_{\underline{\lambda}}$.
 
 \smallskip
 
 \textit{Case 2.1:} ($i \in I_{\underline{\lambda}}$ and) $r$ and $u$ are $i$-avoiding connected in $F$.
 
 In this case, the statements of the lemma clearly hold.
 
 \smallskip
 
 \textit{Case 2.2:} ($i \in I_{\underline{\lambda}}$ and) $r$ and $u$ are not $i$-avoiding connected in $F$.
 
 %
 %
 If $v$ and $u$ are $i$-avoiding connected in $G_n$, then any $i$-avoiding $v$-$u$ path cannot lie entirely in $G_{n,d}(v)$, thus there exists $h \in \mathbb{N}_+$ and there exist $w_1, \ldots, w_{2h} \in R_d(v, G_n)$ such that there exist an $i$-avoiding $v$-$w_1$ path and an $i$-avoiding $u$-$w_{2h}$ path contained completely in $G_{n,d}(v)$, and for any $j \in [h]$, there exists an $i$-avoiding $w_{2j}$-$w_{2j+1}$ path contained completely in $G_{n,d}(v)$ and there exists an $i$-avoiding $w_{2j-1}$-$w_{2j}$ path contained completely in $G'_{n,d}(v)$. Since the complement of the bad event $B_2 \big( \underline{w}, G'_{n,d}(v) \big)$ occurs, we obtain $\big( \underline{t}(w_{2j-1},G'_{n,d}(v)) \big)_i = \big( \underline{t}(w_{2j},G'_{n,d}(v)) \big)_i = 1$ for any $j \in [h]$, and the conditions appearing in the definition of $q_n$ imply that $\big( \Gamma(w_{2j-1}) \big)_i = \big( \Gamma(w_{2j}) \big)_i = 1$ must hold for any $j \in [h]$. The conditions appearing in the definition of $q^*$ imply that $\big( \underline{t}^*(w_{2j-1},G_{\infty}) \big)_i = \big( \underline{t}^*(w_{2j},G_{\infty}) \big)_i = 1$ holds for any $j \in [h]$. Since $r$ is $i$-avoiding connected to $w_1$ and so is $u$ to $w_{2h}$ in $F$, we obtain that $r$ and $u$ are $i$-avoiding friends in $G_{\infty}$.
 
 If $r$ and $u$ are $i$-avoiding friends in $G_{\infty}$, then both $r$ and $u$ are $i$-avoiding connected to infinity. Thus there exist $w_1, w_2 \in R_d(r, G_{\infty})$ such that there exist an $i$-avoiding $r$-$w_1$ path and an $i$-avoiding $u$-$w_2$ path contained completely in $G_{\infty,d}(r)$ and both $w_1$ and $w_2$ are $i$-avoiding connected to infinity. Thus $\big( \underline{t}^*(w_1,G_{\infty}) \big)_i = \big( \underline{t}^*(w_2, G_{\infty}) \big)_i = 1$, so the conditions appearing in the definition of $q^*$ imply that $\big( \Gamma(w_1) \big)_i = \big( \Gamma(w_2) \big)_i = 1$ must hold. The conditions appearing in the definition of $q_n$ imply that $\big( \underline{t}(w_1, G'_{n,d}(v)) \big)_i = \big( \underline{t}(w_2, G'_{n,d}(v)) \big)_i = 1$. Since the complement of the bad event $B_2 \big( \underline{w}, G'_{n,d}(v) \big)$ occurs, we obtain that $w_1$ and $w_2$ are $i$-avoiding connected in $G'_{n,d}(v)$. Since $r$ is $i$-avoiding connected to $w_1$ and so is $u$ to $w_2$ in $F$, we obtain that $v$ and $u$ are $i$-avoiding connected in $G_n$.
 
 Thus in this last case, the statements of the lemma hold.
 %
 %
 %
 %
\end{proof}

\section{Convergence of empirical component size densities} \label{section:conv_of_f_ell_G_n}

The goal of Section~\ref{section:conv_of_f_ell_G_n} is to prove Theorem~\ref{thm:conv_of_f_ell_G_n}. In Section~\ref{subsection:conv_of_giant_block}, we prove the statement about the convergence of the empirical density of the giant color-avoiding connected component in ECER graphs as well as Proposition~\ref{prop:char_of_giant}. In Section~\ref{subsection:conv_of_small_comps}, we prove convergence of $f_{\ell}(G_n)$ to $f^*_{\ell}$ as well as Proposition~\ref{prop:char_of_component_structure}.

Let $\underline{1} \in \mathbb{R}^k$ denote the vector whose every coordinate is equal to~1.

\subsection{Convergence of empirical density of giant color-avoid\-ing component} \label{subsection:conv_of_giant_block}

The goal of Section~\ref{subsection:conv_of_giant_block} is to prove the statement of Theorem~\ref{thm:conv_of_f_ell_G_n} about the convergence of the fraction of vertices contained in the largest color-avoid\-ing  component in ECER graph sequences, i.e., showing that
\[\frac{1}{n} \max_{v \in [n]} \Big| \widetilde{\mathcal{C}}(v, G_n) \Big| \; \stackrel{\mathbb{P}}{\longrightarrow} \; f^*_\infty, \qquad n \to \infty \text{.} \]

\begin{lemma}[The root has infinitely many friends iff its type I vector is $\underline{1}$] \label{lemma:infinitely_many_friends}
 Let $k \in \mathbb{N}_+$, $\underline{\lambda}, \underline{1} \in \mathbb{R}_+^k$ and let $G_{\infty} \colonequals G_{\infty}(r) \sim \mathcal{G}_{\infty}(\underline{\lambda})$. Then the events
 \[ \bigg\{ \Big| \widetilde{\mathcal{C}}^*(r) \Big| = \infty \bigg\} \qquad \text{and} \qquad \big\{ \underline{t}^*(r) = \underline{1} \big\} \]
 $\mathbb{P}$-almost surely coincide.
 
 In particular, $f^*_{\infty} = p^*(\underline{1})$.
\end{lemma}
\begin{proof}
 If $\underline{\lambda}$ is not fully supercritical (cf.\ Definition~\ref{def:fully_supercrit}), then there exists a color $i \in [k]$ for which $\lambda^{\setminus i} \le 1$; let $i \in [k]$ be such a color. Clearly, $i \notin I_{\underline{\lambda}}$, thus $\underline{t}^*(r) \ne \underline{1}$ since the lengths of these vectors are not equal. Now we need to show that $\big\{ \big| \widetilde{\mathcal{C}}^*(r) \big| = \infty \big\}$ almost surely does not occur. By Claim~\ref{claim:GIr_is_BP}, the branching process $\big( \big| R_{\infty,d}^{\setminus i} \big| \big)$ almost surely dies out, which means that $r$ is not $i$-avoiding connected to infinity (cf.\ Definition~\ref{def:friends}), thus the set of $i$-avoiding friends of $r$ is $\mathcal{C} \big( r, G_{\infty}^{\setminus i} \big)$, which is almost surely finite. Since each friend of $r$ is also an $i$-avoiding friend of it, $\widetilde{\mathcal{C}}^* (r, G_{\infty}) \subseteq \mathcal{C} \big( r, G_{\infty}^{\setminus i} \big)$ holds, which implies that $r$ has almost surely finitely many friends.
 
 \medskip
 
 If $\underline{\lambda}$ is fully supercritical, then assume first that the event $\big\{ \big| \widetilde{\mathcal{C}}^*(r) \big| = \infty \big\}$ occurs. We need to show that $r$ is $i$-avoiding connected to infinity for all $i \in [k]$. Let $i \in [k]$ be an arbitrary color. Since $\big| \widetilde{\mathcal{C}}^*(r) \big| = \infty$, clearly
 \[ \Big\{ v \stackrel{G_{\infty}^{\setminus i} }{\longleftrightarrow} r \Big\} \cup \left( \Big\{ v \stackrel{G_{\infty}^{ \setminus i} }{\longleftrightarrow} \infty \Big\} \cap \Big\{ r \stackrel{G_{\infty}^{ \setminus i} }{\longleftrightarrow} \infty \Big\} \right) \]
 holds for infinitely many vertices $v$. Now this event is the union of two events and the second one contains the desired event
 \[ \left\{ r \stackrel{G_\infty^{\setminus i}}{\longleftrightarrow} \infty \right\} \text{,} \]
 therefore if the second event occurs for any vertex $v$, then we are done. So assume that the first event holds for infinitely many vertices $v$. Noting that the degree of every vertex of $G_\infty^{\setminus i}$ is $\mathbb{P}$-almost surely finite, we obtain by K\H{o}nig's lemma \cite{konig} that there exists a path from the root $r$ to infinity in $G_\infty^{\setminus i}$, and we are done.
 
 Now let us assume that $\underline{t}^*(r) = \underline{1}$, and we need to show that this implies $\big\{ |\widetilde{\mathcal{C}}^*(r, G_\infty)| = \infty \big\}$. Let
 \[ V_{\infty}(\underline{1}) \colonequals \big\{ v \in V(G_\infty) : \underline{t}^*(v, G_\infty) = \underline{1} \big\} \text{.} \]
 By the definitions of friendship and type I vectors (cf.\ Definitions~\ref{def:typeI_in_ECBP} and~\ref{def:typeI_in_ECBP}, respectively), we only need to show that
 \begin{equation} \label{eq:V_infty_is_infty}
  \mathbb{P} \Big( \, \big| V_{\infty}(\underline{1}) \big| = +\infty \Bigm| \underline{t}^*(r) = \underline{1} \Big) = 1 \text{.} 
 \end{equation}
 Let $i \in [k]$ be an arbitrary color. By Claim~\ref{claim:GIr_is_BP}, we have that $\big( \big| R_{\infty,d}^{\setminus i} \big| \big)_{d \in \mathbb{N}}$ is a supercritical branching process and the Kesten--Stigum theorem (see, e.g., Theorem A of~\cite{lpp}) implies that conditional on $\big( \underline{t}^*(r) \big)_i = 1$, we have 
 \[ \big| R_{\infty,d}^{\setminus i} \big| \stackrel{\mathbb{P}}{\longrightarrow} \infty, \qquad d \to \infty \text{.} \]
 For any $d \in \mathbb{N}$, conditional on $G_{\infty} \Big[ R^{\le}_{\infty,d} \big( r, G_{\infty}^{\setminus i} \big) \Big]$, the vertices of $R_{\infty,d}^{\setminus i}$ have type I vectors $\underline{1}$ independently of each other with probability $p^*(\underline{1})$, thus the distribution of $\big| R_{\infty,d}^{\setminus i} \cap V_{\infty}(\underline{1}) \big|$ is $\mathrm{BIN} \big( \big| R_{\infty,d}^{\setminus i} \big|, p^*(\underline{1}) \big)$. By the Harris--FKG inequality \cites{harris, fkg} and by the fully supercriticality of $\underline{\lambda}$, we obtain
 \[ p^*(\underline{1}) \ge \prod_{i \in [k]} \mathbb{P} \Big( \big( \underline{t}^*(r) \big)_i = 1 \Big) = \prod_{i \in [k]} \theta^{\setminus i} > 0 \text{.} \]
 Therefore, conditional on the event $\{ \underline{t}^*(r) = \underline{1} \}$, we have
 \[ \Big| R_{\infty,d}^{\setminus i} \cap V_{\infty}(\underline{1}) \Big| \stackrel{\mathbb{P}}{\longrightarrow} \infty, \qquad d \to \infty \text{.} \]
 Hence \eqref{eq:V_infty_is_infty} holds.
\end{proof}

\begin{proof}[Proof of Proposition~\ref{prop:char_of_giant}]
 As we saw in the proof of Lemma~\ref{lemma:infinitely_many_friends}, if $\underline{\lambda}$ is not fully supercritical, then $r$ has almost surely finitely many friends, i.e.\ $f^*_{\infty} = 0$, and if $\underline{\lambda}$ is fully supercritical, then $f^*_{\infty} = p^*(\underline{1}) > 0$.
\end{proof}

\begin{lemma} \label{lemma:conv_giant_block}
 Let $k \in \mathbb{N}_+$, $\underline{\lambda} \in \mathbb{R}_+^k$ and let $G_n \sim \mathcal{G}_n([n],\underline{\lambda})$ for all $n \in \mathbb{N}_+$. If Assumption~\ref{asm:simplifying_assumption} holds, then
 \[ \frac{1}{n} \max_{v \in [n]} \Big| \widetilde{\mathcal{C}}(v, G_n) \Big| \; \stackrel{\mathbb{P}}{\longrightarrow} \; f^*_\infty, \qquad n \to \infty \text{.} \]
\end{lemma}
\begin{proof}
 If $\underline{\lambda}$ is not fully supercritical, then by Proposition~\ref{prop:char_of_giant}, we have $f^*_{\infty} = 0$. Let $i \in [k]$ be a color for which $\lambda^{\setminus i} \le 1$ (since $\underline{\lambda}$ is not fully supercritical, such a color $i$ exists). By Claim~\ref{claim:GI_is_ER} and Theorem 4.8 of~\cite{vdH}, we obtain
 \[ \frac{\Big| \, \mathcal{U}_{\text{max}} \big( G_n^{\setminus i} \big) \Big|}{n} \stackrel{\mathbb{P}}{\longrightarrow} 0 = f^*_{\infty}, \qquad n \to \infty \text{.} \]
 Since every color-avoiding connected component of $G_n$ is contained in a connected component of $G_n^{\setminus i}$, we are done.
 
 Now let us assume that $\underline{\lambda}$ is fully supercritical, and let $i \in [k]$ be an arbitrary color. By Proposition~\ref{prop:good_events}, we obtain
 \[ \lim_{n \to \infty} \mathbb{P} \big( A_1^{\setminus i}(G_n) \cap A_2^{\setminus i}(G_n) \big) = 1 \text{.} \]
 If the event $A_1^{\setminus i}(G_n) \cap A_2^{\setminus i}(G_n)$ holds for all $i \in [k]$, then $\bigcap_{i \in [k]} \mathcal{U}_{\text{max}} \big( G_n^{\setminus i} \big)$ is a color-avoiding connected component of $G_n$ and all the other color-avoiding connected components have size at most $n^{1/4}$. Thus,
 \[ \left| \bigcap_{i \in [k]} \mathcal{U}_{\text{max}} \big( G_n^{\setminus i} \big) \right| = \max_{v \in [n]} \Big| \widetilde{\mathcal{C}}(v, G_n) \Big| \text{.} \]
 By the definition of type I vectors (cf.\ Definition~\ref{def:typeI_in_ECER}),
 \[ \bigcap_{i \in [k]} \mathcal{U}_{\text{max}} \big( G_n^{\setminus i} \big) = \big\{ v \in [n] : \underline{t}(v, G_n) = \underline{1} \big\} \text{.} \]
 Thus by Theorem~\ref{thm:conv_of_typeI} and Lemma~\ref{lemma:infinitely_many_friends}, we have
 \[ \frac{1}{n} \max_{v \in [n]} \Big| \widetilde{\mathcal{C}}(v, G_n) \Big| = \frac{1}{n} \sum_{v \in [n]} \big[ \, \underline{t}(v, G_n) = \underline{1} \, \big] \stackrel{\mathbb{P}}{\longrightarrow} \; p^*(\underline{1}) = f^*_{\infty}, \qquad n \to \infty \text{.} \]
\end{proof}

\subsection{Convergence of \texorpdfstring{$f_{\ell}(G_n)$}{} to \texorpdfstring{$f^*_{\ell}$}{}} \label{subsection:conv_of_small_comps}

The goal of Section~\ref{subsection:conv_of_small_comps} is to prove the remaining part of Theorem~\ref{thm:conv_of_f_ell_G_n}, which we restate now.

\begin{lemma}[Convergence of empirical component size densities] \label{lemma:conv_of_f_ell_G_n}
 Let $k \in \mathbb{N}_+$, $\underline{\lambda} \in \mathbb{R}_+^k$ and let $G_n \sim \mathcal{G}_n([n], \underline{\lambda})$ for all $n \in \mathbb{N}_+$. If Assumption~\ref{asm:simplifying_assumption} holds, then for any $\ell \in \mathbb{N}_+$,
 \[ f_{\ell}(G_n) \; \stackrel{\mathbb{P}}{\longrightarrow} \; f^*_{\ell}, \qquad n \to \infty \text{.} \]
\end{lemma}

The main difficulty is that the outcome of the event $\big\{ \big| \widetilde{\mathcal{C}}(v, G_n) \big| = \ell \big\}$ cannot be determined by looking at a ``small'' neighborhood of $v$ in $G_n$. However, the event that there exists a color $i$ for which the $i$-avoiding component of $v$ is contained in a ball of radius $d-1$ can be determined by looking at the $d$-neighborhood of $v$, and if this event occurs then the $d$-neighborhood of $v$ and the type I vectors of the vertices on the boundary of the $d$-neighborhood of $v$ together determine the set of vertices which are in the color-avoiding connected component $\widetilde{\mathcal{C}}(v, G_n)$ of $v$ with high probability. In order to make these ideas precise, we need some further definitions.
  
\begin{definition}[The color-avoiding horizon $\tau_\infty$ in ECBP trees] \label{def:stopping_time}
 Let $k \in \mathbb{N}_+$, $\underline{\lambda} \in \mathbb{R}_+^k$ and let $G_{\infty} \colonequals G_\infty(r) \sim \mathcal{G}_\infty(\underline{\lambda})$.
 
For all $i \in [k]$, let
 \[ \tau_{\infty}^{\setminus i} \colonequals \max \left\{ d \in \mathbb{N} \ \middle| \ \exists v \in R_{\infty,d}: r \stackrel{G_{\infty}^{\setminus i}}{\longleftrightarrow} v \right\} \text{,} \]
 and let
 \[ \tau_{\infty} \colonequals \min_{i \in [k]} \tau_{\infty}^{\setminus i} \text{.} \]

 For any $d \in \mathbb{N}_+$, let $\mathcal{B}_d^{\tau_{\infty}}$ denote the set of pairwise non-isomorphic possible realizations of $G_{\infty,d}(r)$ for which the event $\{ \tau_{\infty} = d-1 \}$ holds.
 
 For any $\ell \in \mathbb{N}_+$, $d \in \mathbb{N}_+$ and $F \in \mathcal{B}_d^{\tau_{\infty}}$, let
 \[ f^*_{\ell,d,F} \colonequals f^*_{\ell,d,F}(\underline{\lambda}) \colonequals \mathbb{P} \bigg( \, \Big| \widetilde{\mathcal{C}}^* (r, G_{\infty}) \Big| = \ell, \, G_{\infty,d}(r) \simeq F \, \bigg) \text{.} \]
\end{definition}

Note that the event $\{ \tau_{\infty} = d-1 \}$ is measurable with respect to the sigma-algebra generated by $G_{\infty,d}(r)$, thus
\begin{equation*}
 \biguplus_{F \in \mathcal{B}_d^{\tau_{\infty}}} \big\{ G_{\infty,d}(r) \simeq F \big\} = \{ \tau_{\infty} = d-1 \} \text{.}
\end{equation*}
Also note that $\tau_{\infty} = \infty$ if and only if $\underline{t}^*(r) = \underline{1}$, where $\underline{1} \in \mathbb{R}_+^k$. Thus from Proposition~\ref{prop:char_of_giant} and Lemma~\ref{lemma:infinitely_many_friends}, it follows that $\mathbb{P}(\tau_{\infty} = \infty) > 0$ if and only if $\underline{\lambda}$ is fully supercritical.

Observe that
\begin{equation} \label{eq:fstar_elldF_sums_to_1}
 \sum_{\ell \in \mathbb{N}_+} \sum_{d \in \mathbb{N}_+} \sum_{F \in \mathcal{B}_d^{\tau_{\infty}}} f^*_{\ell,d,F} = 1
\end{equation}
holds.

\begin{definition}[Component size $\ell$ with color-avoiding horizon $d-1$ in ECER graphs] \label{def:f_ell_d_F_G}
 Let $k, \ell, n \in \mathbb{N}_+$, $d \in \mathbb{N}_+$, $\underline{\lambda} \in \mathbb{R}_+^k$, let $G \sim \mathcal{G}_n ([n],\underline{\lambda})$ and $v \in [n]$, and let $F \in \mathcal{B}_d^{\tau_{\infty}}$. Let us define
 \[ f_{\ell,d,F}(G) \colonequals f_{\ell,d,F}(\underline{\lambda}, G) \colonequals \frac{1}{n} \sum_{v \in [n]} \mathds{1} \bigg[ \, \Big| \widetilde{\mathcal{C}}(v,G) \Big| = \ell,\,  G_d(v) \simeq F \, \bigg] \text{.} \]
\end{definition}

\begin{lemma}[Convergence of fixed horizon component size densities] \label{lemma:f_ell_d_F_G_n_conv}
 Let $k, \ell \in \mathbb{N}_+$, $d \in \mathbb{N}_+$, and $\underline{\lambda} \in \mathbb{R}_+^k$. Let $G_n \sim \mathcal{G}_n([n], \underline{\lambda})$ for all $n \in \mathbb{N}_+$, and let $F \in \mathcal{B}_d^{\tau_{\infty}}$. If Assumption~\ref{asm:simplifying_assumption} holds, then
 \[ f_{\ell, d, F}(G_n) \stackrel{\mathbb{P}}{\longrightarrow} f^*_{\ell, d, F}, \qquad n \to \infty \text{.} \]
\end{lemma}
\begin{proof}
 Let $G_{\infty}(r) \sim \mathcal{G}_{\infty}(\underline{\lambda})$. Let $\ell_1, \ell_2 \in \mathbb{N}_+$, $d_1, d_2 \in \mathbb{N}_+$, let $v_1, v_2 \in [n]$, $v_1 \ne v_2$ and $F_1 \in \mathcal{B}_{d_1}^{\tau_{\infty}}, F_2 \in \mathcal{B}_{d_2}^{\tau_{\infty}} $. Similarly to the beginning of the proof of Theorem~\ref{thm:conv_of_typeI}, it suffices to prove that
 \begin{equation} \label{eq:f_elldF_no1}
  \lim_{n \to \infty} \mathbb{P} \bigg( \Big| \widetilde{\mathcal{C}}(v_1, G_n) \Big| = \ell_1, \, G_{n,d_1}(v_1) \simeq F_1 \bigg) = f^*_{\ell_1, d_1, F_1}
 \end{equation}
 and
 \begin{equation} \label{eq:f_elldF_no2}
 \lim_{n \to \infty} \mathbb{P} \bigg( \forall j \in [2] : \Big| \widetilde{\mathcal{C}}(v_j,G_n) \Big| = \ell_j, \, G_{n,d_j}(v_j) \simeq F_j \bigg) = \prod_{j \in [2]} f^*_{\ell_j,d_j,F_j}
 \end{equation}
 hold. Note that~\eqref{eq:f_elldF_no1} follows from~\eqref{eq:f_elldF_no2} by the total law of probability, by \eqref{eq:fstar_elldF_sums_to_1} and from the fact that $\lim_{n \to \infty} \mathbb{P} \big( \text{$G_{n,d_2}(v_2)$ is a tree} \big) = 1$. So it is enough to prove~\eqref{eq:f_elldF_no2}.
 
 First, we reformulate the right-hand side of~\eqref{eq:f_elldF_no2}. Since $F_j \in \mathcal{B}_{d_j}^{\tau_{\infty}}$ for any $j \in [2]$, there exists a color $i_j \in [k]$ such that
 $ G_{\infty,{d_j}}(r) \simeq F_j$ implies $\tau_{\infty}^{\setminus i_j} = {d_j}-1$, hence all of the friends of $r$ are in $ R^{\le}_{\infty, d_j-1}$. For any $\ell, d \in \mathbb{N}_+$ and $F \in \mathcal{B}_d^{\tau_{\infty}}$, let $\mathcal{T}_d(\ell,F)$ denote the set of possible assignments of type I vectors to the vertices of $R_{\infty,d}$ for which $\big| \widetilde{\mathcal{C}}^*(r) \big| = \ell$ and $G_{\infty,d}(r) \simeq F$. Note that by \eqref{item:zero_one_law} of Lemma~\ref{lemma:types_trickle_down_the_same_way}, $\mathcal{T}_d(\ell,F)$ is well-defined and given $G_{\infty,d}(r) \simeq F$, any $\Gamma \in \mathcal{T}_d(\ell, F)$ determines $\widetilde{\mathcal{C}}^*(r)$. By the total law of probability, we obtain
 \begin{multline} \label{eq:fstar_elldF_prod_sum}
  \prod_{j \in [2]} f^*_{\ell_j,d_j,F_j} \\
  = \sum_{\substack{\Gamma_1 \in \mathcal{T}_{d_1}(\ell_1, F_1), \\ \Gamma_2 \in \mathcal{T}_{d_2}(\ell_2, F_2)}} \, \prod_{j \in [2]} \mathbb{P} \big( G_{\infty,d_j}(r) \simeq F_j, \, \forall u \in R_{\infty,d_j}: ~ \underline{t}^*(u, G_{\infty}) = \Gamma_j(u) \big) \\
  = \sum_{\substack{\Gamma_1 \in \mathcal{T}_{d_1}(\ell_1, F_1), \\ \Gamma_2 \in \mathcal{T}_{d_2}(\ell_2, F_2)}} \, \prod_{j \in [2]} \left( \mathbb{P} \big( G_{\infty,{d_j}}(r) \simeq F_j \big) \cdot \prod_{u \in R_{d_j}(r,F_j) } \mathbb{P} \big( \underline{t}^*(u, G_{\infty}) = \Gamma_j(u) \big) \right) \\
  = \left( \prod_{j \in [2]} \mathbb{P} \big( G_{\infty,{d_j}}(r) \simeq F_j \big) \right) 
  \cdot \sum_{\substack{\Gamma_1 \in \mathcal{T}_{d_1}(\ell_1, F_1), \\ \Gamma_2 \in \mathcal{T}_{d_2}(\ell_2, F_2)}} \, \prod_{j \in [2]} \, \prod_{u \in R_{d_j}(r,F_j)} p^* \big( \Gamma_j(u) \big) \text{.}
 \end{multline}
 
 Now we reformulate the left-hand side of~\eqref{eq:f_elldF_no2}. Let
 \begin{multline*}
  E(v_1, v_2, d_1, d_2, G_n) \\
  \colonequals \big\{ \text{$G_{n,d_1}(v_1) \cup G_{n,d_2}(v_2)$ is a forest with two components} \big\} \text{,}
 \end{multline*}
 where $G_{n,d_1}(v_1) \cup G_{n,d_2}(v_2)$ is the edge-colored sub(multi)graph of $G_n$ whose vertex and edge set is the (not necessarily disjoint) union of the vertex and edge sets of $G_{n,d_1}(v_1)$ and $G_{n,d_2}(v_2)$, respectively. Then by taking $d \colonequals \max \{ d_1, d_2 \}$ in~\eqref{eq:lim_of_Ed_is_1}, we obtain that
 \begin{multline} \label{eq:lim_ell12_d12_F12}
  \lim_{n \to \infty} \mathbb{P} \bigg( \forall j \in [2]: \Big| \widetilde{\mathcal{C}}(v_j, G_n) \Big| = \ell_j, \, G_{n,d_j}(v_j) \simeq F_j \bigg) \\
  = \lim_{n \to \infty} \mathbb{P} \bigg(  E(v_1, v_2, d_1, d_2, G_n), \, \forall j \in [2]: \Big| \widetilde{\mathcal{C}}(v_j, G_n) \Big| = \ell_j, \, G_{n,d_j}(v_j) \simeq F_j \bigg) \\
  = \lim_{n \to \infty} \mathbb{P} \bigg( \forall j \in [2]: \, \Big| \widetilde{\mathcal{C}}(v_j, G_n) \Big| = \ell_j \biggm| G_{n,d_1}(v_1) \cup G_{n,d_2}(v_2) \simeq F_1 \oplus F_2 \bigg) \\ \cdot \mathbb{P} \big( G_{n,d_1}(v_1) \cup G_{n,d_2}(v_2) \simeq F_1 \oplus F_2 \big)
 \end{multline}
 holds.
 
 Conditional on $\big\{ G_{n,d_1}(v_1) \cup G_{n,d_2}(v_2) \simeq F_1 \oplus F_2 \big\}$, we identify $G_{n,d_j}(v_j)$ with $F_j$ for all $j \in [2]$, and clearly, $R_{d_1}(v_1, G_n) \cap R_{d_2}(v_2, G_n) = \emptyset$ holds. Since $F_j \in \mathcal{B}^{\tau_{\infty}}_{d_j}$, we get $\widetilde{C}(v_j, G_n) \subseteq V(F_j)$ for all $j \in [2]$. Let 
 \[ G'_{n, d_1, d_2}(v_1, v_2) \colonequals G_n - \Big( R^{\le}_{d_1-1} (v_1) \cup R^{\le}_{d_2-1} (v_2) \Big) . \]
 Applying Lemma~\ref{lemma:after_conditioning_still_ECER} twice, we get that conditional on $G_{n,d_1}(v_1) \cup G_{n,d_2}(v_2)$, the subgraph $G'_{n, d_1, d_2}(v_1, v_2)$ is still an ECER graph with parameters $n$ and $\underline{\lambda}$.
 
 Thus by Lemma~\ref{lemma:bad_events} for $G'_{n, d_1, d_2}(v_1, v_2)$ and the vertices of $R_{d_1}(v_1, G_n) \cup R_{d_2}(v_2, G_n)$, we can apply Lemma~\ref{lemma:types_trickle_down_the_same_way} twice to conclude that given $G_{n,d_1}(v_1) \cup G_{n,d_2}(v_2) \simeq F_1 \oplus F_2$, any $\Gamma_j \in \mathcal{T}_{d_j}(\ell_j, F_j)$ determines $\widetilde{\mathcal{C}}(v_j)$ and with high probability, $\big| \widetilde{\mathcal{C}}(v_j, G_n) \big| = \ell_j$ if and only if the type configuration of $R_{d_j}(v_j, G_n)$ is in $\mathcal{T}_{d_j}(\ell_j, F_j)$ for any $j \in [2]$.
 
 Therefore and by Lemma~\ref{lemma:asymptotic_indep_of_typeI} (more precisely, by~\eqref{eq:asymptotic_indep_of_typeI_version2}), we obtain
 \begin{multline} \label{eq:friends_lim_sum_cond}
  \lim_{n \to \infty} \mathbb{P} \bigg( \forall j \in [2]: \, \Big| \widetilde{\mathcal{C}}(v_j, G_n) \Big| = \ell_j \biggm| G_{n,d_1}(v_1) \cup G_{n,d_2}(v_2) \simeq F_1 \oplus F_2 \bigg) \\
  = \! \lim_{n \to \infty} \sum_{\substack{\Gamma_1 \in \mathcal{T}_{d_1}(\ell_1, F_1), \\ \Gamma_2 \in \mathcal{T}_{d_2}(\ell_2, F_2)}} \! \! \mathbb{P} \Bigg( \forall j \in [2] \, \forall u \in R_{d_j}(v_j, G_n) \! : \, \underline{t}\big( u, G'_{n,d_1,d_2}(v_1, v_2) \big) \! = \! \Gamma_j(u) \! \Biggm| \\[-5pt]
  \hspace{205pt} G_{n,d_1}(v_1) \cup G_{n,d_2}(v_2) \simeq F_1 \oplus F_2 \Bigg) \\[15pt]
  = \sum_{\substack{\Gamma_1 \in \mathcal{T}_{d_1}(\ell_1, F_1), \\ \Gamma_2 \in \mathcal{T}_{d_2}(\ell_2, F_2)}} \, \prod_{j \in [2]} \prod_{u \in R_{d_j}(r,F_j) } p^* \big( \Gamma_j(u) \big) \text{.}
 \end{multline}
 
 Plugging~\eqref{eq:conv_of_edgecolored_balls_of_radius_d} and~\eqref{eq:friends_lim_sum_cond} into \eqref{eq:lim_ell12_d12_F12}, and also using~\eqref{eq:fstar_elldF_prod_sum}, we obtain the desired~\eqref{eq:f_elldF_no2}.
\end{proof}
 
We are now ready to prove Lemma~\ref{lemma:conv_of_f_ell_G_n}.

\begin{proof}[Proof of Lemma~\ref{lemma:conv_of_f_ell_G_n}]
 Let $\ell_0 \in \mathbb{N}$ and $\varepsilon > 0$. First, we show
 \begin{equation} \label{eq:lower_bound_on_f_ell}
  \lim_{n \to \infty} \mathbb{P} \big( f_{\ell_0}(G_n) \ge f^*_{\ell_0} - \varepsilon \big) = 1.
 \end{equation}
 By the total law of probability,
 \[ \sum_{d \in \mathbb{N}_+} \sum_{F \in \mathcal{B}_d^{\tau_{\infty}}} f^*_{\ell_0,d,F} = f^*_{\ell_0} \]
 holds, thus we can choose a $D_{\varepsilon} \in \mathbb{N}_+$, and for each $d \in [D_{\varepsilon}]$ a finite subset $\mathcal{B}_{d,\varepsilon}^{\tau_{\infty}}$ of $\mathcal{B}_d^{\tau_{\infty}}$ such that
 \[ \sum_{d \in [D_{\varepsilon}]} \sum_{F \in \mathcal{B}_{d,\varepsilon}^{\tau_{\infty}}} f^*_{\ell_0,d,F} \ge f^*_{\ell_0} - \frac{\varepsilon}{2} \text{.} \]
 By Lemma~\ref{lemma:f_ell_d_F_G_n_conv}, we have
 \[ \sum_{d \in [D_{\varepsilon}]} \sum_{F \in \mathcal{B}_{d,\varepsilon}^{\tau_{\infty}}} f_{\ell_0,d,F}(G_n) \stackrel{\mathbb{P}}{\longrightarrow} \sum_{d \in [D_{\varepsilon}]} \sum_{F \in \mathcal{B}_{d,\varepsilon}^{\tau_{\infty}}} f^*_{\ell_0,d,F} \text{.} \]
 Let us also note that
 \[ \sum_{d \in [D_{\varepsilon}]} \sum_{F \in \mathcal{B}_{d,\varepsilon}^{\tau_{\infty}}} f_{\ell_0,d,F}(G_n) \le \sum_{d \in \mathbb{N}_+} \sum_{F \in \mathcal{B}_d^{\tau_{\infty}}} f_{\ell_0,d,F}(G_n) \leq f_{\ell_0}(G_n) \]
 holds, therefore the desired inequality~\eqref{eq:lower_bound_on_f_ell} follows.
 
 \medskip
 
 Next, we show
 \begin{equation} \label{eq:upper_bound_on_f_ell}
  \lim_{n \to \infty} \mathbb{P} \big( f_{\ell_0}(G_n) \le f^*_{\ell_0} + \varepsilon \big) = 1 \text{.}
 \end{equation}
 By Claim~\ref{claim:sum_of_fstar_ell's}, there exists $L_{\varepsilon} \in \mathbb{N}_+$ such that 
 \begin{equation} \label{eq:lower_bound_on_sum_of_fstar_ell}
  \sum_{\ell \in [L_{\varepsilon}]} f^*_{\ell} \ge 1 - f^*_{\infty} - \frac{\varepsilon}{3}.
 \end{equation}
 It follows from Lemma~\ref{lemma:conv_giant_block} that
 \begin{equation*}
  \lim_{n \to \infty} \mathbb{P} \left( \sum_{\ell = L_{\varepsilon}+1}^{\infty} f_{\ell}(G_n) \ge f^*_{\infty} - \frac{\varepsilon}{3} \right) = 1 \text{,}
 \end{equation*}
 thus by Claim~\ref{claim:sum_of_f_ell_G's}, we have
 \begin{equation*}
  \lim_{n \to \infty} \mathbb{P} \left( \sum_{\ell \in [L_{\varepsilon}]} f_{\ell}(G_n) \le 1 - f^*_{\infty} + \frac{\varepsilon}{3} \right) = 1 \text{.}
 \end{equation*}
 Similarly to~\eqref{eq:lower_bound_on_f_ell}, we obtain that
 \begin{equation*}
  \lim_{n \to \infty} \mathbb{P} \left( f_{\ell}(G_n) \ge f^*_{\ell} - \frac{\varepsilon}{3 L_{\varepsilon}} \right) = 1
 \end{equation*}
 holds for all $\ell \in [L_{\varepsilon}]$. Thus by~\eqref{eq:lower_bound_on_sum_of_fstar_ell}, it follows that
 \begin{multline*}
  1 - f^*_\infty - \frac{2}{3} \varepsilon - f^*_{\ell_0} \le \sum_{\ell \in [L_{\varepsilon}]} f^*_{\ell} - \frac{\varepsilon}{3} - f^*_{\ell_0} = \sum_{\ell \in [L_{\varepsilon}]} \left( f^*_{\ell} - \frac{\varepsilon}{3 L_{\varepsilon}} \right) - f^*_{\ell_0} \\
  \le \sum_{\ell \in [L_{\varepsilon}] } f_{\ell}(G_n) - f_{\ell_0}(G_n) \le 1 - f^*_{\infty} + \frac{\varepsilon}{3} - f_{\ell_0}(G_n) 
 \end{multline*}
 holds with high probability, therefore we obtain~\eqref{eq:upper_bound_on_f_ell}, and hence the statement of the lemma follows.
\end{proof}

Now that we showed that the sequence $f_{\ell}(G_n)$ converges for any $\ell \in \mathbb{N_+}$, we prove a proposition about its limit $f^*_{\ell}$.

\begin{proof}[Proof of Proposition~\ref{prop:char_of_component_structure}]
 Let $G_{\infty} \colonequals G_{\infty}(r) \sim \mathcal{G}_{\infty}(\underline{\lambda})$.
 
 We begin with the proof of \eqref{item:fully}. Let $\underline{\lambda} \in \mathbb{R}_+^k$ be fully critical-subcritical. We need to show that $r$ has almost surely no friends other than itself (cf.\ Definition~\ref{def:friends}). Since $\underline{\lambda}$ is fully critical-subcritical, $\lambda^{\setminus i} \le 1$ holds for all $i \in [k]$, thus by Claim~\ref{claim:GIr_is_BP}, the root $r$ is almost surely not $i$-avoiding connected to infinity. Let $v \in V(G_{\infty})$, $v \ne r$ be an arbitrary vertex and let $i \in [k]$ be a color appearing on the unique path connecting $r$ and $v$ in $G_{\infty}$. Then
 \[ v \stackrel{G_{\infty}^{\setminus i}}{ \centernot \longleftrightarrow} r \text{,} \]
 thus $r$ and $v$ are almost surely not $i$-avoiding friends, so they are not friends either. Therefore, $\mathbb{P} \big( \widetilde{\mathcal{C}}^*(r) = \{ r \} \big) = 1$ and thus by Claim~\ref{claim:sum_of_fstar_ell's}, we are done.
 
 \medskip
  
 Next, we prove \eqref{item:not_fully}. Let $\underline{\lambda} \in \mathbb{R}_+^k$ be not fully critical-subcritical. By the definition of $I_{\underline{\lambda}}$, we have $I_{\underline{\lambda}} \ne \emptyset$; let $i_1 \in I_{\underline{\lambda}}$ and $i_2 \in [k] \setminus \{ i_1 \}$, and let $\ell \in \mathbb{N}_+$. Let $F$ be the following rooted, edge-colored tree: the root of $F$ is $r$, which has exactly $\ell$ neighbors $v_1, \ldots, v_{\ell}$, each of which with one further neighbor $w_1, \ldots, w_{\ell}$, respectively, and the color of the edges $ v_1, \ldots, r v_{\ell-1}$ is $i_1$, and the color of the edges $r w_1, \ldots, r w_{\ell}$ is $i_2$.
 It is not difficult to see that if the event
 \[ \{ G_{\infty,2}(r) \simeq F \} \cap \left\{ \forall j \in [\ell]: w_j \stackrel{G_\infty^{\setminus i_1}}{\longleftrightarrow} \infty \right\} \]
 occurs, then $\widetilde{\mathcal{C}}^*(r, G_{\infty}) = \{ r, v_1, \ldots, v_{\ell-1} \}$ (where we identified the vertices of $G_{\infty,2}(r)$ and $F$). All we are left to show is that the above event occurs with positive probability. Clearly, we have
 \begin{equation*}
  \mathbb{P} \left( \big\{ G_{\infty,2}(r) \simeq F \big\} \cap \left\{ \forall j \in [\ell]: w_j \stackrel{G_\infty^{\setminus i_1}}{\longleftrightarrow} \infty \right\} \right) 
  = \big( \theta^{\setminus i_1} \big)^{\ell} \cdot \mathbb{P} \big( G_{\infty,2}(r) \simeq F \big) \text{.}
 \end{equation*}
 Since $i_1 \in I_{\underline{\lambda}}$, we have $\theta^{\setminus i_1} > 0$, and it is not difficult to see that $\mathbb{P} \big( G_{\infty,2}(r) \simeq F \big) > 0$, therefore we are done.
\end{proof}

\section{Explicit formula for \texorpdfstring{$f^*_\infty$}{} and asymptotic behaviour of \texorpdfstring{$f^*_\infty(\varepsilon)$}{}} \label{section:explicit_formulas}

The goal of the section is to give explicit formulas for $f^*_\infty$ and to prove the asymptotic formula stated in Theorem~\ref{thm:barely_supcrit}. We provide two different methods for the calculation of $f^*_\infty$ in Sections~\ref{subsection:f_infty_girls} and~\ref{subsection:f_infty_balazs}. On the one hand, the method presented in Section~\ref{subsection:f_infty_girls} is simpler and works without Assumption~\ref{asm:simplifying_assumption}.
On the other hand, the method of Section~\ref{subsection:f_infty_balazs} only works under Assumption~\ref{asm:simplifying_assumption}, but it can also be used to prove Theorem~\ref{thm:barely_supcrit}.

\medskip 

Our formulas for $f^*_\infty$ use the Lambert $W$ function, which satisfies
\begin{equation} \label{eq:equation_of_LambertW}
 W(z) \mathrm{e}^{W(z)} = z
\end{equation}
for any $z \in \mathbb{C}$. When restricted to real numbers, \eqref{eq:equation_of_LambertW} is solvable if and only if $z \in [-1/\mathrm{e}, +\infty)$. Moreover, it has exactly one solution if $z \in [0, +\infty)$ and has exactly two solutions if $z \in [-1/\mathrm{e}, 0)$. The solutions satisfying $W(z) \ge -1$ form a branch (usually denoted by $W_0$) called the principal branch of the Lambert $W$ function. The Taylor series of the principal branch is
\[ W(x) = \sum_{n \in \mathbb{N}_+} \frac{(-n)^{n-1}}{n!} x^n \text{,} \]
whose radius of convergence is $1/\mathrm{e}$.
Throughout this section, we denote by $W(x)$ the principal branch of the Lambert $W$ function and we always specify which further properties of $W$ we use.

\subsection{Using a system of equations} \label{subsection:f_infty_girls}

Here we give a necessary and sufficient condition for the color intensity parameter vector $\underline{\lambda}$ to have $f^*_{\infty} > 0$, and we also give an explicit formula for $f^*_{\infty}$ in terms of the Lambert $W$ function (cf.\ \eqref{eq:equation_of_LambertW}). Let us emphasize that in this subsection we do not assume that Assumption~\ref{asm:simplifying_assumption} holds for $\underline{\lambda}$.

Let $G_{\infty} \colonequals G_{\infty}(r) \sim \mathcal{G}_{\infty}(\underline{\lambda})$. By Lemma~\ref{lemma:infinitely_many_friends}, we have
\begin{equation} \label{eq:fstar_infty}
 f^*_{\infty} = p^*(\underline{1}) = \mathbb{P} \left(\, \forall i \in [k] \, : \,  r \stackrel{G_{\infty}^{\setminus i} }{\longleftrightarrow} \infty \, \right) \text{.}
\end{equation}

\begin{definition}[Probability of avoiding a color from a fixed set of colors] \label{def:p_I}
 Let $k \in \mathbb{N}_+$, $\underline{\lambda} \in \mathbb{R}_+^k$ and $G_{\infty} \colonequals G_{\infty}(r) \sim \mathcal{G}_{\infty}(\underline{\lambda})$. Let us define
 \[ p_{\emptyset} \colonequals 0 \]
 and
 \[ p_I \colonequals \mathbb{P} \left( \bigcup_{i \in I} \left\{ r \stackrel{G_{\infty}^{\setminus i}}{\longleftrightarrow} \infty \right\} \right) \]
 for any nonempty $I \subseteq [k]$.
\end{definition}

\begin{definition}[Extended color-avoiding type I vectors in ECBP trees] \label{def:extended_typeI_in_ECBP}
 Let $k \in \mathbb{N}_+$, $\underline{\lambda} \in \mathbb{R}_+^k$, and let $G_{\infty} \colonequals G_{\infty}(r) \sim \mathcal{G}_{\infty}(\underline{\lambda})$. The \emph{extended color-avoiding type I vector} of a vertex $v$ in the ECBP tree $G_{\infty}$ is
 \[ \underline{\widehat{t}}{}^*(v) \colonequals \underline{\widehat{t}}{}^*(v, G_{\infty}) \colonequals \bigg( \mathds{1} \Big[ v \stackrel{G_{\infty}^{\setminus i}}{\longleftrightarrow} \infty \Big] \bigg)_{i \in [k]} \text{.} \]
\end{definition}

\begin{definition}[Probability of extended type I vectors in ECBP trees] \label{def:pstarhat_gamma}
 Let $k \in \mathbb{N}_+$, $\underline{\lambda} \in \mathbb{R}_+^k$, and let $G_{\infty}(r) \sim \mathcal{G}_{\infty}(\underline{\lambda})$. Given any $\underline{\gamma} \in \left\{ 0,1 \right\}^{k}$, let us denote
 \[ \widehat{p}^{\, *}(\underline{\gamma}) \colonequals \widehat{p}^{\, *}(\underline{\gamma},\underline{\lambda}) \colonequals \mathbb{P} \Big( \, \underline{\widehat{t}}{}^*(r)=\underline{\gamma} \, \Big) \text{.} \]
\end{definition}

By the definition of $\widehat{p}^{\, *}(\underline{\gamma})$, clearly
\begin{equation} \label{eq:p_I_expressed_with_pstar}
 p_{I} = \sum_{\text{\scriptsize \begin{tabular}{c} $\underline{\gamma} \in \{ 0,1 \}^k:$ \\ $\sum_{i \in I} \gamma_i \ge 1$ \end{tabular}}} \widehat{p}^{\, *}(\underline{\gamma})
\end{equation}
holds for any nonempty $I \subseteq [k]$.

To determine $f^*_{\infty}$, first note that by Proposition~\ref{prop:char_of_giant}, we can assume that $\underline{\lambda} \in \mathbb{R}_+^k$ is fully supercritical (otherwise $f^*_{\infty} = 0$). Then $I_{\underline{\lambda}} = [k]$, thus $p^*(\underline{\gamma}) = \widehat{p}^{\, *}(\underline{\gamma})$ holds for any $\underline{\gamma} \in \left\{ 0,1 \right\}^{[k]}$. By \eqref{eq:fstar_infty} and \eqref{eq:p_I_expressed_with_pstar}, and by the inclusion-exclusion formula, we obtain
\begin{equation} \label{eq:explicit_fstar_infty}
 f^*_{\infty} = p^*(\underline{1}) = \sum_{I \subseteq [k]} (-1)^{|I|} (1 - p_I) \text{,}
\end{equation}
so it is enough to give an explicit formula for the probabilities $p_I$ for all $I \subseteq [k]$.

\begin{lemma}[Implicit equation for $p_I$] \label{lemma:implicit_equation}
Let $k \in \mathbb{N}_+$, $\underline{\lambda} \in \mathbb{R}_+^k$ and $G_{\infty}(r) \sim \mathcal{G}_{\infty}(\underline{\lambda})$. Then
\[ p_I = 1 - \exp\left(-\sum_{j \in I} \lambda_j p_{I \setminus \{j\}}- p_{I} \sum _{j \in [k]\setminus I} \lambda_j \right) \]
holds for any nonempty $I \subseteq [k]$.
\end{lemma}
\begin{proof} 
 By the definition of $p_I$, the independence of the branches, and the definition of $\widetilde{N}_{\underline{s}}(v)$, we obtain
 \begin{multline*}
  1 - p_I = \mathbb{P} \left( \bigcap_{i \in I} \left\{ r \stackrel{G_{\infty}^{\setminus i} }{\centernot{\longleftrightarrow}} \infty \right\} \right) \\
  = \prod_{i \in I} \mathbb{P} \left( \forall v \in \widetilde{N}_{(i)}(r) ~ \forall j \in I \setminus \{ i \}: ~ \left\{ v \stackrel{G_{\infty}^{\setminus j} }{\centernot{\longleftrightarrow}} \infty \right\} \right) \\
  \cdot \prod_{i \in [k] \setminus I} \mathbb{P} \left( \forall v \in \widetilde{N}_{(i)}(r) ~ \forall j \in I: ~ \left\{ v \stackrel{G_{\infty}^{\setminus j} }{\centernot{\longleftrightarrow}} \infty \right\} \right) \text{,}
 \end{multline*}
 where $(i)$ denotes the color string consisting of the color $i$ only. Since the number of $i$-colored children (i.e., the number of children that are connected to $r$ by an edge of color $i$) has distribution $\mathrm{POI}(\lambda_i)$, its generating function is $\exp \big( - \lambda_i (1-z) \big)$ for any $i \in [k]$. Thus we obtain
 \[ 1 - p_I = \exp \left( - \sum_{i \in I} \lambda_i p_{I \setminus \{ i \}} \right) \cdot \exp \left( - p_I \sum_{i \in [k] \setminus I} \lambda_i \right) \text{.} \]
\end{proof}

Let $k \in \mathbb{N}_+$, $\underline{\lambda} \in \mathbb{R}_+^k$ and consider the following system of equations. 
\begin{equation} \label{eq:implicit_system}
 \begin{cases}
 x_{\emptyset}=0 \\
 x_I = 1 - \exp \left(- \sum_{j \in I} \lambda_j x_{I \setminus \{j\}} - x_{I} \sum_{j \in [k] \setminus I} \lambda_j \right) & \forall I \subseteq [k] , I \neq \emptyset \text{.}
 \end{cases}
\end{equation}

By Lemma~\ref{lemma:implicit_equation}, we have that $(x_I)_{I \subseteq [k]} = (p_I)_{I \subseteq [k]}$ is a solution of this system, but note that this is not a unique solution, for instance $(x_I)_{I \subseteq [k]} = \underline{0}$ is also a solution.

Observe that if $\underline{x}^* = (x^*_I)_{I \subseteq [k]}$ is a solution to the system of equations~\eqref{eq:implicit_system}, then $x^*_I < 1$ holds for any $I \subseteq [k]$.

\begin{definition}[Relevant solution] \label{def:relevant_solution}
 We say that a solution $\underline{x}^*=(x^*_I)_{I \subseteq [k]}$ of the system of equations~\eqref{eq:implicit_system} is relevant if $0 < x^*_I < 1$ holds for any nonempty $I \subseteq [k]$. Otherwise $\underline{x}^*$ is said to be irrelevant.
\end{definition}

The next lemma gives a necessary and sufficient condition for $(p_I)_{I \subseteq [k]}$ to be a relevant solution.

\begin{lemma}[The relevance of the probabilistic solution] \label{lemma:relevant_solution_of_implicit_system}
 Let $k \in \mathbb{N}_+$ and $\underline{\lambda} \in \mathbb{R}_+^k$. The probability vector $(p_I)_{I \subseteq [k]}$ is a relevant solution of the system of equations~\eqref{eq:implicit_system} if and only if $\underline{\lambda}$ is fully supercritical.
\end{lemma}
\begin{proof}
 Let $G_{\infty} \colonequals G_{\infty}(r) \sim \mathcal{G}_{\infty}(\underline{\lambda})$.
 
 First, assume that $\underline{\lambda}$ is fully supercritical. Then for any nonempty $I \subseteq [k]$, we obtain that
 \[ p_I = \mathbb{P} \left( \bigcup_{i \in I} \left\{ r \stackrel{G_{\infty}^{\setminus i}}{\longleftrightarrow} \infty \right\} \right) \ge \mathbb{P} \left( r \stackrel{G_{\infty}^{\setminus k}}{\longleftrightarrow} \infty \right) > 0 \]
 holds, hence $(p_I)_{I \subseteq [k]}$ is indeed a relevant solution.
  
 Now assume that $\underline{\lambda}$ is not fully supercritical, i.e., there exists a color $i \in [k]$ such that $\lambda^{\setminus i} \le 1$. Then
 \[ p_{\{ i \}} = \mathbb{P} \left( r \stackrel{G_{\infty}^{\setminus i}}{\longleftrightarrow} \infty \right) = 0 \text{,} \]
 thus $(p_I)_{I \subseteq [k]}$ is not a relevant solution.
\end{proof}

\begin{theorem}[Uniqueness of relevant solution] \label{thm:unique_relevant_solution}
 Let $k \in \mathbb{N}_+$ and $\underline{\lambda} \in \mathbb{R}_+^k$. If $\underline{\lambda}$ is fully supercritical, then $(p_I)_{I \subseteq [k]}$ is a unique relevant solution.
\end{theorem}
\begin{proof}
By the definition of a relevant solution, it is enough to show that the equation
\begin{equation} \label{eq:implicit_equation_I}
 x_I = 1 - \exp \left(- \sum_{j \in I} \lambda_j p_{I \setminus \{j\}} - x_{I} \sum_{j \in [k] \setminus I} \lambda_j \right)
\end{equation}
has a unique solution in $(0,1)$, namely $p_I$, for any nonempty $I \subseteq [k]$. We prove this by induction on $|I|$.

First, assume that $|I| = 1$, i.e., $I = \{ i \}$ for some $i \in [k]$. Then we need to show
\begin{equation} \label{eq:implicit_equation_i}
 x_{\{ i \}} = 1 - \exp \left( - x_{\{ i \}} \sum_{j \in [k] \setminus \{ i \}} \lambda_j \right) \text{.}
\end{equation}
Now the left-hand side of~\eqref{eq:implicit_equation_i} is linear in the variable $x_{\{ i \}}$, while its right-hand side is exponential in $x_{\{ i \}}$, and it is well-known that such equations have at most two solutions. As we saw earlier, $x_{\{ i \}} = 0$ and $x_{\{ i \}} = p_{\{ i \}}$ are both solutions, and by Lemma~\ref{lemma:relevant_solution_of_implicit_system}, $p_{\{ i \}} > 0$. Thus, \eqref{eq:implicit_equation_i} has indeed a unique solution in $(0,1)$.

Now let $I \subseteq [k]$ with $|I| \ge 2$ and assume that \eqref{eq:implicit_equation_I} holds for any $I' \subseteq [k]$ with $|I'| < |I|$.

\textit{Case 1:} $I \ne [k]$.

Again, the left-hand side of~\eqref{eq:implicit_equation_I} is linear, while its right-hand side is exponential in $x_I$, so there are at most two solutions. Clearly, $p_{I}$ is a solution, and by Lemma~\ref{lemma:relevant_solution_of_implicit_system}, it is in $(0,1)$. In addition, in $x_{I} = 0$ the left-hand side of~\eqref{eq:implicit_equation_I} is clearly 0, while its right-hand side is
\[ 1 - \exp \left(- \sum_{i \in I} \lambda_i p_{I \setminus \{i\}} \right) > 1  - \mathrm{e}^0 = 0 \]
so by the continuity of the linear and exponential functions, the other solution of~\eqref{eq:implicit_equation_I} is negative.

\textit{Case 2:} $I = [k]$.

Then equation~\eqref{eq:implicit_equation_I} becomes
\begin{equation*}
 x_{[k]} = 1 - \exp \left(- \sum_{i \in [k]} \lambda_i p_{[k] \setminus \{ i \}} \right) \text{,}
\end{equation*}
which clearly has a unique solution. By Lemma~\ref{lemma:implicit_equation}, this solution is $p_I$ and it is in $(0,1)$.
\end{proof}

As is clear from the proof of Theorem~\ref{thm:unique_relevant_solution}, we should solve the equations of the system~\eqref{eq:implicit_system} in a nondecreasing order according to the size of their indices $I$. Using the Lambert $W$ function (cf.\ \eqref{eq:equation_of_LambertW}), we obtain the following.

\begin{corollary} \label{cor:recursive_solution}
Let $k \in \mathbb{N}_+$ and $\underline{\lambda} \in \mathbb{R}_+^k$. If $\underline{\lambda}$ is fully supercritical, then 
\begin{gather*}
 p_{I} = 1 + \frac{W \left( - \left( \sum_{i \in [k] \setminus I} \lambda_i \right) \exp \left(- \sum_{i \in I} \lambda_i p_{I \setminus \{i\}} - \sum_{i \in [k] \setminus I} \lambda_i \right) \right)}{\sum_{i \in [k] \setminus I} \lambda_i} \text{,}
\end{gather*}
for any nonempty $I \subsetneq [k]$, and
\[ p_{[k]} = 1 - \exp \left( - \sum_{i \in [k]} \lambda_i p_{[k] \setminus \{i\}} \right) \text{.} \]
\end{corollary}

Substituting these into equation~\eqref{eq:explicit_fstar_infty}, we get the required formula for $f^*_{\infty}$. We note that using the values $p_I$ (and the help of computer software capable of symbolic computations), we can also prove the convergence statement of Theorem~\ref{thm:barely_supcrit} in the case when $k \in \{ 2, 3, 4 \}$ and obtain the values $C(2) = 4$, and $C(3) = 32$, and $C(4)=624$, as stated in Theorem~\ref{thm:barely_supcrit}. Note that in Section \ref{subsection:f_infty_balazs}, we introduce an alternative method with which the convergence statement of Theorem~\ref{thm:barely_supcrit} can be proved for any integer $k\geq 2$ and at the end of Section~\ref{subsection:f_infty_balazs}, we obtain the same values for $C(k)$ when $k \in \{ 2, 3, 4 \}$ using this alternative approach. 

\subsection{Using generating functions} \label{subsection:f_infty_balazs}

The aim of Section~\ref{subsection:f_infty_balazs} is to give an explicit formula for $f^*_\infty$  and to prove Theorem~\ref{thm:barely_supcrit}. 
By Proposition~\ref{prop:char_of_giant}, we can assume that the color intensity parameter vector $\underline{\lambda} \in \mathbb{R}_+^k$ (where $k \in \mathbb{N}_+$) is fully supercritical. In this section, we also assume that Assumption~\ref{asm:simplifying_assumption} holds for $\underline{\lambda}$.

In order to find an explicit formula for $f^*_\infty$, we first calculate the joint generating function of the random variables $\big| \widetilde{R}_{\underline{s}}(r) \big|, \, \underline{s} \in S_h$ for any $h \in \{ 0, 1, \ldots, k-2\}$. Note that by Claim~\ref{claim:GIr_is_BP}, we have that $\big( R_{\infty,d}^{\mathrm{set}(\underline{s})} \big)_{d \in \mathbb{N}}$ is a branching process with offspring distribution $\mathrm{POI}(\lambda_{\mathrm{set}(\underline{s})})$.

\begin{definition}[Generating function of Borel distribution] \label{def:gen_function_of_Borel}
 Let $k \in \mathbb{N}_+$, $\underline{\lambda} \in \mathbb{R}_+^k$, $G_{\infty} \colonequals G_{\infty}(r) \sim \mathcal{G}_{\infty}(\underline{\lambda})$, $h \in \{ 0, 1, \ldots, k-1 \}$ and let $\underline{s} \in S_h$. Let us define
 \[ F_{\underline{s}} : [0,1) \to [0,1) \qquad z \mapsto \mathbb{E} \left( z^{\left|\rule{0cm}{6pt}\right. \! R \left(\rule{0cm}{6pt}\right. \! G_{\infty}^{\mathrm{set}(\underline{s})} \! \left.\rule{0cm}{6pt}\right) \! \left.\rule{0cm}{6pt}\right|} \right) \text{,} \]
 i.e. the generating function of the total number of individuals in the branching process $\big( R_{\infty,d}^{\mathrm{set}(\underline{s})} \big)_{d \in \mathbb{N}}$.
 
 Let us denote $F_{\underline{s}}(1_-) \colonequals \lim_{z \to 1_-} F_{\underline{s}}(z)$, i.e. the probability of the event that the branching process $\big( R_{\infty,d}^{\mathrm{set}(\underline{s})} \big)_{d \in \mathbb{N}}$ dies out.
\end{definition}

Note that $ F_{\underline{s}}$ can be expressed using the Lambert $W$ function (introduced at the beginning of Section~\ref{section:explicit_formulas}) as
 \begin{equation} \label{eq:lambert_F_s}
  F_{\underline{s}}(z)=\frac{-W\left(-\lambda_{\mathrm{set}(\underline{s})} \mathrm{e}^{-\lambda_{\mathrm{set}(\underline{s})}}z \right)}{\lambda_{\mathrm{set}(\underline{s})}}, \qquad z \in [0,1) \text{.}
 \end{equation}

Note that if $\mathrm{set}(\underline{s}) = \mathrm{set}(\underline{s}')$ holds for some color strings $\underline{s}, \underline{s}' \in S_h$ with $h \in \{ 0, 1, \ldots, k-2 \}$, then $F_{\underline{s}} \equiv F_{\underline{s}'}$.

Also note that if $\underline{\lambda}$ is fully supercritical and Assumption~\ref{asm:simplifying_assumption} holds, then $F_{\underline{s}}(1_-) = 1$ for any $\underline{s} \in S_h$ with $h \in \{ 0, 1, \ldots, k-2 \}$, and $F_{\underline{s}}(1_-) < 1$ for any $\underline{s} \in S_{k-1}$.

\begin{definition}[Joint generating function of the number of vertices reachable by different color chronologies in an ECBP tree] \label{def:joint_gen_function_of_Rs}
 Let $k \in \mathbb{N}_+$, $\underline{\lambda} \in \mathbb{R}_+^k$, $G_{\infty}(r) \sim \mathcal{G}_{\infty}(\underline{\lambda})$ and $h \in \{ 0, 1, k-2 \}$. Let us define
 \[ \Phi_h: [0,1)^{S_h} \to [0,1) \qquad \left( z_{\underline{s}} \right)_{\underline{s} \in S_h} \mapsto \mathbb{E} \left( \prod_{\underline{s} \in S_h} z_{\underline{s}}^{ \big| \widetilde{R}_{\underline{s}}(r) \big| } \right) \text{.} \]
 Let us denote
 \[ \Phi_h \left( \left( 1_- \right)_{\underline{s} \in S_h} \right) \colonequals \lim_{\substack{\forall \underline{s} \in S_h:\\ z_{\underline{s}} \to 1_-}} \Phi_h \left( \left( z_{\underline{s}} \right)_{\underline{s} \in S_h} \right) \text{.} \]
\end{definition}

It follows from the definition of $S_0$ (cf.\ Definition~\ref{def:color_strings}) that $\Phi_0(z) = z$.

Note that if Assumption~\ref{asm:simplifying_assumption} holds, then $\mathbb{P} \big( \big| \widetilde{R}_{\underline{s}}(r) \big| < +\infty \big) = 1$ for any $\underline{s} \in S_h$ with $h \in \{ 0, 1, \ldots, k-2 \}$. Therefore,
\[ \Phi_h \left( \left( 1_- \right)_{\underline{s} \in S_h} \right) = 1 \]
for any $h \in \{ 0, 1, \ldots, k-2 \}$.

Now we give a recursive formula for $\Phi_{h+1}$ in terms of $\Phi_h$ for any $h \in \{ 0, 1, \ldots, k-3 \}$. Note that this formula can be viewed as ``explicit'' since it can be written in terms of elementary functions and the Lambert $W$ function (cf.\ \eqref{eq:lambert_F_s}).

\begin{lemma}[Recursion for the joint generating function of the number of vertices reachable by different color chronologies] \label{lemma:joint_gen_function_recursion}
 Let $k \in \mathbb{N}_+$, $\underline{\lambda} \in \mathbb{R}_+^k$ and $G_{\infty}(r) \sim \mathcal{G}_{\infty}(\underline{\lambda})$. Then
 \[ \Phi_{h+1} \left( \left( z_{\underline{s}} \right)_{\underline{s} \in S_{h+1}} \right) = \Phi_h \left( \left( \prod_{i \in [k] \setminus \mathrm{set}(\underline{s})} \exp \Big( \lambda_i \cdot \big( F_{ \underline{s}i} (z_{\underline{s}i}) -1 \big) \Big) \right)_{\underline{s} \in S_h} \right) \]
 holds for any $h \in \{ 0, 1, \ldots, k-3 \}$.
\end{lemma}

In words: one has to plug $\prod_{i \in [k] \setminus \mathrm{set}(\underline{s})} \exp \big( \lambda_i \cdot \big( F_{ \underline{s}i} (z_{\underline{s}i}) -1 \big) \big)$ in place of the variable $z_{\underline{s}}$ for every $\underline{s} \in S_h$ in the function $\Phi_h \big( (z_{\underline{s}})_{\underline{s} \in S_h} \big)$.

\begin{proof}[Proof of Lemma~\ref{lemma:joint_gen_function_recursion}]
 First, note that each color string of length $h+1$ can be uniquely written as the concatenation of a color string of length $h$ and a color not appearing in this color string. Therefore, by the tower rule,
 \begin{multline*}
  \Phi_{h+1} \left( \left( z_{\underline{s}} \right)_{\underline{s} \in S_{h+1}} \right) = \Phi_{h+1} \left( \left( z_{\underline{s}i} \right)_{\underline{s} \in S_h, i \in [k] \setminus \mathrm{set}(\underline{s})} \right) \\
  = \mathbb{E} \left( \mathbb{E} \left(\rule{0cm}{25pt}\right. \prod_{\substack{\underline{s} \in S_h, \\ i \in [k] \setminus \mathrm{set}(\underline{s})}} z_{\underline{s}i}^{ \big| \widetilde{R}_{\underline{s}i}(r) \big|} \left.\rule{0cm}{25pt}\middle|\right. \Big( \widetilde{R}_{\underline{s}}(r) \Big)_{\underline{s} \in S_h} \left.\rule{0cm}{25pt}\right) \right) \text{.}
 \end{multline*}
 By the independence of the different branches of an ECBP tree,
 \[ \mathbb{E} \left(\rule{0cm}{25pt}\right. \prod_{\substack{\underline{s} \in S_h, \\ i \in [k] \setminus \mathrm{set}(\underline{s})}} z_{\underline{s}i}^{ \big| \widetilde{R}_{\underline{s}i}(r) \big|} \left.\rule{0cm}{25pt}\middle|\right. \Big( \widetilde{R}_{\underline{s}}(r) \Big)_{\underline{s} \in S_h} \left.\rule{0cm}{25pt}\right) = \prod_{\substack{\underline{s} \in S_h, \\ i \in [k] \setminus \mathrm{set}(\underline{s})}} \mathbb{E} \bigg( z_{\underline{s}i}^{ \big| \widetilde{R}_{\underline{s}i}(r) \big|} \biggm| \widetilde{R}_{\underline{s}}(r) \bigg) \]
 holds.
 
 Let $\underline{s} \in S_h$ and $i \in [k] \setminus \mathrm{set}(\underline{s})$. Using  that for any $v \in \widetilde{R}_{\underline{s}}(r)$, the function $z \mapsto \exp \big( \lambda_i \cdot \big( F_{\underline{s}i}(z) - 1 \big) \big)$ is the generating function of the cardinality of the set of descendants of $v$ belonging to $\widetilde{R}_{\underline{s}i}(r)$, we obtain
 \[ \mathbb{E} \bigg( z_{\underline{s}i}^{ \big| \widetilde{R}_{\underline{s}i}(r) \big|} \biggm| \widetilde{R}_{\underline{s}}(r) \bigg) = \bigg( \exp \Big( \lambda_i \cdot \big( F_{\underline{s}i}(z) - 1 \big) \Big) \bigg)^{\big| \widetilde{R}_{\underline{s}}(r) \big|} \text{.} \]
 
 Therefore and by the definition of the function $\Phi_h$, we obtain 
 \begin{multline*}
  \Phi_{h+1} \left( \left( z_{\underline{s}} \right)_{\underline{s} \in S_{h+1}} \right) = \mathbb{E} \left( \prod_{\substack{\underline{s} \in S_h, \\ i \in [k] \setminus \mathrm{set}(\underline{s})}} \bigg( \exp \Big( \lambda_i \cdot \big( F_{\underline{s}i}(z) - 1 \big) \Big) \bigg)^{\big| \widetilde{R}_{\underline{s}}(r) \big|} \right) \\
  = \mathbb{E} \left( \prod_{\underline{s} \in S_h} \left( \prod_{i \in [k] \setminus \mathrm{set}(\underline{s})} \exp \Big( \lambda_i \cdot \big( F_{\underline{s}i}(z) - 1 \big) \Big) \right)^{\big| \widetilde{R}_{\underline{s}}(r) \big|} \right) \\
  = \Phi_h \left( \left( \prod_{i \in [k] \setminus \mathrm{set}(\underline{s})} \exp \Big( \lambda_i \cdot \big( F_{ \underline{s}i} (z_{\underline{s}i}) -1 \big) \Big) \right)_{\underline{s} \in S_h} \right) \text{.}
 \end{multline*}
\end{proof}


\begin{definition}[Color strings of length $k-1$ avoiding color $i$] \label{def:color_strings_without_i}
 Let $k \in \mathbb{N}_+$. For any $i \in [k]$, let
 \[ S_{k-1}^{\setminus i} \colonequals \big\{ \underline{s} \in S_{k-1} : \mathrm{set}(\underline{s}) = [k] \setminus \{ i \} \big\} \text{.} \]
\end{definition}

\begin{theorem}[Explicit formula for $f^*_\infty$] \label{thm:explicit_formula}
 Let $k \in \mathbb{N}_+$, let $\underline{\lambda} \in \mathbb{R}^k_+$ for which Assumption~\ref{asm:simplifying_assumption} holds, and let $G_{\infty} \colonequals G_{\infty}(r) \sim \mathcal{G}_{\infty}(\underline{\lambda})$. Then
 \[ f^*_\infty = \sum_{J \subseteq [k]} (-1)^{|J|} \cdot \Phi_{k-2} \left( \left( \prod_{i \in [k] \setminus \mathrm{set}(\underline{s})} q_{L_J,\underline{s}i} \right)_{\underline{s} \in S_{k-2}} \right) \text{,} \]
 where we define
 \[ L_J \colonequals \bigcup_{j \in J} S_{k-1}^{\setminus j} \]
 for any $J \subseteq [k]$, and we define
 \[ q_{L,\underline{s}i} \colonequals
  \begin{cases}
   \exp \Big( \lambda_i \cdot \big( F_{\underline{s}i} (1_-) - 1 \big) \Big) & \text{if $\underline{s}i \in L$,} \\
   1 & \text{if $\underline{s}i \notin L$}
  \end{cases} \]
 for any $L \subseteq S_{k-1}$, $\underline{s} \in S_{k-2}$ and $i \in [k] \setminus \mathrm{set}(\underline{s})$.
\end{theorem}

\begin{proof}
 If $v \in V(G_\infty)$ and $i \in [k]$, let us denote by $\widetilde{N}^*_{(i)}(v)$ the set of children of $v$ that are connected to $v$ by an edge of color $i$.

 By Claim~\ref{claim:GIr_is_BP} and using that the offsprings of any vertex in a branching process tree span a branching process tree with the same offspring distribution, and by the definition of $\Phi_{k-2}$, for any $L \subseteq S_{k-1}$, we obtain
 \begin{multline} \label{eq:extinct_packages_gen_function}
  \mathbb{P} \left( \bigcap_{\underline{s}' \in L} \bigg\{ \, \Big| \widetilde{R}_{\underline{s}'}(r) \Big| < + \infty \, \bigg\} \right) \\
  = \! \mathbb{E} \! \left(\rule{0cm}{30pt}\right. \!\!\! \mathbb{E} \! \left( \! \prod_{\underline{s} \in S_{k-2}} \, \prod_{v \in \widetilde{R}_{\underline{s}}(r)} \prod_{\substack{i \in [k] \setminus \mathrm{set}(\underline{s}): \\ \underline{s}i \in L}} \, \prod_{u \in \widetilde{N}^*_{(i)}(v)} \!\! \mathds{1} \bigg[ \, \Big| R \Big( u, G_{\infty}^{\mathrm{set}(\underline{s}i)} \! \Big) \Big| < + \infty \bigg] \middle| \, \mathcal{F}_{r,k-2} \! \right) \!\!\! \left.\rule{0cm}{30pt}\right) \\
  = \mathbb{E} \left( \prod_{\underline{s} \in S_{k-2}} \, \prod_{v \in \widetilde{R}_{\underline{s}}(r)} \prod_{\substack{i \in [k] \setminus \mathrm{set}(\underline{s}): \\ \underline{s}i \in L}} \, \prod_{u \in \widetilde{N}^*_{(i)}(v)} F_{\underline{s}i} (1_-) \, \right) \\
  = \mathbb{E} \left( \prod_{\underline{s} \in S_{k-2}} \, \prod_{v \in \widetilde{R}_{\underline{s}}(r)} \prod_{\substack{i \in [k] \setminus \mathrm{set}(\underline{s}): \\ \underline{s}i \in L}} \, \exp\left( \lambda_i (F_{\underline{s}i} (1_-) - 1) \right) \, \right) \\
  = \Phi_{k-2} \left( \left( \prod_{i \in [k] \setminus \mathrm{set}(\underline{s})} q_{L,\underline{s}i} \right)_{\underline{s} \in S_{k-2}} \right) \text{.}
 \end{multline}
 
 Now by~\eqref{eq:fstar_infty}, Assumption~\ref{asm:simplifying_assumption}, the inclusion-exclusion formula, the definition of $L_J$, and~\eqref{eq:extinct_packages_gen_function}, we obtain
 \begin{multline*}
  f^*_\infty = \mathbb{P} \left(\, \forall j \in [k] \, : \,  r \stackrel{G_{\infty}^{\setminus j} }{\longleftrightarrow} \infty \, \right) \\
  = \mathbb{P} \left(\rule{0cm}{27pt}\right. \, \bigcap_{j \in [k]} \, \left(\, \bigcap_{\underline{s}' \in S_{k-1}^{\setminus j}} \bigg\{ \, \Big| \widetilde{R}_{\underline{s}'}(r) \Big| < + \infty \, \bigg\} \right)^{\; c} \left.\rule{0cm}{27pt}\right) \\
  = \sum_{J \subseteq [k]} (-1)^{|J|} \cdot \mathbb{P} \left( \, \bigcap_{j \in J} \, \bigcap_{\underline{s}' \in S_{k-1}^{\setminus j}} \bigg\{ \, \Big| \widetilde{R}_{\underline{s}'}(r) \Big| < + \infty \, \bigg\} \right) \\
  = \sum_{J \subseteq [k]} (-1)^{|J|} \cdot \mathbb{P} \left( \, \bigcap_{\underline{s}' \in L_J} \bigg\{ \, \Big| \widetilde{R}_{\underline{s}'}(r) \Big| < + \infty \, \bigg\} \right) \\
  = \sum_{J \subseteq [k]} (-1)^{|J|} \cdot \Phi_{k-2} \left( \left( \prod_{i \in [k] \setminus \mathrm{set}(\underline{s})} q_{L_J,\underline{s}i} \right)_{\underline{s} \in S_{k-2}} \right) \text{.}
 \end{multline*}
\end{proof}

\begin{proof}[Proof of Theorem~\ref{thm:barely_supcrit}]
 Let $\varepsilon \ge 0$ and let $G_{\infty}^{\varepsilon} \colonequals G_{\infty}^{\varepsilon}(r) \sim \mathcal{G}_{\infty} \big( \underline{\lambda}(\varepsilon) \big)$, where $\underline{\lambda}(\varepsilon) = \left(\frac{1+\varepsilon}{k-1}, \ldots, \frac{1+\varepsilon}{k-1} \right) \in \mathbb{R}_+^k$. Since $\underline{\lambda}(\varepsilon)$ is homogeneous,
 \[ \theta^{\setminus k} (\varepsilon) \colonequals \mathbb{P} \left( r \xleftrightarrow{ \! \left(\rule{0cm}{6pt}\right. \! G_{\infty}^{\varepsilon} \!\! \left.\rule{0cm}{6pt}\right)^{ \! \setminus k} \! } \infty \right) = \mathbb{P} \left( r \xleftrightarrow{ \! \left(\rule{0cm}{6pt}\right. \! G_{\infty}^{\varepsilon} \!\! \left.\rule{0cm}{6pt}\right)^{ \! \setminus i} \! } \infty \right) \]
 holds for any $i \in [k]$.
 
 Note that if $\varepsilon < \frac{1}{k-2}$, then Assumption~\ref{asm:simplifying_assumption} holds, so by~\eqref{eq:fstar_infty} and the disjointness of the sets $\widetilde{N}_{k-1}^{\setminus i}(r, G_{\infty}^{\varepsilon})$ for all $i \in [k]$, we obtain
 \begin{multline} \label{eq:fstar_infty_epsilon}
  f^*_{\infty}(\varepsilon) = \mathbb{P} \left( \, \forall i \in [k] \; \exists v \in \widetilde{N}_{k-1}^{\setminus i}(r, G_{\infty}^{\varepsilon}): \; v \xleftrightarrow{ \! \left(\rule{0cm}{6pt}\right. \! G_{\infty}^{\varepsilon} \!\! \left.\rule{0cm}{6pt}\right)^{ \! \setminus i} \! } \infty \right) \\
  = \mathbb{E} \left( \prod_{i \in [k]} \left( 1 - \left( 1 - \theta^{\setminus k}(\varepsilon) \right)^{\left|\rule{0cm}{6pt}\right. \! \widetilde{N}_{k-1}^{\setminus i} \left(\rule{0cm}{6pt}\right. r, G_{\infty}^{\varepsilon} \left.\rule{0cm}{6pt}\right) \left.\rule{0cm}{6pt}\right|} \right) \right) \text{.}
 \end{multline}
 
 Clearly,
 \[ \lambda^{\setminus k}(\varepsilon) \colonequals \sum_{i \in [k-1]} \lambda_j (\varepsilon) = 1 + \varepsilon \]
 holds, thus by Corollary~3.19 of~\cite{vdH}, we have
 \[ \lim_{\varepsilon \to 0_+} \frac{\theta^{\setminus k}(\varepsilon)}{\varepsilon} = 2 \text{.} \]
 Hence,
 \begin{equation} \label{eq:fstar_infty_epsilon_over_epsilon_to_the_k}
  \lim_{\varepsilon \to 0_+} \frac{f^*_{\infty}(\varepsilon)}{\varepsilon^k} = 2^k \lim_{\varepsilon \to 0_+} \frac{f^*_{\infty}(\varepsilon)}{\big( \theta^{\setminus k}(\varepsilon) \big)^k} \text{,}
 \end{equation}
 and now we show
 \begin{equation} \label{eq:f_infty_derivative}
  \lim_{\varepsilon \to 0_+} \frac{f^*_{\infty}(\varepsilon)}{\big( \theta^{\setminus k}(\varepsilon) \big)^k} = \mathbb{E} \left( \prod_{i \in [k]} \left| \widetilde{N}_{k-1}^{\setminus i} (r, G_{\infty}^0) \right| \right) \text{.} 
 \end{equation}
 
 By Assumption~\ref{asm:simplifying_assumption}, there exists $\varepsilon_0 \ge 0$ such that
 \[ \mathbb{E} \left( \prod_{i \in [k]} \left| \widetilde{N}_{k-1}^{\setminus i} (r, G_{\infty}^{\varepsilon_0}) \right| \right) < + \infty \text{.} \]
 Then by the monotone coupling the random variables $\left| \widetilde{N}_{k-1}^{\setminus i} (r, G_{\infty}^{\varepsilon}) \right|$ for all $\varepsilon \in [0, \varepsilon_0]$, and by Bernoulli's inequality, we obtain
 \begin{multline*}
  \prod_{i \in [k]} \left( 1 - \left( 1 - \theta^{\setminus k}(\varepsilon) \right)^{\left|\rule{0cm}{6pt}\right. \! \widetilde{N}_{k-1}^{\setminus i} \left(\rule{0cm}{6pt}\right. r, G_{\infty}^{\varepsilon} \left.\rule{0cm}{6pt}\right) \left.\rule{0cm}{6pt}\right|} \right) \le \prod_{i \in [k]} \left| \widetilde{N}_{k-1}^{\setminus i} (r, G_{\infty}^{\varepsilon}) \right| \\
  \le \prod_{i \in [k]} \left| \widetilde{N}_{k-1}^{\setminus i} (r, G_{\infty}^{\varepsilon_0}) \right| \text{.}
 \end{multline*}
 Thus by the dominated convergence theorem, by the choice of $\varepsilon_0$, and by L'Hospital's rule,
 \begin{multline*}
  \lim_{\varepsilon \to 0_+} \mathbb{E} \left( \prod_{i \in [k]} \frac{1 - \left( 1 - \theta^{\setminus k}(\varepsilon) \right)^{\left|\rule{0cm}{6pt}\right. \! \widetilde{N}_{k-1}^{\setminus i} \left(\rule{0cm}{6pt}\right. r, G_{\infty}^{\varepsilon} \left.\rule{0cm}{6pt}\right) \left.\rule{0cm}{6pt}\right|}}{\theta^{\setminus k}(\varepsilon)} \right) \\
  = \mathbb{E} \left( \lim_{\varepsilon \to 0_+} \prod_{i \in [k]} \frac{1 - \left( 1 - \theta^{\setminus k}(\varepsilon) \right)^{\left|\rule{0cm}{6pt}\right. \! \widetilde{N}_{k-1}^{\setminus i} \left(\rule{0cm}{6pt}\right. r, G_{\infty}^{\varepsilon} \left.\rule{0cm}{6pt}\right) \left.\rule{0cm}{6pt}\right|}}{\theta^{\setminus k}(\varepsilon)} \right) \\
  = \mathbb{E} \left(\rule{0cm}{25pt}\right. \lim_{\varepsilon \to 0_+} \left(\rule{0cm}{20pt}\right. \mathds{1} \bigg[ \forall i \in [k]:~ \left| \widetilde{N}_{k-1}^{\setminus i} (r, G_{\infty}^{\varepsilon}) \right| = \left| \widetilde{N}_{k-1}^{\setminus i} (r, G_{\infty}^0) \right| \bigg] \\
  \cdot \prod_{i \in [k]} \frac{1 - \left( 1 - \theta^{\setminus k}(\varepsilon) \right)^{\left|\rule{0cm}{6pt}\right. \! \widetilde{N}_{k-1}^{\setminus i} \left(\rule{0cm}{6pt}\right. r, G_{\infty}^{\varepsilon} \left.\rule{0cm}{6pt}\right) \left.\rule{0cm}{6pt}\right|}}{\theta^{\setminus k}(\varepsilon)} \\
  + \mathds{1} \bigg[ \exists i \in [k]:~ \left| \widetilde{N}_{k-1}^{\setminus i} (r, G_{\infty}^{\varepsilon}) \right| \ne \left| \widetilde{N}_{k-1}^{\setminus i} (r, G_{\infty}^0) \right| \bigg] \\
  \cdot \prod_{i \in [k]} \frac{1 - \left( 1 - \theta^{\setminus k}(\varepsilon) \right)^{\left|\rule{0cm}{6pt}\right. \! \widetilde{N}_{k-1}^{\setminus i} \left(\rule{0cm}{6pt}\right. r, G_{\infty}^{\varepsilon} \left.\rule{0cm}{6pt}\right) \left.\rule{0cm}{6pt}\right|}}{\theta^{\setminus k}(\varepsilon)} \left.\rule{0cm}{20pt}\right) \left.\rule{0cm}{25pt}\right) \\
  = \mathbb{E} \left( \lim_{\varepsilon \to 0_+} \prod_{i \in [k]} \frac{1 - \left( 1 - \theta^{\setminus k}(\varepsilon) \right)^{\left|\rule{0cm}{6pt}\right. \! \widetilde{N}_{k-1}^{\setminus i} \left(\rule{0cm}{6pt}\right. r, G_{\infty}^0 \left.\rule{0cm}{6pt}\right) \left.\rule{0cm}{6pt}\right|}}{\theta^{\setminus k}(\varepsilon)} \right) \\
  = \mathbb{E} \left( \lim_{\theta^{\setminus k}(\varepsilon) \to 0_+} \prod_{i \in [k]} \frac{1 - \left( 1 - \theta^{\setminus k}(\varepsilon) \right)^{\left|\rule{0cm}{6pt}\right. \! \widetilde{N}_{k-1}^{\setminus i} \left(\rule{0cm}{6pt}\right. r, G_{\infty}^0 \left.\rule{0cm}{6pt}\right) \left.\rule{0cm}{6pt}\right|}}{\theta^{\setminus k}(\varepsilon)} \right) \\
  = \mathbb{E} \left( \prod_{i \in [k]} \left| \widetilde{N}_{k-1}^{\setminus i} (r, G_{\infty}^0) \right| \right) \text{.}
 \end{multline*}
 Thus by~\eqref{eq:fstar_infty_epsilon}, we obtain that \eqref{eq:f_infty_derivative} holds.
 
 Therefore, by~\eqref{eq:fstar_infty_epsilon_over_epsilon_to_the_k}, we obtain
 \[ \lim_{\varepsilon \to 0_+} \frac{f^*_{\infty}(\varepsilon)}{\varepsilon^k} = 2^k \cdot \mathbb{E} \left( \prod_{i \in [k]} \left| \widetilde{N}_{k-1}^{\setminus i} (r, G_{\infty}^0) \right| \right) \equalscolon C(k) \text{.} \]
 
 Now we want to express this constant $C(k)$ using the generating function $\Phi_{k-2}$,
 where we take the underlying parameter $\underline{\lambda}$ to be $\underline{\lambda} = \underline{\lambda}(0) = \left( \frac{1}{k-1}, \ldots, \frac{1}{k-1} \right)$.
 
 Since the sets $\widetilde{N}_{\underline{s}}(r, G_{\infty}^0)$ for all $\underline{s} \in S_{k-1}$ are disjoint, it follows that
 \begin{multline*}
  \mathbb{E} \left( \prod_{i \in [k]} \left| \widetilde{N}_{k-1}^{\setminus i} (r, G_{\infty}^0) \right| \right) = \mathbb{E} \left( \prod_{i \in [k]} \, \sum_{\underline{s} \in S_{k-1}^{\setminus i}} \left| \widetilde{N}_{\underline{s}} (r, G_{\infty}^0) \right| \right) \\
  = \sum_{\big( \underline{s}_i \in S_{k-1}^{\setminus i} \big)_{i \in [k]}} \mathbb{E} \left( \prod_{i \in [k]} \left| \widetilde{N}_{\underline{s}_i} (r, G_{\infty}^0) \right| \right)
 \end{multline*}
 holds.
 
 Let $\underline{s}_i \in S_{k-1}^{\setminus i}$ for all $i \in [k]$. Since the random variables $\big| \widetilde{N}_{\underline{s}_i} (r, G_{\infty}^0) \big|$ for all $i \in [k]$ are conditionally independent given $\big( \big| \widetilde{R}_{\underline{s}}(r, G_{\infty}^0) \big| \big)_{\underline{s} \in S_{k-2}}$, and their conditional distribution is Poisson with parameter $\big| \widetilde{R}_{\underline{s}}(r, G_{\infty}^0) \big|/(k-1)$, we obtain
 \[ \mathbb{E} \left( \prod_{i \in [k]} \left| \widetilde{N}_{\underline{s}_i} (r, G_{\infty}^0) \right| \right) = \frac{1}{(k-1)^k} \mathbb{E} \left( \prod_{i \in [k]} \left| \widetilde{R}_{\underline{s}_i^-} (r, G_{\infty}^0) \right| \right) \text{.} \]
 
 Let us define the multivariate moment generating function
 \[ \varphi \left( (x_{\underline{s}})_{\underline{s} \in S_{k-2}} \right) \colonequals \Phi_{k-2} \left( (\mathrm{e}^{x_{\underline{s}}})_{\underline{s} \in S_{k-2}} \right) = \mathbb{E} \left( \exp \left(\rule{0cm}{18pt}\right. \sum_{\underline{s} \in S_{k-2}} \left| \widetilde{R}_{\underline{s}} (r, G_{\infty}^0) \right| \cdot x_{\underline{s}} \left.\rule{0cm}{18pt}\right) \right) \text{.} \]
 Since $\varphi$ is the multivariate moment generating function of the random vector $\left( \left| \widetilde{R}_{\underline{s}} (r, G_{\infty}^0) \right|\right)_{\underline{s} \in S_{k-2}}$,  we obtain
 \[ \left. \frac{\partial^k}{\partial x_{{\underline{s}_1^-}} \ldots \partial x_{{\underline{s}_k^-}}} \varphi \left( (x_{\underline{s}} )_{\underline{s} \in S_{k-2}} \right) \right|_{(x_{\underline{s}})_{\underline{s} \in S_{k-2}} = \underline{0}} = \mathbb{E} \left( \prod_{i \in [k]} \left| \widetilde{R}_{\underline{s}_i^-} (r, G_{\infty}^0) \right| \right) \text{.} \]
 
 Hence,
 \begin{multline*}
  C(k) = 2^k \cdot \mathbb{E} \left( \prod_{i \in [k]} \left| \widetilde{N}_{k-1}^{\setminus i} (r, G_{\infty}^0) \right| \right) = 2^k \cdot \! \sum_{\big( \underline{s}_i \in S_{k-1}^{\setminus i} \big)_{i \in [k]}} \! \mathbb{E} \left( \prod_{i \in [k]} \left| \widetilde{N}_{\underline{s}_i} (r, G_{\infty}^0) \right| \right) \\
  = \frac{2^k}{(k-1)^k} \cdot \! \sum_{\big( \underline{s}_i \in S_{k-1}^{\setminus i} \big)_{i \in [k]}} \! \mathbb{E} \left( \prod_{i \in [k]} \left| \widetilde{R}_{\underline{s}_i^-} (r, G_{\infty}^0) \right| \right) \\
  = \frac{2^k}{(k-1)^k} \cdot \! \sum_{\big( \underline{s}_i \in S_{k-1}^{\setminus i} \big)_{i \in [k]}} \! \left. \frac{\partial^k}{\partial x_{{\underline{s}_1^-}} \ldots \partial x_{{\underline{s}_k^-}}} \Phi_{k-2} \left( (\mathrm{e}^{x_{\underline{s}}} )_{\underline{s} \in S_{k-2}} \right) \right|_{(x_{\underline{s}})_{\underline{s} \in S_{k-2}} = \underline{0}} \text{,}
 \end{multline*}
 which can be expressed explicitly using an iterated composition of the Lambert $W$ function and elementary functions since the same is true for $\Phi_{k-2}$ by Lemma~\ref{lemma:joint_gen_function_recursion}.
 
 \medskip
 
 In particular, when $k=2$, then $\Phi_0(\mathrm{e}^x) = \mathrm{e}^x$, thus $C(2)=4$.
 
 When $k=3$, then by Lemma~\ref{lemma:joint_gen_function_recursion},
 \[ \Phi_1(\mathrm{e}^{x_1}, \mathrm{e}^{x_2}, \mathrm{e}^{x_3}) = \prod_{i \in [3]} \exp \left( \frac{1+\varepsilon}{2} \big( F_{(i)}(\mathrm{e}^{x_i}) - 1 \big) \right) \text{,} \]
 where by~\eqref{eq:lambert_F_s},
 \[ F_{(i)}(z) = \frac{- W (- \frac{1}{2} \cdot \mathrm{e}^{- \frac{1}{2}} \cdot z)}{\frac{1}{2}} \]
 for any $i \in [3]$.
 Using the identities
 \[ \frac{\mathrm{d}}{\mathrm{d}z} W(z) = \frac{W(z)}{z \big( 1 + W(z) \big)} \qquad \text{if $z \notin \{ 0, -1/\mathrm{e} \}$} \]
 and
 \[ W(z \mathrm{e}^z) = z \qquad \text{if $z \ge 0$} \]
 we can obtain $C(3)=32$.
 
 Similarly (but with significantly more work), we obtain the value $C(4)=624$. Recall that we obtained the same values for $C(k)$ when $k \in \{ 2, 3, 4 \}$ using a different approach at the end of Section~\ref{subsection:f_infty_girls}, too.
 \end{proof}

\section{The case of two colors} \label{section:two_colors}

In Theorem~\ref{thm:conv_of_f_ell_G_n}, we already proved that the empirical color-avoiding connected component size densities converge. In this section, we give explicit formulas for their limit when there are only two colors.

\begin{theorem}
 Let $\lambda_{\text{red}}, \lambda_{\text{blue}} \in \mathbb{R}_+$, $\ell \in \mathbb{N}_+$, and let $G_{\infty} \colonequals G_{\infty}(r) \sim \mathcal{G}_{\infty}\big( (\lambda_{\text{red}}, \lambda_{\text{blue}}) \big)$ be a two-colored ECBP tree. Then
 \begin{multline} \label{eq:two_color_formula}
  f^*_{\ell} = \mathds{1}[\ell = 1] \cdot \widehat{p}^{\, *} \big( (0,0) \big) \\
  + \sum_{m \in \mathbb{N}_+} \Bigg( \binom{m-1}{\ell - 1} (\theta_{\text{blue}})^{\ell - 1} (1 - \theta_{\text{blue}})^{m - \ell} \\
  \cdot \frac{\exp \big( - \lambda_{\text{red}} \cdot (1 - \theta_{\text{red}}) \cdot m \big) \cdot \big(\lambda_{\text{red}} (1 - \theta_{\text{red}}) \cdot m \big)^{m-1}}{m!} \Bigg) \cdot \widehat{p}^{\, *} \big( (1,0) \big) \\
  + \sum_{m \in \mathbb{N}_+} \Bigg( \binom{m-1}{\ell - 1} (\theta_{\text{red}})^{\ell - 1} (1 - \theta_{\text{red}})^{m - \ell} \\
  \cdot \frac{\exp \big( - \lambda_{\text{blue}} \cdot (1 - \theta_{\text{blue}}) \cdot m \big) \cdot \big(\lambda_{\text{blue}} (1 - \theta_{\text{blue}}) \cdot m \big)^{m-1}}{m!} \Bigg) \cdot \widehat{p}^{\, *} \big( (0,1) \big) \text{,}
 \end{multline}
 where
 \begin{align*}
  \widehat{p}^{\, *} \big( (0,0) \big) &= \exp \Big( - \lambda_{\text{red}} - W \big( - \lambda_{\text{red}} \mathrm{e}^{- \lambda_{\text{red}}} \big) - \lambda_{\text{blue}} - W \big( - \lambda_{\text{blue}} \mathrm{e}^{- \lambda_{\text{blue}}} \big) \Big) \text{,} \\
  \widehat{p}^{\, *} \big( (1,0) \big) &= - \widehat{p}^{\, *} \big( (0,0) \big) - \frac{W \big( - \lambda_{\text{red}} \mathrm{e}^{- \lambda_{\text{red}}} \big)}{\lambda_{\text{red}}} \text{,} \\
  \widehat{p}^{\, *} \big( (0,1) \big) &= - \widehat{p}^{\, *} \big( (0,0) \big) - \frac{W \big( - \lambda_{\text{blue}} \mathrm{e}^{- \lambda_{\text{blue}}} \big)}{\lambda_{\text{blue}}} \text{.}
 \end{align*}
\end{theorem}
\begin{proof} 
 By the definition of $f^*_{\ell}$ and by the total law of probability, we obtain
 \[ f^*_{\ell} = \sum_{\underline{\gamma} \in \{0,1\}^2} \mathbb{P} \Big( \, \big| \widetilde{\mathcal{C}}^*(r) \big| = \ell \Bigm| \underline{\widehat{t}}{}^*(r) = \underline{\gamma} \, \Big) \cdot \widehat{p}^{\, *}(\underline{\gamma}) \text{.} \]
 
 Now let us look at the summands separately. Note that any red-avoiding path contains only blue edges, and similarly, any blue-avoiding path contains only red edges.
 
 If $\underline{\gamma} = (1,1)$, then by Lemma~\ref{lemma:infinitely_many_friends}, we obtain
 \[ \mathbb{P} \Big( \, \big| \widetilde{\mathcal{C}}^*(r) \big| = \ell \Bigm| \underline{\widehat{t}}{}^*(r) = (1,1) \, \Big) \le \mathbb{P} \Big( \, \big| \widetilde{\mathcal{C}}^*(r) \big| < \infty \Bigm| \underline{t}^*(r) = (1,1) \, \Big) = 0 \text{.} \]
 
 If $\underline{\gamma} = (0,0)$, then clearly,
 \[ \mathbb{P} \Big( \, \big| \widetilde{\mathcal{C}}^*(r) \big| = \ell \Bigm| \underline{\widehat{t}}{}^*(r) = (0,0) \, \Big) = \mathds{1}[\ell = 1] \text{.} \]
 
 If $\underline{\gamma} = (1,0)$, then conditional on $\underline{\widehat{t}}{}^*(r) = (1,0)$, the branching process $\big( \big| R_{\infty, d}^{\text{red}} \big| \big)_{d \in \mathbb{N}}$ is subcritical and has offspring distribution $\mathrm{POI} \big( \lambda_{\text{red}} (1-\theta_{\text{red}}) \big)$ by Theorem 3.15 of \cite{vdH}, where $\theta_{\text{red}} \colonequals \theta^{\setminus \text{blue}}$ and $\theta_{\text{blue}} \colonequals \theta^{\setminus \text{red}}$ (cf.\ Definition~\ref{def:theta_minus_i}). Thus the number of vertices of $\big( \big| R_{\infty, d}^{\text{red}} \big| \big)_{d \in \mathbb{N}}$ has distribution $\mathrm{BOREL} \big( \lambda_{\text{red}} (1-\theta_{\text{red}}) \big)$. In addition, each vertex of $\big( R_{\infty, d}^{\text{red}} \big)_{d \in \mathbb{N}}$ is of type either $(0,0)$ or $(1,0)$ independently of each other with probabilities 
 \[ \frac{\widehat{p}^{\, *} \big( (0,0) \big)}{\widehat{p}^{\, *} \big( (0,0) \big) + \widehat{p}^{\, *} \big( (1,0) \big)} = \frac{(1 - \theta_{\text{blue}})(1 - \theta_{\text{red}})}{(1 - \theta_{\text{red}})} = 1 - \theta_{\text{blue}} \]
 and
 \[ \frac{\widehat{p}^{\, *} \big( (1,0) \big)}{\widehat{p}^{\, *} \big( (0,0) \big) + \widehat{p}^{\, *} \big( (1,0) \big)} = \theta_{\text{blue}} \text{,} \]
 respectively. Thus by the total law of probability,
 \begin{multline*}
  \mathbb{P} \Big( \, \big| \widetilde{\mathcal{C}}^*(r) \big| = \ell \Bigm| \underline{\widehat{t}}{}^*(r) = (1,0) \, \Big) \\
  = \sum_{m \in \mathbb{N}_+} \bigg( \mathbb{P} \Big( \, \big| \widetilde{\mathcal{C}}^*(r) \big| = \ell \Bigm| \underline{\widehat{t}}{}^*(r) = (1,0), \, \big| R_{\infty}^{\text{red}} \big| = m \, \Big) \\
  \cdot \mathbb{P} \Big( \, \big| R_{\infty}^{\text{red}} \big| = m \Bigm| \, \underline{\widehat{t}}{}^*(r) = (1,0) \Big) \bigg) \\
  = \sum_{m \in \mathbb{N}_+} \Bigg( \binom{m-1}{\ell - 1} (\theta_{\text{blue}})^{\ell - 1} (1 - \theta_{\text{blue}})^{m - \ell} \\
  \cdot \frac{\exp \big( -\lambda_{\text{red}} \cdot (1 - \theta_{\text{red}}) \cdot m \big) \cdot \big(\lambda_{\text{red}} (1 - \theta_{\text{red}}) \cdot m \big)^{m-1}}{m!} \Bigg) \text{.}
 \end{multline*}
 
 If $\underline{\gamma} = (0,1)$, then similarly,
 \begin{multline*}
  \mathbb{P} \Big( \, \big| \widetilde{\mathcal{C}}^*(r) \big| = \ell \Bigm| \underline{\widehat{t}}{}^*(r) = (0,1) \, \Big) \\
  = \sum_{m \in \mathbb{N}_+} \Bigg( \binom{m-1}{\ell - 1} (\theta_{\text{red}})^{\ell - 1} (1 - \theta_{\text{red}})^{m - \ell} \\
  \cdot \frac{\exp \big( -\lambda_{\text{blue}} \cdot (1 - \theta_{\text{blue}}) \cdot m \big) \cdot \big(\lambda_{\text{blue}} (1 - \theta_{\text{blue}}) \cdot m \big)^{m-1}}{m!} \Bigg) \text{.}
 \end{multline*}
 
 Therefore \eqref{eq:two_color_formula} holds. By~\eqref{eq:p_I_expressed_with_pstar}, we obtain
 \begin{align*}
  \widehat{p}^{\, *} \big( (0,0) \big) &= 1 - p_{\{ \text{red}, \text{blue} \}} \text{,} \\
  \widehat{p}^{\, *} \big( (1,0) \big) &= p_{\{ \text{red}, \text{blue} \}} - p_{\{ \text{blue} \}} \text{,} \\
  \widehat{p}^{\, *} \big( (0,1) \big) &= p_{\{ \text{red}, \text{blue} \}} - p_{\{ \text{red} \}} \text{,}
 \end{align*}
 and by Corollary~\ref{cor:recursive_solution}, the stated result follows.
\end{proof}

\section{Open questions} \label{section_open_questions}

Let us conclude our paper with a list of open questions and conjectures that we propose for future research.


\begin{conjecture} \label{conjecture_no_assumption}
 Theorem~\ref{thm:conv_of_f_ell_G_n} holds without Assumption~\ref{asm:simplifying_assumption}.
\end{conjecture}

Let us note that the recent paper~\cite{lichevschapira} resolved this conjecture in the affirmative, but we decided to keep it in our paper for reference.

In Section~\ref{section:explicit_formulas}, we gave explicit, recursive formulas for the asymptotic density $f^*_{\infty}$ of the giant color-avoiding connected component, and in Section~\ref{section:two_colors}, we gave a formula for the asymptotic density $f^*_{\ell}$ of the set of vertices that belong to a color-avoiding connected component of size $\ell$ for any $\ell \in \mathbb{N}_+$ in the case of $k=2$ colors.

\begin{question}
 Is there an explicit formula for $f^*_{\ell}$ for any $\ell \in \mathbb{N}_+$ and for $k \ge 3$ colors?
\end{question}

It is well-known (see~\cite{aldous1997brownian} and the references therein) that the cardinality $\big| \mathcal{U}_{\text{max}}(G_n) \big|$ of the largest component in an (uncolored) critical Erd\H{o}s--R\'{e}nyi graph $G_n$ with $n$ vertices is of order $n^{2/3}$. It follows from Theorem~\ref{thm:conv_of_f_ell_G_n} and Proposition~\ref{prop:char_of_giant} that if $\lambda^{\setminus i} = 1$ for some $i \in [k]$, then
\[ \frac{1}{n} \max_{v \in V(G_n)} \Big| \widetilde{\mathcal{C}}(v, G_n) \Big| \; \stackrel{\mathbb{P}}{\longrightarrow} \; 0, \qquad n \to \infty \text{.} \]

\begin{question} \label{question_max_comp}
 Let $k \in \mathbb{N}_+$, $\underline{\lambda} \in \mathbb{R}_+^k$ and let $G_n \sim \mathcal{G}_n([n],\underline{\lambda})$ for all $n \in \mathbb{N}_+$. What is the order of magnitude of $\max_{v \in [n]} \big| \widetilde{\mathcal{C}}(v, G_n) \big|$ if $\lambda^{\setminus i} = 1$ for some $i \in [k]$, and in particular, if $\lambda^{\setminus i} = 1$ for all $i \in [k]$?
\end{question}

Let us note that the recent paper~\cite{lichevschapira} contains some results in the vein of our Question~\ref{question_max_comp}, see Theorem~1.1 and Proposition~1.3 of~\cite{lichevschapira}.

We know from Theorem~\ref{thm:barely_supcrit} that $C(k)$ is an integer if $k \in \{ 2, 3, 4 \}$.

\begin{question}
 Is $C(k)$ an integer for all $k \in \mathbb{N}_+$?
\end{question}

Our results pertain to one of the most fundamental edge-colored random graph models, i.e., the edge-colored Erd\H{o}s--R\'{e}nyi graph. A natural, more heterogeneous model of edge-colored random graphs is the following generalization of the well-known configuration model (see Section 1.3 in \cite{vdH2}). Let $k \in \mathbb{N}_+$ and let $p: \mathbb{N}_+^k \to [0,1]$ satisfying $\sum_{\underline{d} \in \mathbb{N}_+^k} p(\underline{d}) = 1$. Let us consider a sequence of edge-colored graphs $G_n$ where the edge-colored degree sequence satisfies
\begin{multline*}
 \hspace{-8pt} \lim_{n \to \infty} \frac{1}{n} \Big| \big\{ v \in V(G_n) : \text{the number of half-edges of color $i$ incident to $v$ is $d_i$} \big\} \Big| \\
 = p(d_1, \ldots, d_k)
\end{multline*}
and the half-edges of the same color are matched according to the rules of the configuration model.

\begin{question}
 Does Theorem~\ref{thm:conv_of_f_ell_G_n} have an analogue in the above described edge-colored configuration model? Is there an explicit formula for the asymptotic density of the largest color-avoiding connected component?
\end{question}

Color-avoiding percolation was first introduced for vertex-colored graphs, see~\cites{krause2016hidden,krause2017color}. The definition of color-avoiding connectivity in vertex-colored graphs is ambiguous. The two most interesting versions are as follows. In the first version, two vertices are color-avoiding connected in a vertex-colored graph if there exists a path between them whose internal vertices are not of color $i$ for any color $i$. Whereas in the second version, two vertices are color-avoiding connected if there exists a path between them using any colors and either at least one of the vertices is of color $i$ or there exists a path between them whose vertices are not of color $i$ for any color $i$. In~\cite{molontay2019complexity}, we showed that these two definitions can yield problems with highly different computational complexity.

\begin{question}
 Does Theorem~\ref{thm:conv_of_f_ell_G_n} have an analogue in a vertex-colored Erd\H{o}s--R\'{e}nyi random graph, i.e., in an Erd\H{o}s--R\'{e}nyi random graph whose vertices are randomly colored? Is there an explicit formula for the asymptotic density of the largest color-avoiding connected component?
\end{question}

\section*{Acknowledgement}
The research of Roland Molontay, Bal\'{a}zs R\'{a}th, and Kitti Varga was partially supported by grant NKFI-FK-123962 of NKFI (National Research, Development and Innovation Office). The research of Panna Fekete and Bal\'{a}zs R\'{a}th  was partially supported by the ERC Synergy under Grant No.\ 810115 - DYNASNET. The research of Panna T\'{i}mea Fekete was prepared with the professional support of the Doctoral Student Scholarship Program of the Co-operative Doctoral Program of the Ministry of Innovation and Technology financed from the National Research, Development and Innovation Fund and the National Research, Development and Innovation Office under Grant No.   2018-2.1.1.-UK\_GYAK-2018-00008.  The research of Roland Molontay was partially supported by NKFIH 142169 research grant. 

\appendix

\section*{Appendix}

\begin{proof}[Proof of~\ref{prop:benjamini_schramm}]
 Since Assumption~\ref{asm:simplifying_assumption} holds, Claim~\ref{claim:finiteness_of_rhor} implies that $\widetilde{G}_{\infty,k-2}(r)$ is almost surely finite.
 
 Now we show that for any $d \in \mathbb{N}$, the sequence $\big( G^{\text{uc}}_{n,d}(v_1,v_2) \big)_{n=1}^{\infty}$ converges in distribution to two i.i.d.\ copies of $G^{\text{uc}}_{\infty,d}(r)$ in the sense that
 \begin{equation} \label{eq:conv_of_uncolored_balls_of_radius_d}
  \lim_{n \to \infty} \mathbb{P} \left( G^{\text{uc}}_{n,d}(v_1,v_2) \simeq F^{\text{uc}}_1 \oplus F^{\text{uc}}_2 \right) = \prod_{j \in [2]} \mathbb{P} \left( G^{\text{uc}}_{\infty,d}(r) \simeq F^{\text{uc}}_j \right)
 \end{equation}
 for any pair of fixed finite, rooted, edge-colored (multi)graphs $F_1$ and $F_2$. By Claim~\ref{claim:GI_is_ER}, we have $G^{\text{uc}}_n \sim \mathcal{G}_n([n],\lambda^{\text{uc}})$. Thus by Exercise 2.14 of~\cite{vdHStF}, $G^{\text{uc}}_n$ converges in probability in the local weak sense (cf.\ Definition 2.9(b) of~\cite{vdHStF}) to an uncolored Poisson branching process tree in $\mathcal{G}_\infty(\lambda^{\text{uc}})$. This and Exercise 2.10 of~\cite{vdHStF} imply that the two-rooted graph sequence $\big( G^{\text{uc}}_{n}(v_1, v_2) \big)_{n=1}^{\infty}$ converges in distribution in the local weak sense to two i.i.d.\ copies of $G^{\text{uc}}_{\infty}(r)$, i.e.,~\eqref{eq:conv_of_uncolored_balls_of_radius_d} holds.

 Next, we show that for any $d \in \mathbb{N}$, the two-rooted edge-colored graph sequence $\big( G_{n,d}(v_1,v_2) \big)_{n=1}^{\infty}$ converges in distribution to two i.i.d.\ copies of $G_{\infty,d}(r)$ in the sense that
 \begin{equation} \label{eq:conv_of_edgecolored_balls_of_radius_d_again} 
  \lim_{n \to \infty} \mathbb{P} \big( G_{n,d}(v_1, v_2) \simeq F_1 \oplus F_2 \big) = \prod_{j \in [2]} \mathbb{P} \big( G_{\infty,d}(r) \simeq F_j \big) \text{,}
 \end{equation}
 for any pair of fixed finite, rooted, edge-colored graphs $F_1$ and $F_2$. It follows from the definitions of ECER graphs and ECBP trees (cf.\ Definitions~\ref{def:ECER} and~\ref{def:ECBP}) that
 \begin{multline*}
  \lim_{n \to \infty} \mathbb{P} \left( G_{n,d}(v_1, v_2) \simeq F_1 \oplus F_2 \, \middle| \, G^{\text{uc}}_{n,d}(v_1, v_2) \simeq F^{\text{uc}}_1 \oplus F^{\text{uc}}_2 \right) \\
  = \prod_{j \in [2]} \mathbb{P} \left( G_{\infty,d}(r) \simeq F_j \, \middle| \, G^{\text{uc}}_{\infty,d}(r) \simeq F^{\text{uc}}_j \right) \text{.}
 \end{multline*}
 This and~\eqref{eq:conv_of_uncolored_balls_of_radius_d} imply~\eqref{eq:conv_of_edgecolored_balls_of_radius_d_again}.
 
 For a number $d \in \mathbb{N}$ and for a finite, rooted, edge-colored graph $H$ with depth at most $d$ (i.e.\ the distance of each vertex from the root is at most $d$), let $\mathcal{B}_d(H)$ denote the set of possible pairwise non-isomorphic realizations of the graph $G_{\infty,d}(r)$, i.e.\ of the branching process tree $G_{\infty}(r)$ up to depth $d$, for which the event $\big\{ \widetilde{G}_{\infty,k-2}(r) \simeq H \big\}$ holds. Note that the event $\big\{ \widetilde{G}_{\infty,k-2}(r) \simeq H \big\}$ is measurable with respect to the sigma-algebra generated by $G_{\infty,d}(r)$, thus
 \[ \Big\{ \widetilde{G}_{\infty,k-2}(r) \simeq H \Big\} = \biguplus_{F \in \mathcal{B}_d(H)} \big\{ G_{\infty,d}(r) \simeq F \big\} \text{,} \]
 i.e., the event $\big\{ \widetilde{G}_{\infty,k-2}(r) \simeq H \big\}$ is the disjoint union of the events $\big\{ G_{\infty,d}(r) \simeq F \big\}$ for all $F \in \mathcal{B}_d(H)$. Therefore,
 \begin{equation} \label{eq:switching_from_colordist_to_edgedist_no1}
  \prod_{j \in [2]} \mathbb{P} \Big( \widetilde{G}_{\infty,k-2}(r) \simeq H_j \Big) = \sum_{\substack{F_1 \in \mathcal{B}_d(H_1) \\ F_2 \in \mathcal{B}_d(H_2)}} \, \prod_{j \in [2]} \mathbb{P} \big( G_{\infty,d}(r) \simeq F_j \big) \text{.}
 \end{equation}
 
 Let
 \[ E_d(v_1, v_2, G_n) \colonequals \big\{ \text{$G_{n,d}(v_1, v_2)$ is a forest with two components} \big\} \text{.} \]
 Then
 \begin{multline*}
  \Big\{ \widetilde{G}_{n,k-2}(v_1, v_2) \simeq H_1 \oplus H_2, ~ E_d(v_1, v_2, G_n) \Big\} \\
  = \biguplus_{\substack{F_1 \in \mathcal{B}_d(H_1) \\ F_2 \in \mathcal{B}_d(H_2)}} \big\{ G_{n,d}(v_1, v_2) \simeq F_1 \oplus F_2 \big\}
 \end{multline*}
 for any number $d \in \mathbb{N}$ and for any pair of fixed finite, rooted, edge-colored graphs $H_1$ and $H_2$ with depth at most $d$. Note that by~\eqref{eq:conv_of_edgecolored_balls_of_radius_d_again}, we have 
 \[ \lim_{n \to \infty} E_d(v_1, v_2, G_n) = 1 \text{.} \]
 Thus,
 \begin{multline} \label{eq:switching_from_colordist_to_edgedist_no2}
  \lim_{n \to \infty} \mathbb{P} \Big( \widetilde{G}_{n,k-2}(v_1, v_2) \simeq H_1 \oplus H_2 \Big) \\
  = \lim_{n \to \infty} \mathbb{P} \Big( \widetilde{G}_{n,k-2}(v_1, v_2) \simeq H_1 \oplus H_2, E_d(v_1, v_2, G_n) \Big) \\
  = \lim_{n \to \infty} \sum_{\substack{F_1 \in \mathcal{B}_d(H_1) \\ F_2 \in \mathcal{B}_d(H_2)}} \mathbb{P} \big( G_{n,d}(v_1, v_2) \simeq F_1 \oplus F_2 \big) \text{.}
 \end{multline}
 
 Let $\mathcal{B}_d$ denote the set of pairwise non-isomorphic possible realizations of $G_{\infty,d}(r)$. Clearly,
 \[ \lim_{n \to \infty} \sum_{F_1, F_2 \in \mathcal{B}_d} \mathbb{P} \big( G_{n,d}(v_1, v_2) \simeq F_1 \oplus F_2 \big) = \lim_{n \to \infty} E_d(v_1, v_2, G_n) = 1 \]
 and by~\eqref{eq:conv_of_edgecolored_balls_of_radius_d_again}, we have
 \begin{multline*}
  \sum_{F_1, F_2 \in \mathcal{B}_d} \lim_{n \to \infty} \mathbb{P} \big( G_{n,d}(v_1, v_2) \simeq F_1 \oplus F_2 \big) \\
  = \sum_{F_1, F_2 \in \mathcal{B}_d} \, \prod_{j \in [2]} \mathbb{P} \big( G_{\infty,d}(r) \simeq F_j \big) = 1 \text{,}
 \end{multline*}
 so Scheff\'{e}'s lemma \cite{scheffe} can be applied to obtain
 \begin{equation*}
  \lim_{n \to \infty} \sum_{F_1, F_2 \in \mathcal{B}_d} \left| \ \mathbb{P} \big( G_{n,d}(v_1, v_2) \simeq F_1 \oplus F_2 \big) - \prod_{j \in [2]} \mathbb{P} \big( G_{\infty,d}(r) \simeq F_j \big) \ \right| = 0 \text{.}
 \end{equation*}
 In particular, since $\mathcal{B}_d(H_1), \mathcal{B}_d(H_2) \subseteq \mathcal{B}_d$, we have
 \[ \lim_{n \to \infty} \sum_{\substack{F_1 \in \mathcal{B}_d(H_1), \\ F_2 \in \mathcal{B}_d(H_2)} } \left| \ \mathbb{P} \big( G_{n,d}(v_1, v_2) \simeq F_1 \oplus F_2 \big) - \prod_{j \in [2]} \mathbb{P} \big( G_{\infty,d}(r) \simeq F_j \big) \ \right| = 0 \text{.} \]
 Thus by~\eqref{eq:switching_from_colordist_to_edgedist_no2} and~\eqref{eq:switching_from_colordist_to_edgedist_no1}, we obtain
 \begin{multline*}
  \lim_{n \to \infty} \mathbb{P} \Big( \widetilde{G}_{n,k-2}(v_1, v_2) \simeq H_1 \oplus H_2 \Big) \\
  = \lim_{n \to \infty} \sum_{\substack{F_1 \in \mathcal{B}_d(H_1) \\ F_2 \in \mathcal{B}_d(H_2)}} \mathbb{P} \big( G_{n,d}(v_1, v_2) \simeq F_1 \oplus F_2 \big) \\
  = \sum_{\substack{F_1 \in \mathcal{B}_d(H_1) \\ F_2 \in \mathcal{B}_d(H_2)}} \prod_{j \in [2]} \mathbb{P} \big( G_{\infty,d}(r) \simeq F_j \big) = \prod_{j \in [2]} \mathbb{P} \Big( \widetilde{G}_{\infty,k-2}(r) \simeq H_j \Big) \text{.}
 \end{multline*}
\end{proof}

\label{list_of_symbols_starts_here}

\nomenclature[1A]{$A(v,w)$}{good event for ECER graphs: intersection of the events $A_1^{\setminus i}, A_2^{\setminus i}$, $A_3(v,w)$, $A_4(v,w)$, and all the events $A_5^{\setminus i}(v,w), \ldots, A_9^{\setminus i}(v,w)$ with a supercritical index $i$; see Definition~\ref{def:good_events}}

\nomenclature[1A]{$A_1^{\setminus i}$}{event of an ECER graph having a unique largest $i$-avoiding component; see Definition~\ref{def:good_events}}

\nomenclature[1A]{$A_2^{\setminus i}$}{event of the non-largest $i$-avoiding components of an ECER graph being small; see Definition~\ref{def:good_events}}

\nomenclature[1A]{$A_3(v,w)$}{event of the subgraph of an ECER graph explored from $v$ and $w$ until seeing $k-1$ colors for the first time being a forest with two components; see Definition~\ref{def:good_events}}

\nomenclature[1A]{$A_4(v,w)$}{event of the set of vertices reachable from $v$ or $w$ with at most $k-2$ colors being small in an ECER graph; see Definition~\ref{def:good_events}}

\nomenclature[1A]{$A_5^{\setminus i}(v,w)$}{event of the $i$-avoiding boundary of $\{ v,w \}$ being small in an ECER graph; see Definition~\ref{def:good_events}}

\nomenclature[1A]{$A_6^{\setminus i}(v,w)$}{event of the union of the largest $i$-avoiding components of the unexplored subgraph when exploring with $k-2$ colors from $v$ and $w$ being large in an ECER graph; see Definition~\ref{def:good_events}}

\nomenclature[1A]{$A_7^{\setminus i}(v,w)$}{event of the unexplored subgraph when exploring with $k-2$ colors from $v$ and $w$ having a unique largest $i$-avoiding component in an ECER graph; see Definition~\ref{def:good_events}}

\nomenclature[1A]{$A_8^{\setminus i}(v,w)$}{event of the non-largest $i$-avoiding components of the unexplored subgraph when exploring with $k-2$ colors from $v$ and $w$ being small in an ECER graph; see Definition~\ref{def:good_events}}

\nomenclature[1A]{$A_9^{\setminus i}(v,w)$}{event of the largest $i$-avoiding components of the unexplored subgraph when exploring with $k-2$ colors from $v$ and $w$ being contained in the union of the largest $i$-avoiding components of the original ECER graph; see Definition~\ref{def:good_events}}

\nomenclature[1B]{$\underline{b}(v)$}{boundary vector of $v$, i.e., the vector containing the sizes of $i$-avoiding boundaries $\widetilde{N}^{\setminus i}_{k-1}(v)$ of $v$ for every color $i \in [k]$; see Definition~\ref{def:outer_boundary}}

\nomenclature[1B]{$B(\underline{v})$}{bad event for in ECER graphs: union of the events $B_1(\underline{v})$ and $B_2(\underline{v})$; see Definition~\ref{def:bad_events}}

\nomenclature[1B]{$B_1(\underline{v})$}{union of the events of two vertices being $i$-avoiding connected for the non-supercritical indices $i \in [k] \setminus I_{\underline{\lambda}}$ and for any two different vertices in $\underline{v}$; see Definition~\ref{def:bad_events}}

\nomenclature[1B]{$B_2(\underline{v})$}{union of the events of two vertices being both in a largest $i$-avoiding component but not in the same, or being in the same non-largest $i$-avoiding component, for the supercritical indices $i \in I_{\underline{\lambda}}$ and for any two different vertices in $\underline{v}$; see Definition~\ref{def:bad_events}}

\nomenclature[1B]{$\mathcal{B}_d^{\tau_{\infty}}$}{set of pairwise non-isomorphic possible realizations of an ECBP tree up to depth $d$ with color-avoiding horizon $d-1$; see Definition~\ref{def:stopping_time}}

\nomenclature[1Cb]{$\mathcal{C}(v)$}{set of vertices in the connected component of $v$}

\nomenclature[1Cc]{$\widetilde{\mathcal{C}}(v)$}{color-avoiding connected component of $v$; see Definition~\ref{def:coloravoiding_component}}

\nomenclature[1Cd]{$\widetilde{\mathcal{C}}^*(r)$}{set of friends of the root $r$ (i.e., set of vertices that are $i$-avoiding connected to $r$ possibly through infinity); see Definition~\ref{def:friends}}

\nomenclature[1Da]{$\mathrm{dist}(v,w)$, $\mathrm{dist}(v,W)$}{$\empty$ \\ distance of the vertex $v$ from the vertex $w$ or from the vertex set $W$, respectively; see~\eqref{eq:dist_v_W}}

\nomenclature[1Db]{$D \big( (v_j, \underline{\beta}_j, m_j)_{j \in [2]} \big)$}{$\empty$ \\ data of the occurrence of the event $A_3(v_1,v_2)$ and the boundary vectors of $v_1$ and $v_2$ being $\underline{\beta}_1$ and $\underline{\beta}_2$, respectively, and the sets of vertices reachable from $v_1$ and $v_2$ with at most $k-2$ colors having size $m_1$ and $m_2$, respectively, in an ECER graph; see Definition~\ref{def:data_for_ECER}}

\nomenclature[1Dz]{$D^*(\underline{\beta},m)$}{data of the boundary vector of the root being $\underline{\beta}$ and the set of vertices reachable from the root with at most $k-2$ colors having size $m$ in an ECBP tree; see Definition~\ref{def:data_for_ECBP}}

\nomenclature[1Faa]{$f_{\ell}(G)$}{empirical distribution of vertices in color-avoiding connected components of size $\ell$ in an ECER graph $G$; see Definition~\ref{def:f_ell_G}}

\nomenclature[1Fab]{$f_{\ell,d,F}(G)$}{empirical distribution of vertices whose color-avoiding connected components are of size $\ell$ and whose explored subgraphs of depth $d$ are isomorphic to $F$ in an ECER graph $G$; see Definition~\ref{def:f_ell_d_F_G}}

\nomenclature[1Fb]{$F_{\underline{s}}$}{generating function of the total number of individuals in a branching process with offspring distribution $\mathrm{POI}(\lambda_{\mathrm{set}(\underline{s})})$; see Definition \ref{def:gen_function_of_Borel}}

\nomenclature[1Fc]{$\mathcal{F}_{v,h}$, $\mathcal{F}_{v,w,h}$}{$\empty$ \\ sigma-algebras generated by the edges and non-edges of the subgraphs $\widetilde{G}_h(v)$ or $\widetilde{G}_h(v,w)$; see Definition~\ref{def:exploring_up_to_colordist_h}}

\nomenclature[1Fda]{$f^*_{\ell}$}{probability of the root having exactly $\ell$ friends in an ECBP tree; see Definition~\ref{def:fstar_ell}}

\nomenclature[1Fdb]{$f^*_{\ell,d,F}$}{probability of the root having exactly $\ell$ friends and the ECBP tree up to depth $d$ being isomorphic to $F$; see Definition~\ref{def:stopping_time}}

\nomenclature[1Fe]{$f^*_{\infty}(\varepsilon)$}{the probability of the root having infinitely many friends in an ECBP tree where all $k$ colors have intensity $\frac{1+\varepsilon}{k-1}$; see Definition \ref{def:barely_supercrit}}

\nomenclature[1Ga]{$G \sim \mathcal{G}_n(V, \underline{\lambda})$}{$\empty$ \\ $G$ is an ECER graph on the vertex set $V$ with parameter $n$ and color density parameter vector $\underline{\lambda}$; see Definition~\ref{def:ECER}}

\nomenclature[1Ga]{$G-\underline{E}'$}{edge-colored graph obtained from $G$ by deleting the edges of $\underline{E}'$}

\nomenclature[1Gb]{$G-W$}{(edge-colored) graph obtained from $G$ by deleting the vertices of~$W$}

\nomenclature[1Gb]{$G[W]$}{(edge-colored) subgraph of $G$ induced by the vertex set $W$}

\nomenclature[1Gb]{$G^I$, $G^{\setminus i}$, $G^{\text{uc}}$}{simple, uncolored graphs obtained from the edge-colored multigraph $G$ by deleting all the edges of colors not in $I$ or of color $i$ or none of the edges, respectively, then ignoring the edge-coloring and keeping exactly one edge instead of parallel edges; see Definition~\ref{def:colorsubgraphs}}

\nomenclature[1Gc]{$G_d(v)$, $G_d(v,w)$, $G_{n,d}(v)$, $G_{n,d}(v,w)$}{$\empty$ \\ explored subgraphs of depth $d$ with root $v$ or with roots $v$ and $w$, i.e., rooted or two-rooted subgraphs of the graph $G$, and those of the ECER graph $G_n$, respectively, spanned by the vertices at distance at most $d-1$ from $v$ or from $\{ v, w \}$ plus the edges along with their endpoints that connect these spanned subgraphs to the rest of the graph; see Definition~\ref{def:exploring_up_to_dist_d}}

\nomenclature[1Gd]{$G_1 \oplus G_2$}{disjoint union of the (edge-colored) graphs $G_1$ and $G_2$}

\nomenclature[1Gd]{$G_1 \simeq G_2$}{$G_1$ and $G_2$ are isomorphic}

\nomenclature[1Ge]{$G_\infty(r) \sim \mathcal{G}_{\infty}(\underline{\lambda})$}{$\empty$ \\ $G_{\infty}(r)$ is an ECBP tree with root $r$ and with color intensity parameter vector $\underline{\lambda}$; see Definition~\ref{def:ECBP}}

\nomenclature[1Gf]{$G_{\infty,d}(r)$}{rooted subtree of the ECBP tree $G_{\infty}(r)$ spanned by the vertices at distance at most $d$ from $r$; see Definition~\ref{def:exploring_up_to_dist_d}}

\nomenclature[1Gh]{$G'_d(v)$, $G'_d(v,w)$, $G'_{n,d}(v)$, $G'_{n,d}(v,w)$}{$\empty$ \\ subgraphs of the graph $G$, and those of the ECER graph $G_n$, respectively, that can be obtained by deleting the vertices at distance at most $d-1$ from $v$ or from $\{ v, w \}$; see Definition~\ref{def:exploring_up_to_dist_d}}

\nomenclature[1Gh]{$\widetilde{G}_h(v)$, $\widetilde{G}_h(v,w)$, $\widetilde{G}_{n,h}(v)$, $\widetilde{G}_{n,h}(v,w)$}{$\empty$ \\ rooted or two-rooted subgraphs of the graph $G$, and those of the ECER graph $G_n$, respectively, obtained by exploring from $v$ or from $\{ v, w \}$ until seeing $h+1$ colors for the first time, i.e., the subgraphs spanned by the vertices reachable from $v$ or from $\{ v, w \}$ with at most $h$ colors plus the edges along with their endpoints that connect these spanned subgraphs to the rest of the graph; see Definition~\ref{def:exploring_up_to_colordist_h}}

\nomenclature[1Gi]{$\widetilde{G}_{\infty,h}(r)$}{rooted subtree of the ECBP tree $G_{\infty}(r)$ spanned by the vertices reachable from $r$ with at most $h$ colors plus one more edge; see Definition~\ref{def:exploring_up_to_colordist_h}}

\nomenclature[1Gj]{$\widetilde{G}'_h(v,w)$, $\widetilde{G}'_{n,h}(v,w)$}{$\empty$ \\ subgraph of the graph $G$, and that of the ECER graph $G_n$, respectively, that is left unexplored when exploring from $\{ v,w \}$ with at most $h$ colors; see Definition~\ref{def:exploring_up_to_colordist_h}}

\nomenclature[1I]{$I_{\underline{\lambda}}$}{set of supercritical indices (i.e. of those $i \in [k]$ for which $\lambda^{\setminus i} > 1$); see Definition~\ref{def:set_of_supercrit_indices}}

\nomenclature[1Na]{$N(v)$, $N(W)$}{$\empty$ \\ open neighborhood of the vertex $v$ or that of the vertex set $W$, respectively; see~\eqref{eq:N_W}}

\nomenclature[1Nb]{$\widetilde{N}_h(v)$, $\widetilde{N}_h(W)$}{$\empty$ \\ sets of vertices that can be reached from the vertex $v$ or from the vertex set $W$, respectively, with a path which uses exactly $h$ colors and the color of the last edge is new, but cannot be reached with a path of at most $h-1$ colors; see Definition~\ref{def:outer_boundary}}

\nomenclature[1Nc]{$\widetilde{N}_{k-1}^{\setminus i}(v)$}{$i$-avoiding boundary of $v$, i.e., the set of vertices that can be reached from $v$ with a path which uses all of the colors in $[k] \setminus \{ i \}$ and the color of the last edge does not appear elsewhere on this path, but cannot be reached with a path of at most $k-2$ colors; see Definition~\ref{def:outer_boundary}}

\nomenclature[1Nd]{$\widetilde{N}_{\underline{s}}(v)$}{set of vertices reachable from $v$ with a path on which new colors appear according to the color string $\underline{s}^-$ except the last edge, which is of the color of the last coordinate of $\underline{s}$, and the vertices are not reachable with a path of any shorter color string; see Definition~\ref{def:layers}}

\nomenclature[1Pa]{$p_I$}{probability of the root being $i$-avoiding connected to infinity for some $i \in I$ in an ECBP tree; see Definition~\ref{def:p_I}}

\nomenclature[1Pb]{$p \big( \underline{\gamma}_1, \underline{\gamma}_2 \ \big\vert \ D \big( (v_j, \underline{\beta}_j, m_j)_{j \in [2]} \big) \big)$}{$\empty$ \\ conditional probability of the occurrences of the type II events $T(v_1, v_2, \underline{\gamma}_1)$ and $T(v_2, v_1, \underline{\gamma}_2)$ given the data $D \big( (v_j, \underline{\beta}_j, m_j)_{j \in [2]} \big)$ in an ECER graph; see Definition~\ref{def:p_gamma1_gamma2_cond_D}}

\nomenclature[1Pc]{$p^*(\underline{\gamma})$}{probability of the root having type I vector $\underline{\gamma}$ in an ECBP tree; see Definition~\ref{def:pstar_gamma}}

\nomenclature[1Pd]{$p^* \big( \underline{\gamma} \ \big\vert \ D^*(\underline{\beta},m) \big)$}{$\empty$ \\ conditional probability of the root having type I vector $\underline{\gamma}$ given the data $D^*(\underline{\beta},m)$ in an ECBP tree; see Definition~\ref{def:pstar_gamma_cond_Dstar}}

\nomenclature[1Pe]{$\widehat{p}^{\, *}(\underline{\gamma})$}{probability of the root having extended type I vector $\underline{\gamma}$ in an ECBP tree; see Definition~\ref{def:pstarhat_gamma}}

\nomenclature[1Ra]{$R(W)$}{set of vertices reachable from the vertex set $W$; see~\eqref{eq:R_W}}

\nomenclature[1Rb]{$R_d(v)$, $R_d(W)$, $R_{n,d}(v)$, $R_{n,d}(W)$}{$\empty$ \\ sets of vertices at distance exactly $d$ from the vertex $v$ or from the vertex set $W$ in a graph $G$, and in an ECER graph $G_n$, respectively; see Definition~\ref{def:exploring_up_to_dist_d}}

\nomenclature[1Rc]{$R_{\infty,d}$}{set of vertices at distance exactly $d$ from the root in an ECBP tree $G_{\infty}(r)$; see Definition~\ref{def:exploring_up_to_dist_d}}

\nomenclature[1Rd]{$R_{\infty,d}^I$, $R_{\infty,d}^{\setminus i}$}{$\empty$ \\ sets of vertices at distance exactly $d$ from the root in an ECBP tree $G_{\infty}(r)$ after deleting all the edges of colors not in $I$ or of color $i$, respectively; see Definition~\ref{def:exploring_up_to_dist_d}}

\nomenclature[1Re]{$R^{\le}_d(v)$, $R^{\le}_d(W)$, $R^{\le}_{n,d}(v)$, $R^{\le}_{n,d}(W)$}{$\empty$ \\ sets of vertices at distance at most $d$ from the vertex $v$ or from the vertex set $W$ in a graph $G$, and in an ECER graph $G_n$, respectively; see Definition~\ref{def:exploring_up_to_dist_d}}

\nomenclature[1Rf]{$R^{\le}_{\infty,d}$}{set of vertices at distance at most $d$ from the root in an ECBP tree $G_{\infty}(r)$; see Definition~\ref{def:exploring_up_to_dist_d}}

\nomenclature[1Rg]{$\widetilde{R}_{\underline{s}}(v)$}{set of vertices reachable from $v$ with a path on which new colors appear according to the color string $\underline{s}$, but are not reachable with a path of a shorter color string; see Definition~\ref{def:layers}}

\nomenclature[1Rh]{$\widetilde{R}^{\le}_h(v)$, $\widetilde{R}^{\le}_h(W)$}{$\empty$ \\ sets of vertices reachable from the vertex $v$ or from the vertex set $W$, respectively, with a path of at most $h$ colors; see Definition~\ref{def:layers}}

\nomenclature[1Sa]{$\underline{s}^-$}{color string obtained from $\underline{s}$ by deleting its last coordinate; see Definition~\ref{def:color_string_operations}}

\nomenclature[1Sb]{$S_h$}{set of color strings of length $h$; see Definition~\ref{def:color_strings}}

\nomenclature[1Sb]{$S_{k-1}^{\setminus i}$}{set of color strings of length $k-1$ avoiding color $i$; see Definition~\ref{def:color_strings_without_i}}

\nomenclature[1Sc]{$\mathrm{set}(\underline{s})$}{set of colors appearing in the color string $\underline{s}$; see Definition~\ref{def:color_strings}}

\nomenclature[1Sd]{$\underline{s}i$}{concatenation of color string $\underline{s}$ and color $i$; see Definition~\ref{def:color_string_operations}}

\nomenclature[1Ta]{$\underline{t}(v)$}{color-avoiding type I vector of $v$ in an ECER graph: for each supercritical index $i \in I_{\underline{\lambda}}$ its $i$-th coordinate is 1 if $v$ belongs to a largest $i$-avoiding component and 0 otherwise; see Definition~\ref{def:typeI_in_ECER}}

\nomenclature[1Tb]{$\underline{t}^*(v)$}{color-avoiding type I vector of $v$ in an ECBP tree: for each supercritical index $i \in I_{\underline{\lambda}}$ its $i$-th coordinate is 1 if $v$ is $i$-avoiding connected to infinity and 0 otherwise; see Definition~\ref{def:typeI_in_ECBP}}

\nomenclature[1Tc]{$\underline{\widehat{t}}{}^*(v)$}{extended color-avoiding type I vector of $v$ in an ECBP tree: for each color $i \in [k]$ its $i$-th coordinate is 1 if $v$ is $i$-avoiding connected to infinity and 0 otherwise; see Definition~\ref{def:extended_typeI_in_ECBP}}

\nomenclature[1Td]{$T(v,w,\underline{\gamma})$}{color-avoiding type II event in an ECER graph: the intersection of the type II subevents or their complements according to $\underline{\gamma}$; see Definition~\ref{def:typeII_in_ECER}}

\nomenclature[1Te]{$T^{\setminus i}(v,w)$}{color-avoiding type II subevent in an ECER graph: the event of at least one vertex in the $i$-avoiding boundary belonging to a largest $i$-avoiding component of the unexplored subgraph when exploring with $k-2$ colors from $v$ and $w$; see Definition~\ref{def:typeII_in_ECER}}

\nomenclature[1Tf]{$\widehat{T}(\underline{v},\underline{\gamma},\underline{\beta},m)$}{$\empty$ \\ color-avoiding type III event in an ECER graph: the intersection of the type III subevents or their complements according to $\underline{\gamma}$; see Definition~\ref{def:typeIII_in_ECER}}

\nomenclature[1Tg]{$\widehat{T}^{\setminus i}(\underline{v},\underline{\beta},m)$}{$\empty$ \\ color-avoiding type III subevent in an ECER graph: the event of the existence of a vertex $v_{i,j}$ of $\underline{v}$ belonging to a largest $i$-avoiding component; see Definition~\ref{def:typeIII_in_ECER}}

\nomenclature[1U]{$\mathcal{U}_{\text{max}}(G)$}{vertices of the union of components with maximum cardinality in $G$; see~\eqref{eq:U_max}}

\nomenclature[1V]{$v \stackrel{G^{\setminus i}}{\longleftrightarrow} w$}{$v$ and $w$ are $i$-avoiding connected in $G$; see Definition~\ref{def:coloravoiding_connectivity}}

\nomenclature[1V]{$v \stackrel{G_{\infty}^{\setminus i}}{\longleftrightarrow} \infty$}{$v$ is $i$-avoiding connected to infinity in the ECBP tree $G_{\infty} \colonequals G_{\infty}(r)$; see Definition~\ref{def:friends}}

\nomenclature[2$\lambda$a]{$\underline{\lambda}$}{color density/intensity parameter vector; see Definitions~\ref{def:ECER} and \ref{def:ECBP}}

\nomenclature[2$\lambda$b]{$\lambda^{\setminus i}$}{sum of the coordinates of $\underline{\lambda}$ except the $i$-th one; see Definition~\ref{def:lambda_I}}

\nomenclature[2$\lambda$b]{$\lambda^{\text{uc}}$}{sum of the coordinates of $\underline{\lambda}$; see Definition~\ref{def:lambda_I}}

\nomenclature[2$\lambda$b]{$\lambda_I$}{sum of those coordinates of $\underline{\lambda}$ whose indices belong to $I$; see Definition~\ref{def:lambda_I}}

\nomenclature[2$\lambda$a]{$\underline{\lambda}(\varepsilon)$}{vector of length $k$ whose every coordinate is $\frac{1+\varepsilon}{k-1}$; see Definition \ref{def:barely_supercrit}}

\nomenclature[2$\Phi$]{$\Phi_h$}{joint generating function of the number of vertices reachable from the root by different color strings consisting of exactly $h$ colors in an ECBP tree; see Definition~\ref{def:joint_gen_function_of_Rs}}

\nomenclature[2$\rho$]{$\rho(v)$}{number of vertices reachable from $v$ with at most $k-2$ colors; see Definition~\ref{def:rho}}

\nomenclature[2$\tau$]{$\tau_{\infty}$}{color-avoiding horizon: the minimum of the $i$-avoiding horizons taken over each color $i \in [k]$; see Definition~\ref{def:stopping_time}}

\nomenclature[2$\tau$]{$\tau_{\infty}^{\setminus i}$}{$i$-avoiding horizon: the maximum distance from $r$ of the $i$-avoiding friends of the root $r$; see Definition~\ref{def:stopping_time}}

\nomenclature[2$\theta$]{$\theta^{\setminus i}$}{survival probability of the branching process with offspring distribution $\mathrm{POI}(\lambda^{\setminus i})$; see Definition~\ref{def:theta_minus_i}}

\printnomenclature

\bibliographystyle{abbrv}
\bibliography{coloravoiding}






\end{document}